\documentclass[12pt]{article}
\usepackage{amsmath,amssymb}
\newtheorem{theorem}{Theorem}[section]

\newtheorem{lemma}[theorem]{Lemma}
\newtheorem{definition}[theorem]{Definition}

\newenvironment{proof}{\smallskip\par{\sc Proof.}\enspace}%
 {{\unskip\nobreak\hfil\penalty50\hskip2em
          \hbox{}\nobreak\hfil{\rule[-1pt]{5pt}{10pt}}
          \parfillskip=0pt\finalhyphendemerits=0
          \par\medskip}} %This is newly changed %%
\makeatletter
\def\section{\@startsection {section}{1}{\z@}{3.25ex plus 1ex minus
 .2ex}{1.5ex plus .2ex}{\large\bf}}
\def\subsection{\@startsection{subsection}{2}{\z@}{3.25ex plus 1ex minus
 .2ex}{1.5ex plus .2ex}{\normalsize\bf}}
\@addtoreset{equation}{section} %%
\makeatother
\begin{document}

\begin{center}
\LARGE  Identification of the theory of orthogonal polynomials in
$d$--indeterminates with the theory of $3$--diagonal symmetric interacting
Fock spaces on $\mathbb C^d$
\end{center}

\vspace*{.3in}

\begin{center}
\sc

Luigi Accardi\\
 Centro Vito Volterra\\
 Universit\`{a} di Roma Tor Vergata\\
 Via di Tor Vergata, 00133 Roma, Italy\\
E-Mail: {\tt accardi@volterra.uniroma2.it}

\bigskip

Abdessatar Barhoumi\\
University of Carthage, Tunisia\\
Nabeul Preparatory Engineering Institute\\
Department of Mathematics\\
Campus Universitaire - Mrezgua - 8000 Nabeul\\
E-Mail: {\tt abdessatar.barhoumi@ipein.rnu.tn}

\bigskip

Ameur Dhahri\\
Department of Mathematics\\  Chungbuk National University\\
1 Chungdae-ro, Seowon-gu, Cheongju, Chungbuk 362-763, Korea\\
E-Mail: {\tt ameur@chungbuk.ac.kr}
\end{center}

\vspace*{.3in}

\begin{abstract}
The identification mentioned in the title allows a formulation
of the multidimensional Favard Lemma different from the ones
currently used in the literature and which parallels
the original $1$--dimensional formulation in the sense that
the positive Jacobi sequence is replaced by a sequence
of positive Hermitean (square) matrices and the real Jacobi
sequence by a sequence of positive definite kernels.
The above result opens the way to the program of a purely
algebraic classification of probability measures on
$\mathbb R^d$ with moments of any order and more generally of
states on the polynomial algebra on $\mathbb R^d$.\\
The quantum decomposition of classical real valued random
variables with all moments is one of the main ingredients
in the proof.
\end{abstract}



\tableofcontents

%%%%%%%%%%%%%%%%%%%%%%%%%%%%%%%%%%%%%%%%%%%%%%%%%%%%%%%%%%%%%%%%%%%%%%%%%%%
\section{Introduction}
%%%%%%%%%%%%%%%%%%%%%%%%%%%%%%%%%%%%%%%%%%%%%%%%%%%%%%%%%%%%%%%%%%%%%%%%%%%%

%\noindent
The theory of orthogonal polynomials is one of the classical
themes of calculus since almost two centuries and, in the
$1$--dimensional case, the large literature devoted to this
topic has been summarized in several well known monographs
(see for example \cite{st1943}, \cite{sz1975}, \cite{Chiha78},
\cite{[IsmAsk84]}).
In this case, even if at analytical level many deep
problems remain open, at the algebraic level the situation
is well understood and described by Favard Lemma which, to any
probability measure $\mu$ on the real line
with finite moments of any order, associates
two sequences, called the Jacobi sequences of $\mu$,
\begin{equation}\label{Jac-seq1}
\{(\omega_n)_{n \in \mathbb{N}}, \; (\alpha_n)_{n \in \mathbb{N}}\}\, ,\qquad
\omega_n\in\mathbb R_+ , \;  \alpha_n\in\mathbb{R},
\qquad n=0, 1, 2, \cdots
\end{equation}
subjected to the only constraint that, for any $n,\, k\in\mathbb N$,
\begin{equation}\label{Jac-seq1-constr}
\omega_n=0 \: \Longrightarrow  \: \omega_{n+k}=0
\end{equation}
Conversely, given two such sequences, it gives an inductive
way to uniquely reconstruct:
\begin{description}
\item[(i)] a state on the algebra $\mathcal P$ of polynomials
in one indeterminate (see subsection \ref{St-P}),
\item[(ii)] the orthogonal decomposition of $\mathcal P$
canonically associated to this state.
\end{description}
In this sense one can say that the pair of sequences
(\ref{Jac-seq1}), subjected to the only constraint
(\ref{Jac-seq1-constr}), constitutes a full set of algebraic
invariants for the equivalence classes of probability measures
on the real line with respect to the equivalence relation
$ \mu\sim\nu$  if and only if all moments of $\mu$ and $\nu$
are finite and coincide (moment equivalence of probability
measures on $\mathbb R$).\\
Compared to the $1$--dimensional case the literature
available in the multi-dimensional case is definitively scarse,
even if several publications (see e.g \cite{[DuXu01]},
\cite{[Koor90]}, \cite{[Macdon98]}, \cite{[MarcVanAs06]}) show
an increasing interest to the problem in the past years,
and for several years it has been mainly confined to
applied journals, where it emerges in connection with
different kinds of approximation problems.
The need for an insightful theory was soon perceived by
the mathematical community, for example in the 1953 monograph
\cite{[ErMagOberTric53]} (cited in \cite{[Xu97a]}), the authors
claim that '' $\dots $ there does not seem to be an extensive
general theory of orthogonal polynomials in several
variables $\dots $ ''.\\
Several progresses followed, both on the analytical front
concerning multi--dimensional extensions of
Carleman's criteria \cite{Nussb66}, \cite{[Xu93]}, and
on the algebraic front, with the introduction of the matrix
approach \cite{[KrShe67]} and the early formulations of
the multi--dimensional Favard lemma \cite{[Kowa82a]},
\cite{[Kowa82b]}, \cite{[Xu97a]}.\\
However, even with these progresses in view, one cannot
yet speak of a ''general theory of orthogonal polynomials
in several variables''. In fact the importance of Favard Lemma
consists in the fact that the pair
$(\alpha_n \ , \ \omega_n)$ condensates the {\bf minimal information}
gained from the knowledge of the $n$--th moment with respect to the
knowledge of all the $k$--th moments with $k\le n-1$.
Here the word {\bf minimal} is essential: \\
it is exactly this {\bf minimality} that was missing in all the
numerous approaches to the multi--dimensional Favard Lemma
until a couple of years ago. \\
%On the other hand, it is only pinpointing
%a minimal set of conditions that one can hope to prove a constructive
%reconstruction theorem.\\
The more recent multi--dimensional formulations of Favard Lemma are based on two sequences of
matrices, one of which rectangular, with quadratic constraints among the elements of these
sequences (see \cite{[AcNh02]}, \cite{[AcKuoSt04b]} and \cite{[Xu04]}, where (see Theorem 2.4)
the commutation relations in \cite{[AcNh02]}, \cite{[AcKuoSt04b]}
are expressed in terms of a fixed basis of orthogonal polynomials).
As mentioned in \cite{[Xu04]} such formulations
look far from the elegant simplicity of the $1$--dimensional Favard lemma. \\

\noindent Since the multi--dimensional analogues of positive (resp. real) numbers
are the positive definite (resp. Hermitean) matrices,
one would intuitively expect that a multi--dimensional
extension of the Favard lemma would replace the sequence
$(\omega_n)$ by a sequence of positive definite
matrices $(\tilde \Omega_n)$ and the $(\alpha_n)$--sequence by a sequence of
Hermitean matrices $(a^0_{j,n})$ for each coordinate function $X_j$ on $\mathbb R^d$.
The precise formulation of this naive conjecture is what we call {\it the
multi--dimensional Favard problem} (see section \ref{multi-dim-Fav-probl}). \\

\noindent The main result of the present paper is the proof that {\bf the above
mentioned conjecture is correct}. The {\bf new feature, specific of the
multi--dimensional case}, is that the two sequences $(\tilde \Omega_n)$ and
$(a^0_{j,n})_j$ must be constructed recursively, because the choice of
$\tilde \Omega_{n+1}$ and $a^0_{j,n+1}$ ($j=1,\dots, d$) is constrained by the
choices of the previous pairs.\\
The determination of these constraints, and their recursive formulation,
is based on several new results
and notions that are of independent interest. In particular:\\
  (1) The identification of the theory of orthogonal polynomials
with respect to a state on the algebra of polynomial
functions on $\mathbb R^d$ with the theory of {\bf symmetric}
interacting Fock spaces over $\mathbb C^d$ with a $3$--diagonal structure
(see section \ref{3-diag-dec-P} and the Appendix \ref{App-IFS} on interacting Fock spaces).\\
(2) The explicit form of the {above mentioned minimal set of constraints.\\

\noindent The reconstruction theorem (Theorem \ref{3diag-equiv-new-Favrd}) then shows
that the $d$--dimensional analogue of the principal Jacobi sequence is given by the sequence of
{\bf the real parts} $(\tilde\Omega_{R,n+1})$ of the positive--definite
kernels (block matrices) $(\tilde\Omega_{n+1})$, defining the scalar product
on the space of orthogonal
polynomials of order $n+1$ in terms of the scalar product on the space
of order $n$. In fact, once given this scalar product, $\tilde\Omega_{R,n+1}$ is an
arbitrary kernel, positive--definite with respect to it. The imaginary part
of $\tilde\Omega_{n+1}$, on the contrary, is uniquely fixed by the commutation relations
and by the $n$--th terms of the secondary Jacobi sequence: $a^0_{j,n+1}$ ($j=1,\dots, d$).
The $d$--dimensional
analogue of condition (\ref{Jac-seq1-constr}) consists in the statement that
these kernels map zero--norm vectors into zero--norm vectors. In particular,
if the $n$--th kernel is identically zero, then the $n$--th space of the gradation
consists only of zero vectors, hence the same will be true for all the $N$--th spaces
with $N\ge n$.\\
The $d$--dimensional analogue of the secondary Jacobi sequence is
given by $d$ sequences of symmetric matrices. These are not arbitrary,
but have to satisfy an inductive system of {\bf linear equations}.
The fact that this system always admits the zero--solution,
corresponding to symmetric states on the polynomial algebra, shows that, in
analogy with the one dimensional case, every class of states on the polynomial algebra in $d$
real variables, for the equivalence relation of having the same sequence of
scalar products on the gradation spaces, contains exactly one symmetric measure.\\

\noindent
The proof of all the above results heavily relies,
on the quantum probabilistic approach to the theory of
orthogonal polynomials first proposed, in the $1$--dimensional
case, in the paper \cite{[AcBo98]}, where the notion
of quantum decomposition of a classical random variable
was introduced and used to establish a canonical
identification between the theory of orthogonal polynomials
in $1$ indeterminate and the theory of $1$--mode interacting
Fock spaces (IFS). One can say that the quantum decomposition
of a classical random variable is a re--formulation of the
Jacobi recurrence relation.\\
The early extensions of this approach to the multi--dimensional
case \cite{[AcNh02]}, \cite{[AcKuoSt04b]} constructed
the quantum decomposition of the coordinate random varibles
in terms of creation, annihilation and preservation operators
on an IFS canonically associated to the orthogonal
decomposition of the polynomial algebra in $d$
indeterminates $\mathcal P_d$ with respect to a given state,
however, as mentioned above, for the Favard Lemma they
used rectangular matrices and quadratic relations. This made explicit
construction of the solutions a difficult problem.
An important step towards the solution of this problem was
done in the paper \cite{[AcKuoSta07]} where it was proved
that the reconstruction of the state on $\mathcal P_d$ can
be achieved using only the commutators between creation
and annihilation operators and the preservation operator.
These operators preserve the orthogonal gradation,
therefore each of them is determined by a sequence of
square matrices. Moreover the preservation operator, being
symmetric, is determined by a sequence of Hermitean matrices
while the commutators between creation and annihilation
operators are determined by two positive definite matrix
valued kernels, respectively $(a_ja^+_k)$ and $(a^+_ka_j)$
($\: j,k\in \{1,\dots ,d\}$).\\
Although this framework was much nearer to the one conjectured
in the Favard problem, yet important discrepancies remained,
in particular:
\begin{description}
  \item[(i)] While the sequence of Hermitean matrices is only
one for each coordinate random variable, as conjectured, the
commutators involved are defined by two sequences of positive
definite matrix valued kernels, namely the restrictions, to the gradation spaces,
of $(a_ja^+_k)$ and $(a^+_ka_j)$ ($\: j,k\in \{1,\dots ,d\}$).
  \item[(ii)] Contrarily to the $1$--dimensional case, the
correspondence between families of orthogonal polynomials in $d$ variables and
IFS over $\mathbb C^d$ is not one--to--one.
  \item[(iii)] The multi--dimensional analogue of the compatibility
	condition (\ref{Jac-seq1-constr}) remained obscure.
  \item[(iv)] 	The ''minimality condition'' mentioned above was not
	respected (this fact will be clear from the present paper).
\end{description}
These problems have been settled in the present paper: (i) and (iv) because
the sequence defined by the $(a^+_ka_j)$ is inductively determined, while
the sequence defined by the $(a_ja^+_k)$, i.e. $\tilde\Omega_{n+1}(j,k)$,
has an arbitrary real part; (ii) because the correct one--to--one correspondence
is with {\bf symmetric} IFS $3$--diagonal interacting
Fock spaces over $\mathbb C^d$; (iii) for the reasons explained above.\\

\noindent
In the present approach the emergence of the symmetric tensor
algebra as well as of nontrivial commutation relations
are both consequences of the commutativity of the
coordinate random variables.\\ In this sense {\bf a non commutative structure
is canonically deduced from a commutative one}.\\
From the point of view of physics, this clearly shows the probabilistic origins of
the Heisenberg commutation relations, which have been shown to characterize Gaussian measures
(see \cite{[AcKuoSta07]}). For classes different from the Gaussian one, we obtain a
powerful generalization of the whole mathematical structure of quantum theory
that, in its infinite dimensional version, corresponds to an extension of quantum field theory.
Thus the traditional theory of orthogonal polynomials merges with the program of
nonlinear first and second quantization and provides new tools for it.\\

{\bf Acknowledgements}
The authors are grateful to Hyun Jae Yoo for pointing out an error in
the previous version of the present paper. They are also grateful to
Abdallah Dhahri for many interesting discussions and remarks leading to
several improvements in the exposition.\\
LA acknowledges support by the RSF grant 14-11-00687,
Steklov Mathematical Institute.

%%%%%%%%%%%%%%%%%%%%%%%%%%%%%%%%%%%%%%%%%%%%%%%%%%%%%%%%%%%%%%%%%%%%%%%%%%%%%%%%%%%%%%%%%%%%%%%
\section{ The polynomial algebra }
%%%%%%%%%%%%%%%%%%%%%%%%%%%%%%%%%%%%%%%%%%%%%%%%%%%%%%%%%%%%%%%%%%%%%%%%%%%%%%%%%%%%%%%%%%%%%%%

%%%%%%%%%%%%%%%%%%%%%%%%%%%%%%%%%%%
\subsection{Notations }
%%%%%%%%%%%%%%%%%%%%%%%%%%%%%%%%%%%

Throughout the present paper, for any $m\in\mathbb N$,
$\mathbb{C}^m$ (resp. $\mathbb{R}^m$) will denote the $m$--dimensional
complex (resp. real) vector space referred to the canonical basis
denoted in both cases ($e_j$) $(j\in \{1,\dots,m\})$ and the term
{\it coordinates} will be referred to this basis. Unless otherwise specified,
algebras and vector spaces will be complex.
Let $D:=\{1,\dots,d\}$ ($d\in\mathbb N$) be a finite set and denote
\begin{equation}\label{notat-P}
\mathcal P :=\mathcal P_D := {\mathbb C}[(X_j)_{j\in D}]
\end{equation}
the complex polynomial algebra in the commuting indeterminates
$(X_j)_{j\in D}$ with the\\ $*$--algebra structure uniquely
determined by the prescription that the $X_j$ are self-adjoint.
The principle of identity of polynomial states that a
polynomial is identically zero if and only if all its
coefficients are zero. This is equivalent to say that
the generators $X_j$ ($j\in D$) are algebraically
independent. These generators will also be called {\it coordinates}.\\
By definition $\mathcal P$ has an identity,
denoted $1_{\mathcal P}$, and
\begin{equation}\label{Xj0=1P}
X_j^{0} =1_{\mathcal{P}}   \: , \qquad \forall \, j\in D,
\end{equation}
where $1_{\mathcal{P}}$ denotes the identity of $\mathcal{P}$.\\
For any vector space $V$ we denote $\mathcal{L}(V)$
the algebra of linear maps of $V$ into itself.\\
For $F=\{1,\dots,m\}\subseteq D$ and
$v=(v_1,\dots,v_m)\in \mathbb{R}^m$ we will use the notation:
$$
X_v:=\sum_{j\in F}v_jX_j .
$$
The coordinates $X_j\;\;(j\in D)$ define a linear map
$$
X: v=\sum_{j\in D}v_je_j\in\mathbb{R}^d \longmapsto X_v:=\sum_{j\in D}v_jX_j\in \mathcal{L}(\mathcal{P}).
$$
The real linear span $\mathcal P_{\mathbb R}$ of the generators
$X_j$ induces a natural real structure on $\mathcal P$ given by
\begin{equation}\label{real-struct-P}
\mathcal P = \mathcal P_{\mathbb R} \dot+ i\mathcal P_{\mathbb R}
\end{equation}
where, here and in the following, $\dot+$ in (\ref{real-struct-P})
means direct sum in the real vector space sense.
All the properties considered in this section
continue to hold if one restricts one's attention to the
real algebra $\mathcal{P}_\mathbb{R}$.\\

With the convention (\ref{Xj0=1P}) a \textit{monomial of degree}
$n\in \mathbb N$ is by definition any product of the form
\begin{equation}\label{df-mon}
M := \prod_{j\in F} X_j^{n_j}
\end{equation}
where $F\subseteq D$ is a finite subset, and for any $j \in F$, $\, n_j \in \mathbb{N}$
$$
\sum_{j\in F}n_j = n.
$$
The monomial (\ref{df-mon}) is said to
be \textit{localized in the subset} $F\subseteq D$.\\
The algebra generated by such monomials will be denoted
$$
\mathcal P_F\subseteq \mathcal P:=\mathcal{P}_D .
$$
Notice that, with this definition of localization, if $F\subseteq G\subseteq D$
then any monomial localized in $F$ is also localized in $G$, i.e.
$$
\mathcal P_F\subseteq \mathcal P_G\subseteq \mathcal P .
$$
For all $n \in \mathbb N$ and for any subset $F\subseteq D$,
we use the following notations:
\begin{equation}\label{df-mon-n]}
\mathcal{M}_{F,n]} := \textrm{the set of monomials of degree
less or equal than} \: n \: \textrm{localized in} \: F \qquad\qquad\qquad\qquad
\end{equation}
\begin{equation}\label{df-mon-n}
\mathcal{M}_{F,n} := \textrm{the set of monomials of degree}\:
n \: \textrm{localized in} \: F\qquad\qquad\qquad\qquad\qquad\qquad\qquad\qquad
\end{equation}
\begin{equation}\label{df-PFn]}
\mathcal P_{F,n]} := \textrm{the vector sub--space of}\:
\mathcal{P} \: \textrm{generated by the set} \: \mathcal{M}_{F,n]}
\quad\quad\qquad\qquad\qquad\qquad\qquad\qquad\qquad
\end{equation}
\begin{equation}\label{df-PFn}
\mathcal P_{F,n}^{0} := \textrm{the vector sub--space of} \: \mathcal{P} \:
\textrm{generated by the set} \: \mathcal{M}_{F,n} \}\qquad\qquad\qquad\qquad\qquad\qquad\qquad\qquad\qquad\qquad\qquad\qquad\qquad
\end{equation}
We use the apex $0$ in $\mathcal P_{F,n}^{0}$ to distinguish
the monomial gradation (see (\ref{1.5a}) below), which is
purely algebraic, from the orthogonal gradations, which
will be introduced later on and depend on the choice of
a state on $\mathcal{P}$.
The only monomial of degree $n=0$ is by definition
$$
M_0 := 1_{\mathcal P} .
$$
Therefore
\begin{equation}\label{P0]]}
\mathcal P^0_{F,0} =\mathcal P_{F,0]} = \mathbb C\cdot 1_{\mathcal P}.
\end{equation}
More generally, if $|F| =m$ then for any
$n\in\mathbb N$ there are exactly
\begin{eqnarray}\label{card-D-mon-deg-n}
d_n := \left(
  \begin{array}{c}
    n + m - 1 \\
     m - 1
  \end{array}
\right)
\end{eqnarray}
monomials of degree $n$ localized in $F$ and, by the
principle of identity of polynomials they are
linearly independent. Therefore one has
\begin{equation}\label{dim-Pn}
\mathcal P^0_{F,n} \equiv \mathbb C^{d_n}
\end{equation}
where the isomorphism is meant in the sense of vector spaces.\\
For future use it is useful to think of  $\mathcal P$ as an
algebra of operators acting on itself by left multiplication.
In the following, when no confusion is possible, we will use
the same symbol for an element $Q\in \mathcal P$ and for its
multiplicative action on $\mathcal P$. Sometimes, to emphasize
the fact that $Q$ is considered as an element of the vector
space $\mathcal P$, we will use the notation
$$
Q\cdot 1_{\mathcal P}=:Q\cdot \Phi_0 .
$$
The sequence $(\mathcal P_{F,n]})_{n \in \mathbb{N}}$ is
an increasing filtration of complex finite dimensional
$*$--vector sub--spaces of $\mathcal P$, i.e:
\begin{equation}\label{inclusions}
\mathcal P_{F,0]} \subset \mathcal P_{F,1]} \subset\mathcal P_{F,2]}
\subset \cdots \subset \mathcal P_{F,n]}\subset \cdots
\subset \mathcal P_F\subset \mathcal P .
\end{equation}
Moreover
\begin{equation}\label{union-Pn=P}
\bigcup_{n\in\mathbb N}\mathcal P_{F,n]} = \mathcal P_F
\end{equation}
and, for any $m,\, n \in\mathbb N$
one has
$$
\mathcal P_{F,m]}\cdot \mathcal P_{F,n]} = \mathcal P_{F,m+n]} .
$$
The sequence $(\mathcal P_{F,n}^{0})_{n \in \mathbb{N}}$
defines a vector space gradation of $\mathcal P_{F}$
\begin{equation}\label{1.5a}
\mathcal P_{F}
=\sum^{\centerdot}_{k\in \mathbb{N}} \mathcal P_{F,k}^{0}
\end{equation}
called the monomial decomposition of $\mathcal P$.
In (\ref{1.5a}) the symbol $\displaystyle{\sum^{\centerdot}}$
denotes direct sum in the sense of vector spaces, i.e.
elements of $\mathcal P $ are finite linear sums of elements
in some of the $\mathcal P^{0}_{F,n}$ and
\begin{equation}\label{0-int-P0n}
m\neq n \: \Longrightarrow \:
\mathcal P_{F,m}^{0}\cap\mathcal P_{F,n}^{0}= \{0\}.
\end{equation}
The gradation (\ref{1.5a}) is compatible with the
filtration $(\mathcal P_{F,n]})$ in the sense that,
for any $n \in \mathbb{N}$,
\begin{equation}\label{vect-sp-dir-sum}
\mathcal P_{F,n]} = \sum^{\centerdot}_{k \in \{0,1,\cdots,n \}}
\mathcal P_{F,k}^{0}.
\end{equation}
In particular
$$
\mathcal P_{F}= \mathcal P_{F,n]} \, \dot + \, \Big(\sum^{\centerdot}_{k>n} \mathcal P_{F,k}^{0}\Big)
\: , \qquad\forall \, n\in\mathbb N .
$$
%\eject
\begin{lemma}\label{vect-sub-sp-em}{\rm
%\label{struc-Pn]}, \label{dim-Pn-const}
(i) For any vector sub--space $W\subset \mathcal P_{F}$, the set
\begin{equation}\label{XvW}
XW := \{X_vW \ : \ v\in{\mathbb C}^{F}\}
\end{equation}
is a vector sub--space of ${\mathcal P}_{F}$, where $\mathbb {C}^{F}:=\{ v \in \mathbb{C}^{m}:\: v_{j}=0 \: \textrm{if} \: j \notin F \}$.\\
(ii) For each $n\in\mathbb N$, one has
\begin{equation}\label{XPFn0=PFn+10]}
X\mathcal P_{F,n}^{0} = \mathcal P_{F,n+1}^{0}
\end{equation}
\begin{equation}\label{PFn+1]=XjPFn]=P0]}
\mathcal P_{F,n+1]} =X\mathcal P_{F,n]} \, \dot + \, \mathcal P_{F,0]}
= \mathcal P_{F,n]} \, \dot + \, \mathcal P_{F,n+1}^{0}.
\end{equation}
(iii) For $n\in\mathbb N$, let $\mathcal P_{n+1}$ be a vector
sub--space of $\mathcal P_{n+1]}$ such that
\begin{eqnarray}\label{1.9a}
\mathcal P_{n]} \, \dot + \, \mathcal P_{n+1}
=\mathcal P_{n+1]}.
\end{eqnarray}
Then as a vector space $\mathcal P_{n+1}$ is isomorphic
to $\mathcal P^{0}_{n+1}$.
}\end{lemma}

\begin{proof}
(i) The set (\ref{XvW}) coincides with the set
$$
\Big\{\sum_{j\in F}X_j\xi^{(j)}_w \ : \ \xi^{(j)}_w\in W, \:
\forall \, j\in F\Big\}
$$
and this is clearly a vector space.\\
(ii) Since $\mathcal{M}_{F,n}$ is a linear basis of $\mathcal P_{F,n}^{0}$, $\bigcup_{j\in F}X_j
\mathcal{M}_{F,n}\subset \mathcal P_{F,n+1}^{0}$
is a system of generators of the sub--space
$X\mathcal P_{F,n}^{0}$. Hence
$X\mathcal P_{F,n}^{0}\subset \mathcal P_{F,n+1}^{0}$.
The converse inclusion is clear because
$\bigcup_{j\in F}X_j\mathcal{M}_{F,n}$ is also a system of
generators of $\mathcal P_{F,n+1}^{0}$.
This proves (\ref{XPFn0=PFn+10]}). (\ref{PFn+1]=XjPFn]=P0]})
follows from (\ref{vect-sp-dir-sum}) and (\ref{XPFn0=PFn+10]}).\\
(iii) Since the sum in (\ref{1.9a}) is direct and the spaces are
finite dimensional, one has
$$
\textrm{dim}(\mathcal P^{0}_{n+1})
=\textrm{dim}(\mathcal P_{n+1]})-\textrm{dim}(\mathcal P_{n]})
=\textrm{dim}\left(\mathcal P_{n+1}\right).
$$
\end{proof}

%%%%%%%%%%%%%%%%%%%%%%%%%%%%%%%%%%%%%%%%%%%%%%%%%%%%%%%%%%%%%%%%%%%%%%%%%%%%%%%%%%%%%%%%%%%%%%%%%
\subsection{$\mathcal P$ and the symmetric tensor algebra over ${\mathbb C}^d$}
%%%%%%%%%%%%%%%%%%%%%%%%%%%%%%%%%%%%%%%%%%%%%%%%%%%%%%%%%%%%%%%%%%%%%%%%%%%%%%%%%%%%%%%%%%%%%%%%%

In the present paper the number $d\in\mathbb N^*:= \mathbb N\setminus \{0\}$
will be fixed and
$$
D\equiv \{1,\cdots ,d\}
$$
in the following the index $D$ will be omitted and we will use the notations:
\begin{eqnarray*}
\mathcal{P}_D=\mathcal{P}\; , \qquad
\mathcal P^0_{n}:=\mathcal P^0_{D,n}\; , \qquad
\mathcal P_{n]}:=\mathcal P_{D,n]}\; , \qquad n\in\mathbb{N}
\end{eqnarray*}
with the convention $$
\mathcal{P}_{-1}^0=\mathcal{P}_{-1]}=\{0\}.
$$
The natural real structure on $\mathbb C$ given by
$\mathbb C=\mathbb R  \dot + i \mathbb R$ induces a real structure on
$\mathbb C^d=\mathbb R^d  \dot + i \mathbb R^d$
the associated (componentwise) involution given by complex conjugation:
\begin{equation}\label{inv-Cd}
(u+iv)^*:=u-iv \; , \qquad
u+iv\in \mathbb C^d \ := \ \mathbb R^d \ \dot+ \ i\mathbb R^d .
\end{equation}
In the following we fix the choice $V:=\mathbb C^d$ and we denote
$(e_j)_{j\in D}$ the canonical basis of $\mathbb{C}^d$ which is a
{\bf real basis}, i.e. a basis of $\mathbb{R}^d\subset\mathbb{C}^d$.
$\otimes$ will denote algebraic tensor product
and $\widehat{\otimes}$ its symmetrization. The tensor algebra over
${\mathbb C}^d$ is the vector space
$$
\hbox{Tens}({\mathbb C}^d):=
\sum^{\centerdot}_{n \in \mathbb{N}}({\mathbb C}^d)^{\otimes n}
$$
with multiplication given by
$$
(u_{n}\otimes\cdots\otimes u_{1})\otimes
(v_{m}\otimes\cdots\otimes v_{1})
:=u_{n}\otimes\cdots\otimes u_{1}\otimes
v_{m}\otimes\cdots\otimes v_{1}
$$
for any $m,\: n\in{\mathbb N}$ and all $u_{j},\: v_{j}\in\mathbb C^d.$
The extension to $\mathbb C^d$ of the natural real structure
on $\mathbb C$ given by $\mathbb C=\mathbb R  + i \mathbb R$
and the associated involution, induces a $*$--algebra structure
on $\mathcal{T}({\mathbb C}^d)$ whose involution is
characterized by the property that
\begin{equation}\label{df-sym-invol}
(v_{n}\otimes\cdots\otimes v_{1})^*
:=v^*_{n}\otimes\cdots\otimes v^*_{1}
 \: , \qquad  \forall \, n\in{\mathbb N} \, , \;
\forall \, v\in\mathbb C^d .
\end{equation}
For $n\in\mathbb N^*$, the $*$--sub--space of
$({\mathbb C}^d)^{\otimes n}$ generated by the elements of the form
\begin{equation}\label{df-sym-elem}
v^{\otimes n} := v\otimes\cdots\otimes v \quad
(n\hbox{--times})  \: , \qquad
\forall \, n\in{\mathbb N} \, , \; \forall \, v\in\mathbb C^d .
\end{equation}
is called {\bf the symmetric tensor product of $n$--copies of} $\mathbb C^d$
and denoted by $({\mathbb C}^d)^{\widehat{\otimes} n}$.
$({\mathbb C}^d)^{\widehat \otimes n}$ coincides with the fixed point sub--space
of the linear action, on $({\mathbb C}^d)^{\otimes n}$, of the $n$--th
order permutation group $\mathcal S_{n}$ given by
$$
\widehat{\sigma}\left(v_n\otimes v_{n-1}\otimes\dots\otimes v_1\right)
:= v_{\sigma_{n}}\otimes v_{\sigma_{n-1}}\otimes\dots\otimes v_{\sigma_1}
\, , \quad
v_n\otimes v_{n-1}\otimes\dots\otimes v_1\in(\mathbb{C}^d)^{\otimes n}
\ , \ \sigma\in\mathcal S_{n} .
$$
By definition:
$$
({\mathbb C}^d)^{\widehat \otimes 0} := \mathbb C
$$
$$
\hbox{Tens}_{sym}({\mathbb C}^d)
:= \sum^{\centerdot}_{n \in \mathbb{N}}({\mathbb C}^d)^{\widehat{\otimes} n} .
$$
$\hbox{Tens}_{sym}({\mathbb C}^d)$ is the graded abelian $*$--sub--algebra of
$\hbox{Tens}({\mathbb C}^d)$ generated by the elements of the form
(\ref{df-sym-elem}) and is called {\bf the symmetric tensor algebra
over ${\mathbb C}^d$}. \\
The following Lemma reformulates some known results in a language and
with the notations that will be used later.

\begin{lemma}\label{lem1.4}{\rm
Let $(e_j)_{j\in D}$ be the canonical linear basis of ${\mathbb C}^d$.
The map
\begin{equation}\label{df-iso-S0n1}
e_{j} \longmapsto X_{j} \ , \ j\in D
\: , \qquad  1_{\mathcal{T}_{sym}({\mathbb C}^d)} \mapsto 1_{\mathcal{P}}
\end{equation}
extends uniquely to a is a gradation preserving isomorphism
of commutative $*$--algebras:
\begin{equation}\label{df-iso-S0}
S^{0}:=\sum^{\centerdot}_{n \in \mathbb{N}} S^{0}_n \ : \
\hbox{Tens}_{sym}({\mathbb C}^d):=
\sum^{\centerdot}_{n \in \mathbb{N}}({\mathbb C}^d)^{\widehat{\otimes}n}
\ \rightarrow \
\sum^{\centerdot}_{n \in \mathbb{N}}\mathcal P^{0}_n\equiv\mathcal P .
\end{equation}
In particular for all $n\in{\mathbb N}^{*}$ and for all maps
$j \ : \ \{1,\dots ,n\} \ \to \ \{1,\dots ,d\}$:
\begin{equation}\label{df-iso-S0n2}
e_{j_n}\widehat{\otimes}\cdots\widehat{\otimes} e_{j_1}
\longmapsto X_{j_n}\cdots X_{j_1}
%\cdot \Phi_0
\end{equation}
and, in the notations of section (\ref{Ex-FFS}) below:
\begin{equation}\label{Ex-FFS0}
e_{j}\widehat{\otimes} (\ \cdot \ )  = \ell^*_{e_{j}}= X_{j} .
\end{equation}
}\end{lemma}

\begin{proof} The thesis follows from the fact that the $e_j$'s
(resp. $X_j$'s) ($j\in D$) are algebraically independent (i.e. the
terms appearing in (\ref{df-iso-S0n2}) and the correpsonding identities
are linearly independent) self--adjoint
generators of the commutative $*$--algebra $\mathcal{T}_{sym}({\mathbb C}^d)$
(resp. $\mathcal{P}$) and that the correspondence (\ref{df-iso-S0n1})
is $1$--to--$1$.
\end{proof}

{\bf Remark}. In analogy with the identification of $X_{j}$ with its
action as multiplication operator on $\mathcal P$, $e_j$ can be identified
with the symmetric tensor multiplication by $e_j$. If confusion may arise, we
use the notation
$$
\hat M_{e_j}(e_{j_n}\widehat{\otimes}\cdots\widehat{\otimes} e_{j_1})
:= e_{j}\widehat{\otimes}e_{j_n}\widehat{\otimes}\cdots\widehat{\otimes} e_{j_1}.
$$
With this notation and the corresponding one for the $X_{j}$'s, one has
\begin{equation}\label{intertw-Mej-Xj}
S^{0}\hat M_{e_j}(S^{0})^{-1} = M_{X_{j}} \: , \qquad j\in D .
\end{equation}

\begin{lemma}\label{lemm}{\rm
Let $(\mathcal P_n)_{n\in{\mathbb N}}$ be any family of
sub--spaces of $\mathcal P$ such that
$$
\mathcal P_{k+1]}=\mathcal P_{k]}\, \dot + \, \mathcal P_{k+1}
\, ,\qquad\forall \, k\in{\mathbb N},
$$
$$
\mathcal{P}_{0}=\mathcal P_{0]}=\mathcal P^{0}_0={\mathbb C}1_{\mathcal P}.
$$
Then, for all $n\in{\mathbb N}$, there exists a vector
space isomorphism
\begin{equation}\label{df-iso-Sn}
S_n \ : \ ({\mathbb C}^d)^{\widehat{\otimes}n}
\rightarrow\mathcal P_n
\end{equation}
and the map
\begin{equation}\label{df-iso-S}
S:=\sum^{\centerdot}_{n \in \mathbb{N}}S_n \ : \
\hbox{Tens}_{sym}({\mathbb C}^d):=
\sum^{\centerdot}_{n \in \mathbb{N}}({\mathbb C}^d)^{\widehat{\otimes}n}
\ \rightarrow \
\sum^{\centerdot}_{n \in \mathbb{N}}\mathcal P_n\equiv\mathcal P
\end{equation}
is a gradation preserving vector space isomorphism.
}\end{lemma}

\begin{proof}
From Lemma \ref{vect-sub-sp-em} we know that,
for all $n\in{\mathbb N}$, $\mathcal P_n$ has the same
dimension as $\mathcal P^0_n$ (given by (\ref{card-D-mon-deg-n})).
Hence there exists a vector space isomorphism
$$
T_n \ : \ \mathcal P^0_n \ \rightarrow \
\mathcal P_n\: , \qquad\forall \, n\in{\mathbb N}.
$$
Defining $S_n := T_n\circ S^0_n $ where $S^0_n $ is given
by (\ref{df-iso-S0}), (\ref{df-iso-Sn}) follows.
This implies that the map defined by (\ref{df-iso-S}) is a gradation
preserving vector space isomorphism.
\end{proof}

{\bf Remark}.  In general the map defined by (\ref{df-iso-S}) is not an
isomorphism of commutative $*$--algebras in particular the analogue for $S$
of (\ref{intertw-Mej-Xj}) does not hold. To obtain this additional property
will require a different choice for the vector space isomorphisms
$T_n$ (see section \ref{multi-dim-Favard-Lm} below).

%%%%%%%%%%%%%%%%%%%%%%%%%%%%%%%%%%%%%%%%%%%%%%%%%%%%%%%%%%%%%%%%%%%%%%%%%%%%%%%%
\subsection{ States on $\mathcal P$ }\label{St-P}
%%%%%%%%%%%%%%%%%%%%%%%%%%%%%%%%%%%%%%%%%%%%%%%%%%%%%%%%%%%%%%%%%%%%%%%%%%%%%%%%

For the terminology on pre--Hilbert spaces we refer to
Appendix \ref{App-Orth-proj-PHS}.\\
Denote $\mathcal S(\mathcal P)$ the set of linear functionals on $\mathcal P$
that are real on real polynomials, $1$ on the identity and positive on
polynomials of the form $P=|Q|^2$ with $P,Q\in\mathcal P$. Such linear functionals
will be called {\bf states}.\\
Any probability measure on $\mathbb R^d$ with all moments induces a state
on $\mathcal P$. The converse finds an obstruction in the existence, for $d>1$,
of positive polynomials $P$ not expressible in the form $P=|Q|^2$. We refer to
the paper \cite{[AlpJorgKim14]} for references on this old and deep problem,
that is related to the polynomial version of Hilbert's $17$--th problem.\\
Even in case of existence and even in the case $d=1$, there may be many
probability measures on $\mathbb R^d$ defining the same state on $\mathcal P$
(non uniqueness in the moment problem). On the contrary, the state on
$\mathcal P$ is uniquely defined. For this reason in the following
we restrict our attention to states on $\mathcal P$.\\
As shown in the following of the present paper,
all the constructions related to orthogonal polynomials are valid in the
more general framework of states on $\mathcal P_{V}$.
Therefore in the following we will discuss this more general framework.\\

\noindent Any state $\varphi\in \mathcal S(\mathcal P)$
defines a pre--scalar product
$\langle \ \cdot \ , \ \cdot \ \rangle_\varphi $
on $\mathcal P$ given by
\begin{equation}\label{varphi-sc-pr}
(a,b)\in \mathcal P\times \mathcal P\mapsto
\langle a, b\rangle_\varphi :=\varphi (a^*b)\in\mathbb C
\end{equation}
satisfying the conditions
$$
\langle 1_{\mathcal{P}}, 1_{\mathcal{P}}\rangle_{\varphi} =1
$$
\begin{equation}\label{<ab,c>=<b,a*c>}
\langle ab, c\rangle_{\varphi} =\langle b, a^*c\rangle_{\varphi}
\, ,  \qquad \forall \, \: a,\, b,\, c\in \mathcal P ,
\end{equation}
where $a^*$ denotes the adjoint of $a$ in $\mathcal P$.
In particular the operators
$X_j$ are symmetric as pre--Hilbert space operators.
Thus the pair
$$
(\mathcal P \, ,  \, \langle \ \cdot \ , \ \cdot \ \rangle_\varphi)
$$
is a commutative pre--Hilbert algebra.

\begin{lemma}\label{st-ind-sc-pr}{\rm
For a pre--scalar product
$\langle \ \cdot \ , \ \cdot \ \rangle$
on $\mathcal{P}$ the following statements are  {\bf equivalent}:\\
(i)
There exists a state $\varphi$ on $\mathcal{P}$ such that:
\begin{equation}\label{sc-pr-ind-by-st}
\varphi(f^*g)  = \langle f, g\rangle
\: ,\qquad f, g\in \mathcal{P}.
\end{equation}
(ii)
The pre--scalar product
$\langle \ \cdot \ , \ \cdot \ \rangle$ satisfies
\begin{equation}\label{<1,1>=1}
\langle 1_{\mathcal{P}}, 1_{\mathcal{P}}\rangle =1
\end{equation}
and, for each $j\in D$, multiplication by the coordinate $X_j$
is a symmetric linear operator on $\mathcal{P}$ with respect to
$\langle \ \cdot \ , \ \cdot \ \rangle$, i.e.:
\begin{equation}\label{Xj-sym}
\langle X_{j}f, g\rangle =\langle f, X_{j}g\rangle .
\end{equation}
}\end{lemma}

\begin{proof}
\textbf{(ii) $\Rightarrow$ (i)}.
Every scalar product on $\mathcal P$ is induced by the linear functional:
\begin{equation}\label{<1,Q1>}
\varphi (Q):=\langle 1_{\mathcal{P}}, Q\cdot 1_{\mathcal{P}}\rangle
\: ,\qquad  Q\in\mathcal P .
\end{equation}
Condition (\ref{Xj-sym}) implies that $\varphi$ is a $*$--functional on
$\mathcal P$, i.e. for any $Q\in\mathcal P$,
$\overline{\varphi (Q)}=\varphi (Q^*) $, where $*$ denotes the involution
on $\mathcal P$. Hence condition (\ref{<1,Q1>}) implies that $\varphi$
is positive. Then, because of (\ref{<1,1>=1}), $\varphi$ is a state
on $\mathcal P$.\\
\textbf{(i) $\Rightarrow$ (ii)}. This is clear and has already been
discussed before the statement of the Theorem.
\end{proof}

%%%%%%%%%%%%%%%%%%%%%%%%%%%%%%%%%%%%%%%%%%%%%%%%%%%%%%%%%%%%%%%%%%%%%%%%%%%%%%%%%%%%%%%%%%%%%%
\section{The multi--dimensional Favard problem}\label{multi-dim-Fav-probl}
%%%%%%%%%%%%%%%%%%%%%%%%%%%%%%%%%%%%%%%%%%%%%%%%%%%%%%%%%%%%%%%%%%%%%%%%%%%%%%%%%%%%%%%%%%%%%%

%%%%%%%%%%%%%%%%%%%%%%%%%%%%%%%%%%%%%
\subsection{Fundamental lemmas}
%%%%%%%%%%%%%%%%%%%%%%%%%%%%%%%%%%%%

\begin{definition}\label{df-monic-ss}{\rm
For $n\in\mathbb{N}$ we say that a sub--space
$\mathcal{P}_n\subset \mathcal{P}_{n]}$
is {\bf monic of degree $n$} if it
has a {\bf real} linear basis $B_n$ with the
property that for each $b\in B_n$, the highest order
term of $b$ is a non-zero multiple of a single monomial of degree $n$
and each monomial of degree $n$ appears exactly once in the basis $B_n$.\\
Such a basis is called {\bf a perturbation of the monomial basis
of order $n$} in the coordinates $(X_j)_{j\in D}$ or simply {\bf a monic basis of order $n$}
if no confusion is possible.
}\end{definition}

{\bf Remark}. For a monic sub--space one has:
\begin{equation}\label{df-monic-ss2}
\mathcal{P}_{n]}=\mathcal{P}_{n-1]} \  \dot+ \ \mathcal{P}_n
\end{equation}
(with the convention $\mathcal{P}_{-1]}=\{0\}$).
Notice that monic bases arise naturally
in the Gram--Schmidt orthogonalization process of monomials.\\

Let $\varphi$ be a state on $\mathcal P$ and denote
$$
\langle \ \cdot \ , \ \cdot \ \rangle :=
\langle \ \cdot \ , \ \cdot \ \rangle_\varphi
$$
the corresponding pre--scalar product.
When no ambiguity is possible, the elements $\xi$
of $\mathcal{P}$ (resp. $\mathcal P_{n]}$, $\mathcal P^0_{n}$)
satisfying
$$
\langle \xi, \xi \,\rangle =0
$$
will be simply called {\it zero norm vectors} without
explicitly mentioning the pre--scalar product (or the
associated state $\varphi$).
By the Schwarz inequality the set of zero norm vectors in
$\mathcal{P}$ (resp. $\mathcal P_{n]}$, $\mathcal P^0_{n}$),
denoted $\mathcal{N}_{\varphi}$ (resp.
$\mathcal{N}_{\varphi,n]}$, $\mathcal{N}_{\varphi,n}$) is a
$*$-sub--space satisfying
\begin{equation}\label{pres-0-norm-el}
\mathcal{P}\mathcal{N}_{\varphi,n}\subseteq
\mathcal{P}\mathcal{N}_{\varphi,n]}\subseteq
\mathcal{P}\mathcal{N}_{\varphi}\subseteq \mathcal{N}_{\varphi}.
\end{equation}
In particular $\mathcal{N}_{\varphi}$ is a $*$--ideal
of $\mathcal{P}$.
The monomial decomposition (\ref{1.5a}) is
compatible with the filtration $(\mathcal P_{F,n]})$
in the sense of (\ref{vect-sp-dir-sum}), therefore
$$
\mathcal P= \mathcal P_{n]} \, \dot + \, \Big(\sum^{\centerdot}_{k>n} \mathcal P_{k}^{0}\Big)
\: , \qquad\forall \, n\in\mathbb N .
$$
For reasons that will be clear in the reconstruction theorem
of section \ref{rec-theo} we want to keep the discussion at
a pure vector space, rather than Hilbert space level.
In particular we don't want to quotient out the zero norm
vectors. Therefore, rather than the usual Gram--Schmidt
orthonormalization procedure, we use its pre--Hilbert space
variant, described in Appendix \ref{App-Orth-proj-PHS}.

\begin{lemma}\label{from-st-to-grad}{\rm
Let $\varphi$ be a state on $\mathcal P$ and denote
$\langle \ \cdot \ , \ \cdot \ \rangle
= \langle \ \cdot \ , \ \cdot \ \rangle_\varphi$
the associated pre--scalar product.
Then there exists a gradation
\begin{equation}\label{orth-pol-grad1}
\mathcal P
=\bigoplus_{n\in\mathbb N}\mathcal P_{n,\varphi}
% \ , \ \langle  \ \cdot \ , \ \cdot \ \rangle_{n,\varphi})
\end{equation}
called {\bf the $\varphi$--orthogonal gradation} of $\mathcal P$,
with the following properties:
\begin{description}
  \item[(i)] (\ref{orth-pol-grad1}) is orthogonal for the
% unique
pre--scalar product $\langle  \cdot \ , \ \cdot \ \rangle$;
% on $\mathcal P$ defined by the conditions:
% $$
% \langle \ \cdot \ , \ \cdot \
% \rangle|_{\mathcal P_{n,\varphi}}
% =\langle \ \cdot \ , \ \cdot \ \rangle_{n,\varphi}
% \qquad ; \qquad\forall \, n\in\mathbb N
% $$
% $$
% \mathcal P_{m,\varphi} \ \perp \ \mathcal P_{n,\varphi}
% \qquad ; \qquad\forall \, m\neq n
% $$
  \item[(ii)] (\ref{orth-pol-grad1}) is compatible with the
filtration $(\mathcal P_{n]})$ in the sense that
\begin{equation}\label{filt-rec}
\mathcal P_{k]}=\bigoplus_{h\in\{0,1,\cdots,k\}}
\mathcal P_{h,\varphi}
\: , \qquad\forall \, k\in\mathbb N ;
\end{equation}
  \item[(iii)] for each $n\in\mathbb N$ the space
$\mathcal P_{n,\varphi}$ is monic.
%  \item[(iv)] the pre--scalar product
%  $\langle \ \cdot \ , \ \cdot \ \rangle$,
%  defined in item (i) above, is induced by a state on
%  $\mathcal P$, i.e. satisfies the conditions of
%  Lemma \ref{st-ind-sc-pr}.
\end{description}
Conversely, let be given:
\begin{description}
  \item[(j)] a vector space direct sum decomposition of $\mathcal P$
\begin{equation}\label{rec-2-1}
\mathcal P=\sum^{\centerdot}_{n \in \mathbb{N}}\mathcal P_n
\end{equation}
such that $\mathcal P_0=\mathbb C\cdot 1_{\mathcal P}$, and
for each $n\in\mathbb N,\ \mathcal P_n$ is monic
of degree $n$,
  \item[(jj)] for all $n\in\mathbb N$ a pre--scalar product
$\langle \ \cdot \ , \ \cdot \ \rangle_n$ on
$\mathcal P_n$ with the property that
$1_{\mathcal P}$ has norm $1$ and the unique pre--scalar product
$\langle \ \cdot \ ,  \cdot \ \rangle$ on $\mathcal P$
defined by the conditions:
\begin{equation}\label{restr-scal-prod}
\langle \ \cdot \ , \ \cdot \ \rangle|_{\mathcal P_n}
=\langle \ \cdot \ , \ \cdot \ \rangle_n\: , \qquad\forall \, n\in\mathbb N ,
\end{equation}
\begin{equation}\label{cond-orth-grad}
\mathcal P_n\perp\mathcal P_m\: , \qquad\forall \, m\neq n,
\end{equation}
is such that the operators of multiplication by the coordinates
$X_j$ ($j\in D$) are $\langle \ \cdot \ , \ \cdot \ \rangle$--symmetric
linear operators on $\mathcal P$.

Then there exists a state $\varphi$ on $\mathcal P$ such
that the decomposition (\ref{rec-2-1}) is the
orthogonal polynomial decomposition of $\mathcal P$ with
respect to $\varphi$.
\end{description}
}\end{lemma}

\begin{proof}
Let be given a state $\varphi$ on $\mathcal P$.
In the above notations, for each $k\in\mathbb N$ define
inductively the sub--space $\mathcal P_{k,\varphi}$ and the two sequences of
$\langle \ \cdot \ , \ \cdot \ \rangle$--orthogonal
projectors
$$
P_{k],\varphi} \ : \ \mathcal P\rightarrow\mathcal P_{k]}
 \, ,  \qquad P_{k,\varphi} \ : \ \mathcal P\rightarrow\mathcal P_{k,\varphi}\: , \qquad \forall\, k\in\mathbb N
$$
compatible with the real structures of the corresponding spaces
(i.e. $P_{k],\varphi}(\mathcal P_{\mathbb R})\subseteq \mathcal P_{\mathbb R,k]}$,
$P_{k,\varphi}(\mathcal P_{\mathbb R})\subseteq \mathcal P_{\mathbb R,k}$,
and in this case we speak of {\bf real projectors}) as follows. \\
For $k=0$, define $\: \mathcal P_{0,\varphi}:= \mathcal P_{0]}$ and
$$
P_{0,\varphi}:= P_{0],\varphi} \ : \ Q\in\mathcal P \ \longmapsto \
\varphi(Q)1_{\mathcal P}
=\langle 1_{\mathcal P},Q\cdot 1_{\mathcal P}\rangle 1_{\mathcal P}
\in \mathcal P_{0]}=:\mathcal P_{0,\varphi}
\: , \qquad \forall \, Q\in\mathcal P .
$$
Clearly $P_{0,\varphi}$ is a real projector.
Having defined the real projectors
$$
\{P_{0,\varphi}, P_{1,\varphi}, \cdots , P_{n,\varphi} \}
\: , \qquad
\{P_{0],\varphi}, P_{1],\varphi}, \cdots , P_{n],\varphi} \}
$$
so that for each $k\in\{0,1,\dots ,n\}$ the space
$\mathcal P_{k,\varphi}$ is monic and (\ref{filt-rec})
is satisfied, in the notation (\ref{df-mon-n}), define
\begin{equation}\label{df-Pn+1fi}
\mathcal P_{n+1,\varphi}:=\hbox{lin-span}
\{M_{n+1}-P_{n],\varphi}(M_{n+1}) \ : \ M_{n+1}\in\mathcal M_{n+1}\}.
\end{equation}
Then the space $\mathcal P_{n+1,\varphi}$ is monic of order
$n+1$ since the generating set on the right hand side of (\ref{df-Pn+1fi})
is clearly a basis, it is real because such is the projector $P_{n],\varphi}$
and it is a perturbation of the monomial basis of order $n$ because
the $P_{n],\varphi}(M_{n+1})$ are polynomials of degree $n$.
In particular the sum
$$
\mathcal P_{n+1,\varphi} + \mathcal P_{n]} = \mathcal P_{n+1]}
$$
is direct, hence such is also the decomposition
$$
\mathcal P=
\mathcal P_{n+1,\varphi} \dot+ \mathcal P_{n]} \dot+ \mathcal P_{(n+1}
$$
($\mathcal P_{(n+1}$ denotes the space of polynomials of degree $>n+1$).

Define ${\mathcal K}_{0,1}$ (resp. ${\mathcal K}_{0,0}$) the sub--space of
$\mathcal P_{n+1,\varphi}$ generated by the
non--$\langle \ \cdot \ , \ \cdot \ \rangle$--zero norm (resp.
$\langle \ \cdot \ , \ \cdot \ \rangle$--zero norm) vectors in the set
on the right hand side of (\ref{df-Pn+1fi}). Since the elements of this set are
linearly independent, ${\mathcal K}_{0,1}\cap {\mathcal K}_{0,0}=\{0\}$ and
by construction
${\mathcal K}_{0,1}\dot + {\mathcal K}_{0,0}=\mathcal P_{n+1,\varphi}$.
By the induction assumption on the $\langle\ \cdot \ , \ \cdot \ \rangle_{n}$,
the real structure on $\mathcal P$ induces a real structure on
$\mathcal P_{n+1,\varphi}$.

Applying Corollary \ref{proj-PHS} of Appendix
\ref{App-Orth-proj-PHS} with ${\mathcal K}=\mathcal P$,
$\mathcal{K}_{0}=\mathcal P_{n+1,\varphi}$,
$\mathcal{K}_{1}:=\mathcal P_{n]} \dot+ \mathcal P_{(n+1}$
and ${\mathcal K}_{0,1}$
any vector space supplement of the
$\langle \ \cdot \ , \ \cdot \ \rangle$--zero norm
sub--space ${\mathcal K}_{0,0}$ of $\mathcal P_{n+1,\varphi}$, we define the orthogonal
projection
$$
P_{n+1,\varphi} \ : \ \mathcal P \ \to \ \mathcal P_{n+1,\varphi}
$$
which by construction is onto $\mathcal P_{n+1,\varphi}$ hence
orthogonal to $P_{n],\varphi}$. Therefore the operator
$$
P_{n+1],\varphi} \ := \  P_{n],\varphi} + P_{n+1,\varphi}
$$
is the orthogonal projection onto $\mathcal P_{n+1]}$.
Finally, given $\varphi$, the conditions of Lemma
\ref{st-ind-sc-pr} are satisfied by the associated
pre--scalar product on $\mathcal P$.
This completes the induction construction.

To prove the converse, notice that
the fact that each $\mathcal P_n$ is monic implies that the
decomposition (\ref{rec-2-1}) satisfies condition
(\ref{filt-rec}).
In fact this is true for $\mathcal P_0$ by construction and,
supposing it true for $k\in\mathbb N$, it follows for $k+1$
from the monicity condition.
Thus, by induction, property (\ref{filt-rec}) holds for
each $n\in\mathbb N$.
Because of Lemma \ref{st-ind-sc-pr}, condition (jj) implies
that the pre--scalar product
$\langle \ \cdot \ , \ \cdot \ \rangle$ is induced by
a state $\varphi$ in the sense of the identity
(\ref{varphi-sc-pr}).
This implies that the decomposition (\ref{rec-2-1}) is the
orthogonal polynomial decomposition of $\mathcal P$ with
respect to the state $\varphi$.
\end{proof}

The following Lemma shows
that the isomorphism, defined abstractly in Lemma \ref{lemm}
can be explicitly constructed if the gradation on $\mathcal P$
is the one constructed in Lemma \ref{from-st-to-grad}.

\begin{lemma}\label{new-id-symm-tens}{\rm
Let be given a vector space direct sum decomposition
of $\mathcal P$ of the form (\ref{rec-2-1}) satisfying
conditions (j) and (jj) of Lemma \ref{from-st-to-grad}.
Let $B_{n}$ be a perturbation of the monomial basis
in $\mathcal P_{n}$ (see Definition \ref{df-monic-ss})
and for each monomial $M_n\in\mathcal{M}_{D,n}$
denote $p_n(M_n)$ the corresponding element of $B_{n}$.
Then the map
\begin{equation}\label{new-symm-tens-vn-a+vn}
\pi_n:\;e_{j_n}\widehat{\otimes} e_{j_{n-1}}\widehat{\otimes}\cdots
\widehat{\otimes} e_{j_{1}}
\in (\mathbb C^d)^{\widehat{\otimes} n} \
\longmapsto \ p_n\left(X_{j_n}X_{j_{n-1}} \cdots X_{j_{1}}
\right)\cdot 1_{\mathcal P}\in \mathcal P_{n}
\end{equation}
where $n\in\mathbb{N}^*\;(\pi_0=id_{\mathbb{C}})$ and
$\widehat{\otimes}$ denotes symmetric tensor product,
extends to a vector space isomorphism.
}\end{lemma}

\begin{proof}
A basis $B_{n}$ as in the statement of the Lemma exists because
$\mathcal P_{n}$ is monic.
Denoting $j:\{1,\dots ,n\}\to \{1,\dots ,d\}$ a generic
function, the map
\begin{equation}\label{new-3.32a}
e_{j_n}\otimes e_{j_{n-1}}\otimes\cdots
\otimes e_{j_{1}}\longmapsto
p_n\left(X_{j_n}X_{j_{n-1}} \cdots X_{j_{1}}
\right)\cdot 1_{\mathcal P}\in \mathcal P_{n}
\end{equation}
is well defined on a linear basis of $(\mathbb C^d)^{\otimes n}$
because $X_{j_n}X_{j_{n-1}} \cdots X_{j_{1}}$ is a monomial of
degree $n$. Since both sides in (\ref{new-3.32a}) are
multi--linear, by the universal property of the tensor product
it extends to a linear map, denoted $\widehat{\pi}_{n}$, of
$(\mathbb C^d)^{\otimes n}$ into $\mathcal P_{n}$.
This map is surjective because when $j$ runs over all maps
$\{1,\dots ,n\}\to \{1,\dots ,d\}$,
$p_n\left(X_{j_n}X_{j_{n-1}} \cdots X_{j_{1}}
\right)\cdot 1_{\mathcal P}$ runs over a linear basis
of $\mathcal P_{n}$.
Since the right hand side of (\ref{new-3.32a}) is invariant
under permutations of the indices $j_n,j_{n-1},\cdots, j_{1}$,
$\widehat{\pi}_n$ induces a linear map of the vector space of
equivalence classes of elements of $(\mathbb{C}^d)^{\otimes n}$
with respect to the equivalence relation induced by the linear
action of the permutation group. Since this quotient space is
canonically isomorphic to the symmetric tensor product
$(\mathbb{C}^d)^{\widehat{\otimes} n}$, this induced map defines
a linear extension of the map (\ref{new-symm-tens-vn-a+vn}). \\
This extension is an isomorphism because we have already
proved that surjectivity and injectivity follow from the
fact that the equivalence class under permutations of
any $n$--tuple $(j_n,j_{n-1},\cdots, j_{1})$ defines a
unique element of the basis
$\left\{p_n(M_{n})\cdot 1_{\mathcal P} \, ; \;
M_{n}\in\mathcal P_{n}\right\}$ of $\mathcal P_{n}$.
\end{proof}

\noindent{\bf Remark}. The construction of Lemma \ref{from-st-to-grad}
depends on the choice of the vector space supplement of
the zero norm sub--space of $\mathcal P_{n,\varphi}$. However any
vector in another supplement will differ by a zero norm vector
from a vector in the previous choice. Therefore, at Hilbert
space level, the two choices will coincide.

%%%%%%%%%%%%%%%%%%%%%%%%%%%%%%%%%%%%%%%%%%%%%%%%%%%%%%%%%%%%%%%%%%%%%%%%%%%%%%%%%%%%%%%%%%%%%%%%%%%%%%%%
\subsection{ Statement of the multi--dimensional Favard problem}
%%%%%%%%%%%%%%%%%%%%%%%%%%%%%%%%%%%%%%%%%%%%%%%%%%%%%%%%%%%%%%%%%%%%%%%%%%%%%%%%%%%%%%%%%%%%%%%%%%%%%%%%

From Lemma \ref{from-st-to-grad} we know that the
orthogonal polynomial decomposition of $\mathcal P$ with
respect to a state $\varphi$ induces a decomposition of
$\mathcal P$ of the form (\ref{rec-2-1}).
Given such a decomposition, for every $n\in\mathbb N$,
we can use the vector space isomorphisms $\pi_n$ defined
in Lemma \ref{new-id-symm-tens} to transfer the pre--Hilbert
structure of $\mathcal P_{n}$ on the symmetric tensor
product space $(\mathbb C^d)^{\widehat{\otimes} n}$.
Imposing the orthogonality of the $\mathcal P_{n}$'s one obtains
a gradation preserving unitary isomorphism between $\mathcal P$,
with the orthogonal polynomial gradation induced by the state
$\varphi$, and a symmetric interacting Fock space structure
over $\mathbb C^d$ (see Appendix \ref{App-S-IFS}).
The converse of this statement is at basis of the
multi--dimensional Favard problem:\\
Given a symmetric interacting Fock space structure over
$\mathbb C^d$ (see section \ref{App-S-IFS} below)
$$
\bigoplus_{n\in\mathbb N}\left((\mathbb C^d)^{\widehat{\otimes} n}
 \ , \ \langle \ \cdot \ ,\Omega_n \ \cdot \
\rangle_{\widehat{\otimes}n}\right):
$$
\begin{description}
  \item[(i)] does there exist a state $\varphi$ on $\mathcal P$
whose associated symmetric IFS is the given one?
  \item[(ii)] it is possible to parameterize all solutions
of problem (i) and to characterize them constructively?
\end{description}
The second part of the present paper is devoted to the
solution of this problem.
Before that, in the following section, we establish
some notations and necessary conditions.

%%%%%%%%%%%%%%%%%%%%%%%%%%%%%%%%%%%%%%%%%%%%%%%%%%%%%%%%%%%%%%%%%%%%%%%%%%%%%%%%%%%%%%%%%%%%%%%
\section{The symmetric Jacobi relations}\label{sec-symm-jac-rel}
%%%%%%%%%%%%%%%%%%%%%%%%%%%%%%%%%%%%%%%%%%%%%%%%%%%%%%%%%%%%%%%%%%%%%%%%%%%%%%%%%%%%%%%%%%%%%%%

%%%%%%%%%%%%%%%%%%%%%%%%%%%%%%%%%%%%%%%%%%%%%%%%%%%%%%%%%%%%%%%%%%%%%%%%%%%%%%%%%
\subsection{The orthogonal gradation and the three--diagonal recurrence relations}
%%%%%%%%%%%%%%%%%%%%%%%%%%%%%%%%%%%%%%%%%%%%%%%%%%%%%%%%%%%%%%%%%%%%%%%%%%%%%%%%%%

In this section we fix a state $\varphi$ on $\mathcal P$
and we follow the notations of Lemma \ref{from-st-to-grad}
with the exception that we omit the index $\varphi$.
Thus we write $\langle \ \cdot \ , \ \cdot \ \rangle$
for the pre--scalar product
$ \langle \ \cdot \ , \ \cdot \ \rangle_\varphi$,
$P_{k]}:\mathcal P\rightarrow\mathcal P_{k]}$
($k\in\mathbb N$) for the
$\langle \ \cdot \ , \ \cdot \ \rangle$--orthogonal
projector in the pre-Hilbert space sense, constructed in the proof of
Lemma \ref{from-st-to-grad}, $\mathcal P_{k+1}$ for the
space defined by (\ref{df-Pn+1fi}) and
\begin{equation}\label{df-PFn2}
P_{n}=P_{n]}-P_{n-1]}
\end{equation}
the corresponding projector. We know that
\begin{equation}\label{PFn]-real}
P_{n]}(\mathcal P_{{\mathbb R}})\subseteq
\mathcal P_{{\mathbb R}}\cap\mathcal P_{n]}
=\mathcal P_{{\mathbb R},n]}
\: , \qquad \forall \, \, n\in\mathbb N ,
\end{equation}
and that the sequence $(\mathcal{P}_{n]})_{n \in \mathbb{N}}$
is an increasing filtration with union  $ \mathcal P$
(see (\ref{inclusions}) and (\ref{union-Pn=P})).
It follows that the sequence of projections (\ref{df-PFn2})
is a partition of the identity in
$(\mathcal P,\langle \ \cdot \ , \ \cdot \ \rangle)$, i.e.
\begin{equation}\label{mut-orth-PFn}
P_{n}P_{m}=\delta_{mn}P_{m}\, ,  \qquad P_{n}=P_{n}^*
\, ,\qquad \forall \, \, m,\: n\in\mathbb N ,
\end{equation}
\begin{equation}\label{sum-PFn=1}
\sum_{n\in\mathbb N} P_{n}=\lim_n P_{n]} =1_\mathcal P .
\end{equation}

\begin{lemma}\label{lm-stabil-Jac-proj}{\rm
Suppose that, for some $m\in\mathbb N^*$, the range of $P_{m}$ is
contained in the sub--space of zero--norm vectors. Then the same is
true for any $n\geq m$, i.e.
\begin{equation}\label{stabil-Jac-proj}
P_{n}(\mathcal P)\subseteq \mathcal N \: ,  \qquad \forall \, n\geq m .
\end{equation}
}\end{lemma}

\begin{proof}
Under our assumptions for any monomial $M_{m}$ of degree $m$, one has\\
$M_{m}-P_{m-1]}(M_{m})\in \mathcal N$. This implies that
$M_{m}\in \mathcal P_{m-1]}+\mathcal N$. Since multiplication by coordinates
leaves $\mathcal N$ invariant, this implies that for each $j\in D$,
$X_jM_{m}\in \mathcal P_{m]}+\mathcal N$. Therefore for any monomial
$M_{m+1}$ of degree $m+1$, $M_{m+1}\in \mathcal P_{m]}+\mathcal N$. In
particular $M_{m+1}-P_{m]}(M_{m+1})\in \mathcal N$, i.e.
$\mathcal P_{m+1}\subseteq \mathcal N$ and this is equivalent to the thesis.
\end{proof}

\begin{theorem}\label{th-Symm-Jac-rel}{\rm
With the notation
\begin{eqnarray*}\label{PF-1]:=0}
P_{-1]}:=0
\end{eqnarray*}
for any $j\in D$ and any $n\in{\mathbb N}$, one has
\begin{equation}\label{Symm-Jac-rel}
X_jP_{n}=P_{n+1}X_jP_{n}+P_{n}X_jP_{n}+P_{n-1}X_jP_{n}.
\end{equation}
}\end{theorem}

\begin{proof}
Because of (\ref{sum-PFn=1}), for any $j \in D$,
$$
X_j=1_{\mathcal{P}}\cdot X_{j}\cdot 1_{\mathcal{P}}=\sum_{m,n \in \mathbb{N}}P_{m}X_jP_{n}.
$$
Therefore for each $n\in\mathbb N$,
$$
X_j P_{n}=\sum_{m\in \mathbb{N}} P_{m}X_j P_{n}.
$$
Since
$$
X_j{\mathcal P}_{n}\subseteq{\mathcal P}_{n+1]}
$$
it follows that
$$
X_j P_{n}=P_{n+1]}X_jP_{n}.
$$
Since $(P_{m]})$ is increasing, if $m>n+1$
then
$$
P_{m]}P_{n+1]}=P_{m-1]}P_{n+1]}=P_{n+1]},
$$
hence
$$
P_{m}X_jP_{n}=
P_{m}P_{n+1]}X_jP_{n}=
(P_{m]}-P_{m-1]})P_{n+1]}X_jP_{n}=0.
$$
If $m<n-1$, then the first part of the proof implies that
$$
P_{m}XP_{n}=(P_{n}XP_{m})^*=0.
$$
Summing up: $P_{m}X_jP_{n}$ can be non-zero
only if $m\in\{n-1,n,n+1\}$ and this proves
(\ref{Symm-Jac-rel}).
\end{proof}

\begin{definition}{\rm
The identity (\ref{Symm-Jac-rel}) is called the
symmetric Jacobi relation.
}\end{definition}

%%%%%%%%%%%%%%%%%%%%%%%%%%%%%%%%%%%%%%%%%%%%%%%%%%%%%%%%%%%%%%%%%%%%%%%%%%%%%%%%%%%%%
\subsection{The CAP operators and the quantum decomposition of the coordinates}
%%%%%%%%%%%%%%%%%%%%%%%%%%%%%%%%%%%%%%%%%%%%%%%%%%%%%%%%%%%%%%%%%%%%%%%%%%%%%%%%%%%%%

For each $n\in\mathbb N$ and $j\in D$, define the operators
\begin{equation}\label{1df-a+j|n}
a^{+}_{j|n}:=P_{n+1}X_jP_{n}\Big|_{{\mathcal P}_{n}}
\ : \ \mathcal P_{n}\longrightarrow {\mathcal P}_{n+1}
\end{equation}
\begin{equation}\label{1df-a0j|n}
a^{0}_{j|n}:=P_{n}X_jP_{n}\Big|_{{\mathcal P}_{n}}
\ : \ {\mathcal P}_{n}\longrightarrow {\mathcal P}_{n}
\end{equation}
\begin{equation}\label{1df-a-j|n}
a^{-}_{j|n}:=P_{n-1}X_jP_{n}\Big|_{{\mathcal P}_{n}}
\ : \ {\mathcal P}_{n}\longrightarrow \mathcal P_{n-1}
\end{equation}

\noindent{\bf Remark}. Notice that for each $n\in\mathbb N$, $j\in D$
and $\varepsilon\in\{+,0,-\}$, the operators $a^{\varepsilon}_{j|n}$
map polynomials with real coefficients into polynomials with the same
property. In fact both multiplication by coordinates and the projections
$P_{n}$ satisfy this condition (see (\ref{PFn]-real})).\\
Notice that, $D$ being a finite set, the spaces
${\mathcal P}_{n}$ are finite dimensional.
Moreover, in the present algebraic context, the sum
\begin{equation}\label{weak-sum}
{\mathcal P}=\bigoplus_{n\in{\mathbb N}}{\mathcal P}_{n}
\end{equation}
is orthogonal and meant {\bf in the following weak sense}, i.e. for each
element $Q\in {\mathcal P}$ there is a finite set
$I\subset{\mathbb N}$ such that
\begin{equation}\label{finite-decomposition}
    Q=\sum_{n\in I}p_n
\: ,\qquad p_n\in {\mathcal P}_{n}.
\end{equation}

\begin{theorem}\label{th-Q--dec}{\rm
For any $j \in D$, the following operators are well defined on ${\mathcal P}$:
\begin{eqnarray*}\label{1df-a+j}
a^{+}_j&:=&\sum_{n\in{\mathbb N}}a^{+}_{j|n}\\
%\end{eqnarray*}
%\begin{eqnarray*}\label{1df-a0j}
a^{0}_j&:=&\sum_{n\in{\mathbb N}}a^{0}_{j|n}\\
%\end{eqnarray*}
%\begin{eqnarray*}\label{1df-a-j}
a^{-}_j&:=&\sum_{n\in{\mathbb N}}a^{-}_{j|n}
\end{eqnarray*}
and one has
\begin{equation}\label{q-dec-Xj}
X_j=a^{+}_j+a^{0}_j+a^{-}_j
\end{equation}
in the sense that both sides of (\ref{q-dec-Xj})
are well defined on $\mathcal P$ and the equality holds.
Moreover the decomposition on the right hand side of (\ref{q-dec-Xj})
is unique in the sense that, if $b^{+}_j$, $b^{0}_j$, $b^{-}_j$ are linear
operators on $\mathcal P$ satisfying (\ref{1df-a+j|n}), (\ref{1df-a0j|n}),
(\ref{1df-a-j|n}), then they coincide with $a^{+}_j$,
$a^{0}_j$, $a^{-}_j$ respectively. Finally the operators $a^{+}_j$,
$a^{0}_j$, $a^{-}_j$ map polynomials with real coefficients into
polynomials with the same property.
}\end{theorem}

\begin{proof}
For all $ j\in D$, using the symmetric Jacobi relation (\ref{Symm-Jac-rel}),
one has
\begin{eqnarray*}
  (a^{+}_j+a^{0}_j+a^{-}_j) &=&   \sum_{n\in \mathbb{N}}(a^{+}_{j|n}+a^{0}_{j|n}+a^{-}_{j|n})\\
   &=& \sum_{n\in \mathbb{N}} (P_{n+1}X_jP_{n}+P_{n}X_jP_{n}+P_{n-1}X_jP_{n}) \\
  &=& \sum_{n\in \mathbb{N}}X_jP_{n}=X_j .
\end{eqnarray*}
Finally uniqueness follows from the identity
$b^{+}_j+b^{0}_j+b^{-}_j=a^{+}_j+a^{0}_j+a^{-}_j$ and the fact that, for
$\epsilon \neq \epsilon '$ ($\epsilon, \epsilon '\in\{-1,0,+1\}$) the ranges
of the operators $a^{\epsilon}_j-b^{\epsilon}_j$ and
$a^{\epsilon'}_j-b^{\epsilon'}_j$ are orthogonal. therefore the operators
$a^{\epsilon}_j$ and $b^{\epsilon}_j$ coincide on all $n$--particle spaces,
hence on $\mathcal P$.
The last statement follows from the Remark after the definition of
the operators $a^{\varepsilon}_{j|n}$.
\end{proof}

\begin{definition}\label{df-quant-dec}{\rm
The identity (\ref{q-dec-Xj}) is called the {\bf quantum decomposition of}
$X_j$ with respect to the state $\varphi$.
}\end{definition}

\noindent{\bf Remark}.
The {\bf quantum decomposition of} $X_j$ with respect to $\varphi$
allows to extend the map $X:\mathbb{R}^d\to\mathbb{R}^d$ to a map
$X:\mathbb{C}^d\to\mathbb{C}^d$ as follows:
If $v=(v_1,\dots,v_d)\in\mathbb{C}^d$, we denote
\begin{equation}\label{notation}
a^{\varepsilon}_{v}:=\sum_{j\in D}v_ja^{\varepsilon}_{j}
\: ,\qquad\varepsilon\in\{+,0\}\: ,\qquad
a^{-}_{v}:=\sum_{j\in D}\bar v_ja^{-}_{j}.
\end{equation}
Then one defines, in the notation (\ref{inv-Cd})
\begin{equation}\label{df-gen-flds1}
X_v:= a^{+}_{v}+a^{0}_{v}+a^{-}_{v^*} \: ,\qquad v\in\mathbb{C}^d .
\end{equation}
With this definition one has
$$
(X_v)^*=X_{v^*}.
$$

\subsection{Properties of the quantum decomposition}

Notice that, by construction,
for any $j\in D$ and $ n\in{\mathbb N}$, the maps
$$
a^{+}_{j|n} :=P_{n+1}X_jP_{n}
$$
satisfy
\begin{equation}\label{rang-a+j|n-in-PF,n+1a}
a^{+}_{j|n} ({\mathcal P}_{\mathbb R,n}) \subseteq {\mathcal P}_{\mathbb R,n+1}
\end{equation}
hence in particular
\begin{equation}\label{rang-a+j|n-in-PF,n+1}
a^{+}_{j|n} ({\mathcal P}_{n}) \subseteq {\mathcal P}_{n+1}
\end{equation}
and recall that, by construction, the non-zero elements
of ${\mathcal P}_{n+1}$ are polynomials of degree $n+1$.

\begin{lemma}\label{(a+j*=a-j}{\rm
For any $j \in D$ and $n\in\mathbb N$, one has
\begin{eqnarray*}\label{a0j*=a0-j}
(a^{+}_{j|n})^{*} &=& a^{-}_{j|n+1}\: , \qquad
(a^{+}_{j})^{*} = a^{-}_{j}\label{a+j*=a-j};\\
(a^{0}_{j|n})^{*} &=& a^{0}_{j|n} \: , \qquad
(a^0_{j})^*=a^0_{j}.
\end{eqnarray*}
}\end{lemma}

\begin{proof}
For an arbitrary $j \in D$ and $n\in\mathbb N$ we have
$$
(a^{+}_{j|n})^{*}=(P_{n+1}X_jP_{n})^*=
P_{n}X_jP_{n+1}=a^{-}_{j|n+1}.
$$
Recall that, with the notation (\ref{1df-a-j|n}),
$$
a^{-}_{j|n}=P_{n-1}X_j P_{n}
 \ : \ {\mathcal P}_{n}\longrightarrow {\mathcal P}_{n-1}.
$$
Thus
$$
(a^+_{j})^*=\Big(\sum_{n \in \mathbb{N}} a^{+}_{j|n}\Big)^*
=\sum_{n \in \mathbb{N}}(a^{+}_{j|n})^*
=\sum_{n \in \mathbb{N}}a^{-}_{j|n+1}
$$
and, with the change of variables
$n+1=:m\in{\mathbb N}^*:={\mathbb N}\setminus\{0\}$,
this becomes
$$
(a^+_{j})^*=\sum_{m\in{\mathbb N}^*}a^{-}_{j|m}
=\sum_{n\in{\mathbb N}}a^{-}_{j|n}=a^-_{j}
$$
because
$$
a^{-}_{j|0}=0.
$$
Summing up
\begin{eqnarray*}
&&(a^+_{j})^*=a^-_{j}\: , \qquad (a^-_{j})^*=((a^+_{j})^*)^*=a^+_{j};\\
&&(a^{0}_{j|n})^*=(P_{n}X_jP_{n})^* =a^{0}_{j|n};\\
&&(a^0_{j})^* =\Big(\sum_{n\in{\mathbb N}} a^{0}_{j|n}\Big)^*
=\sum_{n\in{\mathbb N}} (a^{0}_{j|n})^*
=\sum_{n\in{\mathbb N}} a^{0}_{j|n}=a^0_{j}.
\end{eqnarray*}
\end{proof}

\begin{lemma}\label{pres-0-norm-vec}{\rm
For any $j \in D$, the operators
$$
X_j \: , \qquad a^{+}_{j}\: , \qquad a^{-}_{j}\: , \qquad a^{0}_{j}
$$
preserve the space $\mathcal{N}_{\varphi}$ of zero norm vectors.
}\end{lemma}

\begin{proof}
It is sufficient to show that,
for each $n\in\mathbb N$  if $\xi\in\mathcal P_{n}$
is a zero norm vector, then the same is true for the vectors
$$
X_j\xi \qquad ; \qquad a^{+}_{j|n}\xi \: , \qquad
a^{0}_{j|n}\xi\: , \qquad a^{-}_{j|n}\xi \: , \qquad j\in D .
$$
That $X_j\xi$ is a zero norm vector follows from
$$
\left| \langle X_j\xi,X_j\xi \rangle \right|
=\left| \langle X_j^2\xi,\xi \rangle \right|
\leq \left| \langle X_j^2\xi,X_j^2\xi \rangle \right|^{1/2}
\:  \left| \langle \xi,\xi \rangle \right|^{1/2} =0.
$$
From this and the quantum decomposition (\ref{q-dec-Xj})
it follows that the vector
$$
X_{j}P_{n}\xi = a^{+}_{j|n}\xi + a^{0}_{j|n}\xi + a^{-}_{j|n}\xi
$$
has zero norm.
Since the right hand side is a sum of three mutually
orthogonal vectors, it follows that each of them is
a zero norm vector.
\end{proof}

\begin{lemma}\label{Pk-pol-dk}{\rm
In the notations of Definition \ref{df-3d-dec}, for $n\in\mathbb N$, let be given:
\begin{description}
\item[(i)]  two monic vector sub--spaces in the coordinates $(X_j)$
$\mathcal P_{n-1}\subset \mathcal{P}_{n-1]},\;
\mathcal P_{n}\subset \mathcal P_{n]}$ respectively of degree $n-1$ and $n-1$,
\item[(ii)] two arbitrary linear maps
\begin{equation}\label{a0v|k-a}
v\in \mathbb C^d \longmapsto
A^0_{v|n}\in\mathcal L_a(\mathcal P_{n}, \mathcal P_{n})
\end{equation}
\begin{equation}\label{a-v|k-a}
v\in \mathbb C^d \longmapsto
A^{-}_{v|n}\in\mathcal L_a(\mathcal P_{n}, \mathcal P_{n-1}).
\end{equation}
\end{description}
Then, defining for any $v\in\mathbb C^d$ the map
\begin{equation}\label{df-a+v|k1-a}
A^+_{v|n} :=
X_v\Big|_{\mathcal P_{n}} \ - \ A^0_{v|n} \ - \ A^{-}_{v|n},
\end{equation}
the vector space
\begin{equation}\label{constr-df-tPk+1}
\tilde{\mathcal P}_{n+1} := \{A^+_{v|n}\mathcal P_{n}  \, ; \;  \ v\in \mathbb C^d \}
\end{equation}
has the form
\begin{equation}\label{constr-df-Pk+1b}
\tilde{\mathcal P}_{n+1}
= \mathcal P_{n+1} \dot + \left(\tilde {\mathcal P}_{n+1}\cap \mathcal P_{n]}\right)
\end{equation}
where $\mathcal P_{n+1}$ is a monic vector sub--space of degree $n+1$ and
$\dot +$ denotes direct sum of linear spaces.
}\end{lemma}

\begin{proof}
Since $\mathcal{P}_n$ is monic of degree $n$ in the coordinates $(X_j)$, it has a linear basis
$B_n:=(\xi_{n,M})_{M\in \mathcal M_{e,n}}$ which is a perturbation
of the monomial basis of degree $n$. From the definition (\ref{df-a+v|k1-a}) of $A^+_{v|n}$
we know that, for each $j\in D$ and $M\in \mathcal M_{e,n}$, one has
\begin{equation}\label{3.14a}
A^{+}_{j|n}\xi_{n,M}=X_j\xi_{n,M}-A^0_{j|n}\xi_{n,M}-A^-_{j|n}\xi_{n,M}.
\end{equation}
The assumptions on $A^0_{j|n}$ and $A^-_{j|n}$ imply that
$A^0_{j|n}\xi_{n,M}+A^-_{j|n}\xi_{n,M}$ is a polynomial of degree $\le n$.
Therefore, when $\xi_{n,M}$ varies in $B_n$ and
$X_j$ varies among all coordinate functions,
$A^+_{j|n}\xi_{n,M}$ defines a set of monic polynomials whose leading
terms contain the set of all monomials of degree $n+1$ (with possible
repetitions). Therefore from this set one can extract a perturbation of
the monomial basis of order $n+1$. Denote by $B_{n+1}$ this basis and
$\mathcal{P}_{n+1}$ its linear span. By construction $\mathcal{P}_{n+1}$
is a monic vector sub--space of $\mathcal P_{n+1]}$.
The definition of perturbation of a monomial basis implies that
\begin{equation}\label{dir-sum1}
\mathcal{P}_{n+1} \cap \left(\tilde{\mathcal P}_{n+1} \cap \mathcal{P}_{n]} \right) = \{ 0 \}
\end{equation}
because the non-zero elements of the space $\mathcal{P}_{n+1}$
are polynomials of degree $n+1$.
Let us prove that the identity (\ref{constr-df-Pk+1b}) holds.
To this goal it will be sufficient to prove that the set
$$
\left\{ A^{+}_{j|n}\xi_{n,M} \ ; \  \xi_{n,M}\in B_{n}\right\}
$$
is contained in the left hand side of (\ref{dir-sum1}).
By construction $\mathcal{P}_{n+1}$ contains $B_{n+1}$. Let $\xi_{n, M}\in B_{n}$
be such that
%is a monic sub--space of $\mathcal{P}_{n+1]}$ and
$$
A_{j|n}^+ \xi_{n, M} = X_j M +  Q_{n]} \notin B_{n+1}.
$$
Since $B_{n+1}$ is a perturbation of the monomial basis of order $n+1$
in the $(X_j)$--coordinates
there exists $k \in D$ and $M^\prime \in \mathcal{M}_{e,n}$ such that
$$
A_{k|n}^+ \xi_{n, M^\prime} = X_k M^\prime + R_{n]} \in B_{n+1}
$$
($ R_{n]}$ is a polynomial of degree $\le n$) and
$$
X_k M^\prime = X_j M .
$$
It follows that
$$
A_{j|n}^+ M - A_{k|n}^+ M^\prime  \in \mathcal{P}_{n]} \cap \tilde{\mathcal P}_{n+1}.
$$
Therefore
$$
A_{j|n}^+ M = A_{k|n}^+ \xi_{n, M^\prime}
+ \left(A_{j|n}^+ M - A_{k|n}^+ M^\prime \right) \in \mathcal P_{n+1} \dot +
\left(\tilde{\mathcal P}_{n+1} \cap \mathcal{P}_{n]} \right).
$$
This proves (\ref{constr-df-Pk+1b}).
\end{proof}

{\bf Remark}.
The vector space sum  $\tilde{\mathcal P}_{n+1} + \mathcal{P}_{n]}$ is not direct.
However the vector space sum  $\mathcal P_{n+1} + \mathcal{P}_{n]}$ is direct and one has
$$
\tilde{\mathcal P}_{n+1} + \mathcal{P}_{n]}
=\mathcal P_{n+1} \dot + \mathcal{P}_{n]}.
$$

\noindent {\bf Remark}.
If the operators $A^{\varepsilon}_{v|n}$ are the CAP operators associated to a given
state on  $\mathcal{P}$, then the sub--space $\mathcal{P}_{n]} \cap \tilde{\mathcal P}_{n+1}$
necessarily consists of zero--norm vectors because, in this case, operators in the
spaces $A^{+}_{v|n}\mathcal P_{n}$ are orthogonal to $\mathcal{P}_{n]}$.

%%%%%%%%%%%%%%%%%%%%%%%%%%%%%%%%%%%%%%%%%%%%%%%%%%%%%%%%%%%%%
\subsection{Commutation relations}\label{Comm-rels}
%%%%%%%%%%%%%%%%%%%%%%%%%%%%%%%%%%%%%%%%%%%%%%%%%%%%%%%%%%%%%

In this section we briefly recall some known facts
about commutation relations canonically associated to
orthogonal polynomials (see \cite{[AcNh02]}, \cite{[AcKuoSt04b]}) which
will be used in the following section.
We refer the reader to \cite{[AcKuoSt04b]} for more detailed analysis.

\begin{theorem}\label{thm-com-rel}{\rm
Let be given:\\
-- a pre--Hilbert space $H$;\\
-- an orthogonal gradation of $H$:
% in the linear space sense:
$$
H=\bigoplus_{n\in\mathbb N}H_{n};
$$
%i.e. $H$ is the algebraic linear span of the $H_{n}$ ($n\in\mathbb N$).\\
-- a family of operators
$a^\pm_{j}:H_{n}\to H_{n\pm 1}$, $a^0_{j}:H_{n}\to H_{n}$,
($j\in\{1,\cdots ,d\}$)
$$
a^0_{j}=(a^0_{j})^* \qquad ;\qquad a^-_{j}=(a^+_{j})^*
\qquad ; \qquad j\in\{1,\cdots ,d\}.
$$
Define the operators $Y_j$ ($j\in\{1,\cdots ,d\}$) on $H$ by
\begin{equation}\label{df-Yj}
Y_j := a^+_{j} + a^0_{j} + a^-_{j}
\: , \qquad j\in\{1,\cdots ,d\}.
\end{equation}
Then the decomposition (\ref{df-Yj}) is unique and
the operators $Y_j$ commute on the algebraic linear span of the $H_{n}$
if and only if the operators $a^+_{j}$, $a^0_{j}$, $a^-_{j}$ satisfy
the following commutation relations on the same domain:
for all $j,k\in\{1,\cdots ,d\}$ such that $j<k$
\begin{equation}\label{creat-comm1}
[a^+_{j},a^+_{k}] = 0
\end{equation}
\begin{equation}\label{comm+j-k0j0k1}
[a^+_{j},a^-_{k}] + [a^0_{j},a^0_{k}] + [a^-_{j},a^+_{k}] = 0
\end{equation}
\begin{equation}\label{comm+j0k1}
[a^+_{j},a^0_{k}] + [a^0_{j},a^+_{k}] = 0
\end{equation}
}\end{theorem}

\begin{proof}
Clearly the operators $a^+_{j}$, $a^0_{j}$ are well defined on the
algebraic linear span of the $H_{n}$ and leave this domain invariant.
Given (\ref{df-Yj}) one has, for each $j,\: k\in\{1,\cdots ,d\}$:
\begin{eqnarray}\label{comm-Xj-1}
  0 =  [Y_j,Y_k] &=& [(a^+_{j} + a^0_{j} + a^-_{j}),(a^+_{k} + a^0_{k} + a^-_{k})] \\
                 &=& [a^+_{j},a^+_{k}] + [a^+_{j},a^0_{k}] + [a^0_{j},a^+_{k}] +[a^+_{j},a^-_{k}] \nonumber \\
                 &+& [a^0_{j},a^0_{k}] + [a^-_{j},a^+_{k}] + [a^0_{j},a^-_{k}] + [a^-_{j},a^0_{k}] + [a^-_{j},a^-_{k}] \nonumber
\end{eqnarray}
The mutual orthogonality of the $H_k$'s and the properties of the
$a^\epsilon_{k}$ imply that the commutativity of the $Y_j$s, is equivalent
the fact the expressions on different rows of the right hand side of
(\ref{comm-Xj-1}) are separately equal to zero.
Since the $5$--th row is the adjoint of the first one and the $4$--th row
is equivalent to the adjoint of the second one, the vanishing of all the
rows is equivalent to (\ref{creat-comm1}), (\ref{comm+j-k0j0k1}),
(\ref{comm+j0k1}) for all $j,k\in\{1,\cdots ,d\}$. But this is equivalent to
the validity of these relations for all $j,k\in\{1,\cdots ,d\}$ such that $j<k$
because all the relations are identically satisfied for $j=k$ and, exchanging the
roles of $j$ and $k$, the left hand sides of (\ref{creat-comm1}),
(\ref{comm+j-k0j0k1}) are transformed into its opposite and that of
(\ref{comm+j0k1}) remains unaltered.\\
Finally the uniqueness of the decomposition (\ref{df-Yj}) is established as in the
proof of Theorem \ref{th-Q--dec}.
\end{proof}

\section{Orthogonal polynomials and symmetric interacting Fock spaces}\label{Orth-pol-sym-IFS}

The notion of symmetric interacting Fock space is discussed in
Appendix \ref{App-S-IFS} below and in this section we will use freely the
definitions and notations of this appendix. The following theorem shows 
that orthogonal polynomial gradations define a very special sub--class 
of {\bf symmetric interacting Fock spaces}.

\begin{theorem}\label{Orth-pol-def-sym-IFS}{\rm
Let $\varphi$ be a state on $\mathcal P$ and let
${\mathcal P}=\bigoplus_{n\in{\mathbb N}}{\mathcal P}_{n}$ its orthogonal
polynomial gradation. Denote:\\
-- for $n\in\mathbb N$, $\langle \ \cdot \ ,\ \cdot \ \rangle_{n}$
the restriction on $\mathcal P_n$ of the pre--scalar product
$\langle \ \cdot \ ,\ \cdot \ \rangle$ induced by $\varphi$ on $\mathcal P$;\\
-- for $j\in D$
\begin{equation}\label{q-dec-reconstr}
X_{j}=a^{+}_{j}+a^{0}_{j}+a^{-}_{j}
\end{equation}
the quantum decomposition of the coordinate $X_{j}$ with respect to $\varphi$;\\
-- $a^{+}:v=\sum_{j\in D}v_{j}e_{j}\in\mathbb C^d\to a^{+}_{v}:=\sum_{j\in D}v_{j}a^{+}_{j}
\in \mathcal L_a(\mathcal P,\langle \ \cdot \ ,\ \cdot \ \rangle)$
the creation map.\\
Then the pair    
\begin{equation}\label{df-PIFS}
\left( (\mathcal P_n, \langle \ \cdot \ ,\ \cdot \ \rangle_{n})_{n\in\mathbb N} ,a^+\right)
\end{equation}
is a symmetric interacting Fock space with the following properties:\\

(i) The restriction on $\mathcal P_{\mathbb R}$ of the pre--scalar product
$\langle \ \cdot \ ,\ \cdot \ \rangle$ is real valued and
there exists a family of gradation preserving self--adjoint operators 
$a^{0}_{j}:\mathcal P\cdot \Phi_0\to \mathcal P\cdot \Phi_0 $ ($j\in D$)
such that
\begin{equation}\label{aej-pres-PR}
a^{\varepsilon}_{j}(\mathcal P_{\mathbb R}\cdot \Phi_0) \subseteq
\mathcal P_{\mathbb R}\cdot \Phi_0
\: ,\qquad \forall \varepsilon\in \{+,0,-\} \ , \ j\in D ,
\end{equation}
and the coordinate operators $X_j$ mutually commute;\\
 (ii) the vacuum vector $\Phi$ of the IFS (\ref{df-PIFS}) (identified
to the vector $\Phi_0\in \mathcal P\cdot \Phi_0$ ) is cyclic for the
polynomial algebra generated by the family (\ref{q-dec-reconstr}).\\

\noindent Conversely, given a symmetric interacting Fock space on $\mathbb C^d$
$$
\left( (\hat{\mathcal P}_{n},
\langle \ \cdot \ ,\ \cdot \ \rangle_{IFS,n}), \hat a^+\right)
$$
and a family of gradation preserving operators $\hat {a}^{0}_j$ ($j\in D$)
such that the operators
\begin{equation}\label{comm-coord}
\hat X_j:=\hat a^{+}_{j}+\hat a^{0}_{j}+(\hat a^{+}_{j})^*
\qquad ;\qquad j\in D
\end{equation}
commute and, denoting $\hat{\mathcal P}$, (resp. $\hat{\mathcal P}_{\mathbb R}$)
the $*$--algebra (resp. real $*$--algebra) generated by the $\hat X_j$
conditions (i) and (ii) above are satisfied.\\
Then there exists a unique state $\varphi$ on $\mathcal P$ characterized by
the property that for all maps $n:D\to\mathbb N$, denoting $\Phi$
the vacuum vector of $\hat{\mathcal P}$, one has:
\begin{equation}\label{fXjn=fldn}
\varphi (X_{1}^{n_1}\cdots X_{d}^{n_d})
=\langle \Phi, \hat X_{1}^{n_1}\cdots \hat X_{d}^{n_d}\Phi\rangle
\qquad ;\qquad \forall n_1,\dots n_d\in\mathbb N 
\end{equation}
Moreover
%, because of conditions (\ref{aej-pres-PR}), (\ref{comm-coord}) scalar products
the expectation values (\ref{fXjn=fldn}) are real valued.\\
In particular, there is a symmetric IFS isomorphism (see Definition
(\ref{df-SIFSS}))
$$
U \ : \ \left( (\mathcal P_{n},
\langle \ \cdot \ ,\ \cdot \ \rangle_{n}),a^+\right)
\to \left( (\hat{\mathcal P}_{n},
\langle \ \cdot \ ,\ \cdot \ \rangle_{IFS,n}),\hat a^+\right)
$$
preserving the real structures of both spaces and such that
\begin{equation}\label{IFS-coord}
X_{j} = U^*\hat{a}^{+}_jU+U^*\hat{a}^{0}_jU+(U^*\hat{a}^{+}_jU)^*
\end{equation}
is the quantum decomposition of the $X_{j}$ with respect to $\varphi$.
}\end{theorem}

\begin{proof}
Let $\varphi$ be a state on $\mathcal P$ and let $a^{\varepsilon}$ be the
associated CAP operators. Let us first prove the pair (\ref{df-PIFS})
satisfies the conditions of Definition \ref{df-IFS}.
We know that $\mathcal P_0$ is $1$--dimensional with the scalar product uniquely
determined by the condition $\|\Phi_0\|=1$. Lemma (\ref{(a+j*=a-j}) implies that
$a^+$ is adjointable. Finally
$$
\mathcal P_{n+1}=P_{n+1}\mathcal P\cdot\Phi_0=(P_{n+1]}-P_{n]})\mathcal P\cdot\Phi_0
=(P_{n+1]}-P_{n]})P_{n+1]}\mathcal P\cdot\Phi_0
=P_{n+1}\mathcal P_{n+1]}\cdot\Phi_0
$$
and $\mathcal P_{n+1]}\cdot\Phi_0$ is the complex linear span of the set
$\left\{X_v\mathcal P_{n]} \ : \ v\in \mathbb R^d\right\}$. Therefore, to verify
condition (\ref{prop-creat}) of Definition (\ref{df-IFS}), it is
sufficient to prove that $a^+(V)\mathcal P_{n}$ contains
$\left\{P_{n+1}X_v\mathcal P_{n]}\cdot\Phi_0 \ : \ v\in \mathbb R^d\right\}$.
This follows from the symmetric Jacobi relations because for any $v\in \mathbb R^d$:
$$
P_{n+1}X_vP_{n]}
=P_{n+1}X_v(P_{n} + P_{n-1]})
=P_{n+1}X_vP_{n} =a^+_{v|n}.
$$
Thus $\left( (\mathcal P_n, \langle \ \cdot \ ,\ \cdot \ \rangle_{n}),a^+\right)$
is an IFS. That it is a symmetric IFS follows from Definition (\ref{df-SIFSS})
and the commutativity of the creators, established in section (\ref{Comm-rels}).
Property (i) follows from the quantum decomposition of the coordinates.
Property (ii) holds by definition of $\mathcal P$.\\

\noindent Conversely, let $\left( (\hat{\mathcal P}_{n},
\langle \ \cdot \ ,\ \cdot \ \rangle_{IFS,n}),\hat a^+\right)$
be an interacting Fock space on $\mathbb C^d$ and suppose that
conditions (i) and (ii) above are satisfied in the sense specified in the
statement of the theorem.
Then, since the operators
\begin{equation}\label{comm-coord2}
\hat{X}_j:=a^{+}_{j}+a^{0}_{j}+(a^{+}_{j})^* \: ,\qquad j\in D,
\end{equation}
are self--adjoint, property (i) implies that the complex $*$--algebra
$\hat{\mathcal P}$ generated by them is commutative.\\
Since $\mathcal P$ is isomorphic to the free abelian $*$--algebra with
identity and $d$ self--adjoint generators, there exists a $*$--algebra
homomorphism $\pi:\mathcal P\to\hat{\mathcal P}$ characterized by the
property that
$$
\pi (X_j):=\widehat X_j =  \hat{a}^{+}_j+\hat{a}^{0}_j+(\hat{a}^{+}_j)^*
\: ,\qquad j\in D.
$$
Denoting $\varphi_F$ the restriction of the Fock state
$\langle\Phi \ \cdot  \ ,  \  \cdot \ \Phi\rangle$ on $\hat{\mathcal P}$,
define the state $\varphi$ on $\mathcal P$ by
\begin{equation}\label{fi:=fiF-circ-pi}
\varphi:=\varphi_F\circ \pi .
\end{equation}
Then (\ref{fXjn=fldn}) holds by construction.
Since the monomials are linearly independent in $\mathcal P$, for any map
$n:D\to\mathbb N$, the map
$$
X_{1}^{n_1}\cdots X_{d}^{n_d}\Phi_0 \ \mapsto \
\widehat X_1^{n_1}\cdots \widehat X_{d}^{n_d}\Phi
%(\hat{a}^{+}1}+\hat{a}^{0}1}+(\hat{a}^{+}1})^*)^{n_1}\cdots
%(\hat{a}^{+}d}+\hat{a}^{0}d}+(\hat{a}^{+}d})^*)^{n_d}\Phi
$$
can be extended to a linear map $U:\mathcal P\cdot \Phi_0 \to
\hat{\mathcal P}\cdot \Phi$ which is onto by condition (ii).
(\ref{fXjn=fldn}) implies that this extension preserves scalar products,
therefore $U$ is a unitary isomorphism of pre--Hilbert spaces. It preserves the
real structure of the corresponding spaces because of condition (i).
Moreover $U$ satisfies, for $j \in D$,
\begin{eqnarray}
X_{j} &=& U^*(\hat{a}^{+}_j+\hat{a}^{0}_j+(\hat{a}^{+}_j))^*U \nonumber \\
   &=&  U^*\hat{a}^{+}_jU+U^*\hat{a}^{0}_jU+U^*(\hat{a}^{+}_j)^*U \nonumber \\
   &=&  U^*\hat{a}^{+}_jU+U^*\hat{a}^{0}_jU+(U^*\hat{a}^{+}_jU)^* \label{1st-qdec-Xj}
\: ,\qquad j\in D
\end{eqnarray}
which implies
$$
\mathcal P_{n]}=U^*\hat{\mathcal P}_{n]}U
\: ,\qquad n\in \mathbb N .
$$
Therefore, since $U$ is unitary,
$$
\mathcal P_{n}=\mathcal P_{n-1]}^\perp\cap\mathcal P_{n]}
=U^*\hat{\mathcal P}_{n-1]}^\perp U\cap U^*\hat{\mathcal P}_{n]}U
=U^*\hat{\mathcal P}_{n}U
\: , \qquad n\in \mathbb N .
$$
Denote $X_{j}=a^{+}_{j}+a^{0}_{j}+(a^{+}_{j})^*$ the quantum decomposition
of the $X_{j}$ associated to the state $\varphi$ defined by
(\ref{fi:=fiF-circ-pi}). Then (\ref{1st-qdec-Xj}) implies that
$$
X_{j}=a^{+}_{j}+a^{0}_{j}+(a^{+}_{j})^*
= U^*\hat{a}^{+}_jU+U^*\hat{a}^{0}_jU+(U^*\hat{a}^{+}_jU)^*
$$
and the operators $a^{\pm}_{j}$ (resp. $a^{0}_{j}$) and
$U^*\hat{a}^{\pm}_{j}U$ (resp. $U^*\hat{a}^{0}_jU$) are of degree
$\pm 1$ (resp. $0$) with respect to the same orthogonal gradation.
From the uniqueness of the quantum decomposition
(see Theorem \ref{thm-com-rel}) we conclude that
$$
a^{\pm}_{j}=U^*\hat{a}^{\pm}_{j}U
\: , \qquad a^{0}_{j}=U^*\hat{a}^{0}_jU
\: , \qquad j\in D.
$$
Thus $U$ is an isomorphism of IFS.
Since the $X_j$ commute, we know from Theorem (\ref{thm-com-rel}) that
the operators $\hat{a}^+_{j}$ mutually commute so that the IFS is symmetric
(see Definition \ref{df-SIFSS}).
\end{proof}

Theorem (\ref{Orth-pol-def-sym-IFS}) motivates the following definition.

\begin{definition}\label{comm-coord3}{\rm
Let $\left( (\hat{\mathcal P}_{n},
\langle \ \cdot \ ,\ \cdot \ \rangle_{IFS,n}),\hat{a}^+\right)$ be an
interacting Fock space on $\mathbb C^d$. A family of gradation preserving
self--adjoint operators
$a^{0}_{j}:\hat{\mathcal P}_{n}\to \hat{\mathcal P}_{n}$ ($j\in D$)
is said to define a {\bf $3$--diagonal structure} on
$\left( (\hat{\mathcal P}_{n},
\langle \ \cdot \ ,\ \cdot \ \rangle_{IFS,n}),\hat{a}^+\right)$
if the operators $\hat{X}_j$, defined by (\ref{comm-coord1}), satisfy
conditions (i) and (ii) of the second part of Theorem
(\ref{Orth-pol-def-sym-IFS}).
}\end{definition}

{\bf Remark}.
From the Remark after Theorem \ref{Orth-pol-def-sym-IFS} it follows that
an interacting Fock space with a $3$--diagonal structure is necessarily
symmetric. Therefore, by Lemma \ref{df-sym-IFS}, we can identify it,
up to isomorphism, to its symmetric tensor representation (see
Lemma \ref{df-sym-IFS}).\\

\noindent{\bf Remark}.
The assignment of a gradation preserving self--adjoint operator
$a^{0}_{j}:\mathcal P\to \mathcal P$ ($j\in D$) is equivalent to the
assignment of a sequence of self--adjoint operators
$a^{0}_{j|n}:\mathcal P_{n}\to \mathcal P_{n}$ ($n\in\mathbb N$).

\begin{definition}\label{df-pair-Omn-a0jn}{\rm
Let be given a finite dimensional vector space $V$, and
a sequence\\ $\tilde\Omega^{\widehat\otimes}
:=(\tilde\Omega^{\widehat\otimes}_{n})$,
inductively defined as in Theorem \ref{equiv-SIFS-seq-PDkern}.\\
Let $\Gamma(V,\tilde\Omega):=((V^{\widehat\otimes n},\langle \ \cdot
 \ ,  \ \cdot \ \rangle_{n}),\ell^*)$ be the symmetric IFS on
$V$ associated to the pair $(V \ , \ (\tilde\Omega^{\widehat\otimes}_{n}))$
according to Theorem \ref{equiv-SIFS-seq-PDkern} and let, for each
$n\in\mathbb N$ and $j\in D$,\\
$a^{0}_{j|n}:(V^{\widehat\otimes n},\langle \ \cdot
 \ ,  \ \cdot \ \rangle_{n})\to (V^{\widehat\otimes n},\langle \ \cdot
 \ ,  \ \cdot \ \rangle_{n})$
be a sequence of self--adjoint operators.\\
The pair $(\tilde\Omega^{\widehat\otimes} \ , \ (a^{0}_{j|n}))$
is said to induce a $3$--diagonal structure on $\Gamma(V,\tilde\Omega)$,
if the family of gradation preserving self--adjoint operators
$a^{0}_{j}:\Gamma(V,\tilde\Omega)\to \Gamma(V,\tilde\Omega)$ ($j\in D$)
is a $3$--diagonal structure on $\Gamma(V,\tilde\Omega)$ in the sense
of Definition \ref{comm-coord}.
}\end{definition}

\begin{theorem}\label{equiv-seq-PDkern3d-struc}{\rm
In the notations of Theorem \ref{Orth-pol-def-sym-IFS} and of Definition
\ref{df-pair-Omn-a0jn}, any state $\varphi$ on $\mathcal P$
uniquely defines a pair $(\tilde\Omega^{\widehat\otimes} \ , \ (a^{0}_{j|n}))$
that induces a $3$--diagonal structure on $\Gamma(V,\tilde\Omega)$.\\
Conversely, any pair $(\tilde\Omega^{\widehat\otimes} \ , \ (a^{0}_{j|n}))$
that induces a $3$--diagonal structure on
$\Gamma(V,\tilde\Omega^{\widehat\otimes})$
uniquely defines a state $\varphi$ on $\mathcal P$.
}\end{theorem}

\begin{proof}
 Both statements are immediate consequences of the corresponding
statements in Theorem \ref{Orth-pol-def-sym-IFS}.
\end{proof}

{\bf Remark}. Theorem \ref{equiv-seq-PDkern3d-struc} implies that
the (standard) interacting Fock spaces on $\mathbb C^d$ of the form
\begin{equation}\label{tens-repr-sym-IFS-V2}
\left\{ \left(V^{\widehat\otimes  n} \ ,  \
\langle \ \cdot \ , \Omega_n \ \cdot \ \rangle_{\widehat\otimes, n}\right)
 \ , \ \hat\ell^* \right\}
\end{equation}
with a $3$--diagonal structure provide a universal model for the theory
of orthogonal polynomials in $d$ variables.\\

{\bf Remark}.
From section \ref{Comm-rels}
we know that the operators (\ref{comm-coord}) commute if and only if the
relations (\ref{creat-comm1}), (\ref{comm+j-k0j0k1}), (\ref{comm+j0k1})
hold. On the other hand from Theorem \ref{equiv-IFS-seq-PDkern} we know
that IFS on $\mathbb C^d$ are characterized by sequences of PD kernels
on $\mathbb C^d$ and, from the identity (\ref{n-IFS-PD-kern}) we know
that these PD kernels have the form $a^-(u)a^+(v)$ ($u,v\in\mathbb C^d$).
Since products of this form appear in the commutators in
(\ref{creat-comm1}), (\ref{comm+j-k0j0k1}), (\ref{comm+j0k1}),
it follows that these commutation relations create constraints between
the kernels
defining the scalar products in the IFS and the operators $a^{0}_{j}$.
In the following section we will investigate these constraints.\\

%%%%%%%%%%%%%%%%%%%%%%%%%%%%%%%%%%%%%%%%%%%%%%%%%%%%%%%%%%%%%%%%%%%%%%%%%%%%%%%%%%%%%
\section{Implications of the commutation relations}\label{Impl-comm-rels}
%%%%%%%%%%%%%%%%%%%%%%%%%%%%%%%%%%%%%%%%%%%%%%%%%%%%%%%%%%%%%%%%%%%%%%%%%%%%%%%%%%%%%

With the notations (\ref{1df-a+j|n}), (\ref{1df-a0j|n}), (\ref{1df-a-j|n}),
the tri--diagonal relation (\ref{Symm-Jac-rel}) takes the form
$$
X_jP_{n}= a^+_{j|n} + a^0_{j|n} + a^-_{j|n}
\: , \qquad \forall j\in D \ , \ \forall n\in{\mathbb N},
$$
or equivalently, due to Lemma \ref{(a+j*=a-j}
\begin{equation}\label{Symm-Jac-rel-ind}
a^+_{j|n} = X_jP_{n} - a^0_{j|n} - (a^+_{j|n-1})^*
\: , \qquad \forall j\in D \ , \ \forall n\in{\mathbb N}.
\end{equation}
This can be interpreted as an inductive relation that, given
$a^+_{j|n-1}$ ($j\in D$), the scalar product on $\mathcal P_n$ and $a^0_{j|n}$,
uniquely defines $a^+_{j|n}$. Notice that, if $a^0_{j|n}$
is chosen to be a pre--Hilbert space operator, in particular mapping
zero norm vectors into zero norm vectors, and if it maps real vectors in
$\mathcal P_n$ into real vectors, then $a^+_{j|n-1}$ will
have the same properties because $X_j$ has these properties
and $(a^+_{j|n-1})^*$ has these properties by the induction construction.\\
In this section we will establish the constraints, imposed by the
commutation relations, on the objects that define the induction relation,
namely the $a^+_{j|n-1}$ ($j\in D$), the scalar product on
$\mathcal P_n$ and the $a^0_{j|n}$.\\

{\bf Remark}.
Recall that, if $A$ is an adjointable operator on a pre--Hilbert space, then
its real and imaginary parts are defined by
\begin{equation}\label{df-ReA-ImA}
A = \frac{1}{2}(A+A^*) + \frac{1}{2}(A-A^*)
=: \hbox{Re}(A)+i\hbox{Im}(A).
\end{equation}
Similarly, for any PD kernel $\tilde\Omega$ one has
$$
\tilde\Omega(e_j,e_k)^* = \tilde\Omega(e_k,e_j)\qquad\qquad\qquad\qquad\qquad\qquad\qquad
$$
therefore
\begin{equation}\label{df-OmR-OmI}
\tilde\Omega(e_j,e_k)
= \frac{1}{2}((\tilde\Omega(e_j,e_k))+\tilde\Omega(e_j,e_k)^* )
+ \frac{1}{2}((\tilde\Omega(e_j,e_k))-\tilde\Omega(e_j,e_k)^* )
\end{equation}
$$
= \frac{1}{2}((\tilde\Omega(e_j,e_k))+\tilde\Omega (e_k,e_j))
+ \frac{1}{2}((\tilde\Omega(e_j,e_k))-\tilde\Omega (e_k,e_j))
=: \tilde\Omega_R (e_j,e_k) + \tilde\Omega_I (e_j,e_k)
$$
with
$$
\tilde\Omega_R (e_j,e_k)  = \tilde\Omega_R (e_k,e_j)= \tilde\Omega_R (e_j,e_k)^*
\: , \qquad
-\tilde\Omega_I (e_j,e_k) = \tilde\Omega_I (e_k,e_j)= \tilde\Omega_I (e_j,e_k)^* .
$$
Thus any PD kernel $\tilde\Omega$ is the sum of a symmetric kernel
and a symplectic kernel.\\

\noindent In this section we will use the notations (\ref{1df-a+j|n}),
(\ref{1df-a0j|n}),
(\ref{1df-a-j|n}) and in the following $(\tilde\Omega_{n})$ will denote
the sequence of positive definite (PD) kernels defined by
$\tilde\Omega_{0}=1\in \mathbb C$ and
\begin{equation}\label{df-tilde-Omega(n+1)}
\tilde\Omega_{n+1}(e_j,e_k):=(a^-_{j}a^+_{k})_{|n}:=(a^+_{j|n})^*a^+_{k|n}
\: , \qquad \forall n\in\mathbb N \ , \ \forall j,k\in D.
\end{equation}
Since the operators $a^+_{k|n}$ map real polynomials into real polynomials,
it follows that also the operators $\tilde\Omega_{n}$ have this property.
By linearity this is equivalent to say that the $a^+_{k|n}$ map
maps real vectors in $\mathcal P_n$ into real vectors.

\begin{lemma}\label{struc-PD-kern}{\rm
The commutation relations (\ref{comm+j-k0j0k1}), i.e.
\begin{equation}\label{comm+j-k0j0k1-pf}
[a^+_{j},a^-_{k}] + [a^0_{j},a^0_{k}] + [a^-_{j},a^+_{k}] = 0
\end{equation}
are equivalent to
\begin{equation}\label{comm+j-k0j0k1c0}
\tilde\Omega_{1}(e_j,e_k)=\tilde\Omega_{1}(e_k,e_j)\in\mathbb R
\end{equation}
\begin{equation}\label{comm+j-k0j0k1c}
\hbox{Im}(\tilde\Omega_{n+1}(e_j,e_k))
=\hbox{Im}(a^+_{k|n-1}(a^+_{j|n-1})^*)
+\hbox{Im}(a^0_{k|n}a^0_{j|n})
\: , \quad \forall n\geq 1,
\end{equation}
for all $j,k\in D$ such that $j<k$ and all $n\in\mathbb N$.
}\end{lemma}

\begin{proof}
For $j,k$ and $n$ as in the statement, the commutation
relation (\ref{comm+j-k0j0k1}) is
$$
[a^+_{j},a^-_{k}] + [a^0_{j},a^0_{k}] + [a^-_{j},a^+_{k}] = 0 \Leftrightarrow
 [a^+_{j}a^-_{k}-a^-_{k}a^+_{j}]
+ [a^0_{j}a^0_{k}-a^0_{k}a^0_{j}] + [a^-_{j}a^+_{k}-a^+_{k}a^-_{j}] = 0
$$
\begin{equation}\label{comm+j-k0j0k1a}
\Leftrightarrow (a^+_{j})^*a^+_{k} - (a^+_{k})^*a^+_{j}
= a^+_{k}a^-_{j}-a^+_{j}a^-_{k}+ a^0_{k}a^0_{j}-a^0_{j}a^0_{k}.
\end{equation}
These are identically satisfied for $j=k$ and, exchanging $j$ and $k$,
one finds an equivalent relation.
Therefore it is sufficient to consider the case $j<k$.\\
On $\mathcal P_0$, (\ref{comm+j-k0j0k1a}) is equivalent to:
$$
(a^+_{j})^*a^+_{k}\Phi_0 - (a^+_{k})^*a^+_{j}\Phi_0 = a^+_{k}a^-_{j}\Phi_0
-a^+_{j}a^-_{k}\Phi_0 + a^0_{k}a^0_{j}\Phi_0-a^0_{j}a^0_{k}\Phi_0
$$
$$
\Leftrightarrow
(a^+_{j})^*a^+_{k}\Phi_0 - (a^+_{k})^*a^+_{j}\Phi_0 = 0.
$$
Recalling (\ref{df-tilde-Omega(n+1)}) the above identity becomes
$$
\tilde\Omega_{1}(e_j,e_k)\Phi_0-\tilde\Omega_{1}(e_k,e_j)\Phi_0
$$
and, since $\tilde\Omega_{1}(e_j,e_k)$ maps $\mathbb C\cdot \Phi_0$ into itself,
the above identity is equivalent (up to obvious identifications) to
$$
\tilde\Omega_{1}(e_j,e_k)=\tilde\Omega_{1}(e_k,e_j) \ \in\mathbb C
$$
and from condition (\ref{aej-pres-PR}) and the identity
$$
\tilde\Omega_{n+1}(e_j,e_k)^*:=(((a^+_{j})^*a^+_{k})_{|n})^*
=((a^+_{k})^*a^+_{j})_{|n}=\tilde\Omega_{n+1}(e_k,e_j)
$$
it follows that $\tilde\Omega_{1}(e_j,e_k)\in\mathbb R$.
This proves (\ref{comm+j-k0j0k1c0}). Let $n>0$.
From
$$
(a^+_{k})^*a^+_{j} = ((a^+_{j})^*a^+_{k})^*
$$
one deduces that for any $\xi_{n},\eta_{n}\in\mathcal P_{n}$
%, restricting to $\mathcal P_n$, one obtains
$$
\langle (a^+_{j})^*a^+_{k}\xi_{n},\eta_{n} \rangle_{n}
=\langle \xi_{n},(a^+_{k})^*a^+_{j}\eta_{n} \rangle_{n}
\Leftrightarrow ((a^+_{k})^*a^+_{j})_{|n} = (((a^+_{j})^*a^+_{k})_{|n})^* .
$$
Therefore the identity (\ref{comm+j-k0j0k1a}), restricted to $\mathcal P_n$
is equivalent to the fact that, for each $n\in\mathbb N$ and each $j\in D$,
\begin{equation}\label{comm+j-k0j0k-n}
(a^+_{j|n})^*a^+_{k|n} - (a^+_{k|n})^*a^+_{j|n}
= a^+_{k|n-1}a^-_{j|n}-a^+_{j|n-1}a^-_{k|n}+ a^0_{k|n}a^0_{j|n}-a^0_{j|n}a^0_{k|n}
\end{equation}
or equivalently
\begin{equation}\label{comm+j-k0j0k1b}
\tilde\Omega_{n+1}(e_j,e_k)-\tilde\Omega_{n+1}(e_j,e_k)^*
=\tilde\Omega_{n+1}(e_j,e_k)-\tilde\Omega_{n+1}(e_k,e_j)
=2i\hbox{Im}(\tilde\Omega_{n+1}(e_j,e_k))
\end{equation}
$$
= (a^+_{k}a^-_{j})_{|n}-(a^+_{j}a^-_{k})_{|n}+ (a^0_{k}a^0_{j})_{|n}
-(a^0_{j}a^0_{k})_{|n}
$$
%The kernel $a^+_{k}a^-_{j}$ can be defined inductively.
Now notice that for any $\xi_{n},\eta_{n}\in\mathcal P_{n}$
$$
a^+_{k}a^-_{j}\eta_{n}=a^+_{k|n-1}a^-_{j|n}\eta_{n}
=a^+_{k|n-1}(a^+_{j|n-1})^*\eta_{n}
$$
i.e.
$$
(a^+_{k}a^-_{j})_{|n}=a^+_{k|n-1}(a^+_{j|n-1})^*=(a^+_{j|n-1}(a^+_{k|n-1})^*)^* .
$$
Since the $a^0_{j}$ preserve the gradation and are self--adjoint,
$(a^0_{k}a^0_{j})_{|n}=a^0_{k|n}a^0_{j|n}$,
therefore (\ref{comm+j-k0j0k1b}) becomes
\begin{equation}\label{comm+j-k0j0k1e}
2i\hbox{Im}(\tilde\Omega_{n+1}(e_j,e_k))
= (a^+_{k}a^-_{j})_{|n}-(a^+_{j}a^-_{k})_{|n}+ (a^0_{k}a^0_{j})_{|n}
-(a^0_{j}a^0_{k})_{|n}
\end{equation}
$$
=a^+_{k|n-1}(a^+_{j|n-1})^* - a^+_{j|n-1}(a^+_{k|n-1})^*
+a^0_{k|n}a^0_{j|n}-a^0_{j|n}a^0_{k|n}
$$
$$
=a^+_{k|n-1}(a^+_{j|n-1})^* - (a^+_{k|n-1}(a^+_{j|n-1})^*)^*
+a^0_{k|n}a^0_{j|n}-(a^0_{k|n}a^0_{j|n})^*
$$
$$
=2i\hbox{Im}(a^+_{k|n-1}(a^+_{j|n-1})^*)
+2i\hbox{Im}(a^0_{k|n}a^0_{j|n})
$$
and this is equivalent to (\ref{comm+j-k0j0k1c}).
\end{proof}

{\bf Remark}.
Lemma \ref{struc-PD-kern} implies that the commutation relations
(\ref{comm+j-k0j0k1}),
associated to a state on $\mathcal P$, inductively fix the symplectic
parts of the kernels $\tilde\Omega_{n+1}$. Since, adding a symplectic
kernel to any PD kernel, one still obtains a PD kernel, fixing the
imaginary part of a PD kernel leaves its symmetric part completely
arbitrary up to the conditions of positive--definiteness and of
preservation of the real structure.

\begin{lemma}\label{lm-1st-ind-rel-a0j}{\rm
The commutation relations (\ref{comm+j0k1}), i.e.
\begin{equation}\label{comm+j0k1a}
[a^+_{j},a^0_{k}] + [a^0_{j},a^+_{k}] = 0
\end{equation}
are equivalent to
\begin{equation}\label{comm+0-ind}
a^0_{j|n+1}a^+_{k|n}-a^0_{k|n+1}a^+_{j|n}
= a^+_{k|n}a^0_{j|n}-a^+_{j|n}a^0_{k|n}
\end{equation}
for all $j,k\in D$ such that $j<k$ and all $n\in\mathbb N$.
}\end{lemma}

\begin{proof}
The commutation relations (\ref{comm+j0k1a})
are identically satisfied for $j=k$ and, exchanging $j$ and $k$,
one finds an equivalent relation.
Therefore it is sufficient to consider the case $j<k$.
In this case, with arguments similar to those used in the proof of Lemma
(\ref{struc-PD-kern}), one shows that (\ref{comm+j0k1a}) is equivalent to
$$
a^+_{j}a^0_{k}-a^0_{k}a^+_{j} + a^0_{j}a^+_{k}-a^+_{k}a^0_{j} = 0
\Leftrightarrow
$$
$$
\Leftrightarrow
(a^+_{j}a^0_{k})_{|n}-(a^0_{k}a^+_{j})_{|n} + (a^0_{j}a^+_{k})_{|n}
-(a^+_{k}a^0_{j})_{|n} = 0
\: ,\quad \forall n\in\mathbb N
$$
$$
\Leftrightarrow
a^+_{j|n}a^0_{k|n}-a^0_{k|n+1}a^+_{j|n} + a^0_{j|n+1}a^+_{k|n}
-a^+_{k|n}a^0_{j|n} = 0
$$
$$
\Leftrightarrow
a^0_{j|n+1}a^+_{k|n}-a^0_{k|n+1}a^+_{j|n}
 = a^+_{k|n}a^0_{j|n} - a^+_{j|n}a^0_{k|n}
$$
that is (\ref{comm+0-ind}).
\end{proof}

{\bf Remark}.
Since the inductive form of the creators is uniquely determined
by condition (\ref{Symm-Jac-rel-ind}), the identity (\ref{comm+0-ind})
can be interpreted as
a necessary condition to be satisfied by the $a^{0}_{j|n+1}$ once given
the $a^{0}_{j|n}$ ($j\in D$). Notice that the inductive system of
equations (\ref{comm+0-ind}) always admits the zero solution
given by the sequence
$$
a^{0}_{j|n} =0 \: , \qquad \forall j\in D \ , \ \forall n\in\mathbb N .
$$
\begin{lemma}\label{th2-comm-crs}{\rm
The commutation relations (\ref{creat-comm1}) (commutativity of creators) are
equivalent to the following identities
\begin{equation}\label{comm-cr1a}
a^{0}_{k|n+1}a^+_{j|n}- a^{0}_{j|n+1}a^+_{k|n}
=X_{k}a^+_{j|n}-X_{j}a^+_{k|n}
+ 2i\hbox{Im}(a^+_{k|n-1}(a^+_{j|n-1})^*)
+ 2i\hbox{Im}(a^0_{k|n}a^0_{j|n})
\end{equation}
for all $j,k\in D$ such that $j<k$ and all $n\in\mathbb N$.
}\end{lemma}

\begin{proof}
The commutativity of creators is identically satisfied
for $j=k$ and, exchanging $j$ and $k$, one finds the same relation
up to a common sign.
Therefore it is sufficient to consider the case $j<k$.\\
Due to (\ref{Symm-Jac-rel-ind}), the commutativity
of creators is equivalent to
$$
a^+_{j}a^+_{k}=a^+_{k}a^+_{j} \ \iff \
a^+_{j}a^+_{k}P_{n}=a^+_{k}a^+_{j}P_{n} \ \iff \
a^+_{j|n+1}a^+_{k|n}=a^+_{k|n+1}a^+_{j|n} \quad ;\quad
\forall j\in D,  \forall n\in\mathbb N
$$
Using the quantum decomposition of the $X_{j}$ this becomes equivalent to
$$
(X_{j} - a^{0}_{j} - (a^+_{j})^*)a^+_{k|n}
=(X_{k|n+1}-a^{0}_{k|n+1}-(a^+_{k|n})^*)a^+_{j|n}
$$
$$
 \iff
X_{j}a^+_{k|n} - a^{0}_{j|n+1}a^+_{k|n} - (a^+_{j|n})^*a^+_{k|n}
=X_{k}a^+_{j|n}-a^{0}_{k|n+1}a^+_{j|n}-(a^+_{k|n})^*a^+_{j|n}
$$
$$
 \iff
a^{0}_{k|n+1}a^+_{j|n}- a^{0}_{j|n+1}a^+_{k|n}
=X_{k}a^+_{j|n}-X_{j}a^+_{k|n}
+ (a^+_{j|n})^*a^+_{k|n} - (a^+_{k|n})^*a^+_{j|n}
$$
$$
 \iff
a^{0}_{k|n+1}a^+_{j|n}- a^{0}_{j|n+1}a^+_{k|n}
=X_{k}a^+_{j|n}-X_{j}a^+_{k|n}
+ 2i\hbox{Im}(\tilde\Omega_{n+1}(e_{j},e_{k}))
$$
Using (\ref{comm+j-k0j0k1c}) this becomes
$$
a^{0}_{k|n+1}a^+_{j|n}- a^{0}_{j|n+1}a^+_{k|n}
=X_{k}a^+_{j|n}-X_{j}a^+_{k|n}
+ 2i\hbox{Im}(a^+_{k|n-1}(a^+_{j|n-1})^*)
+ 2i\hbox{Im}(a^0_{k|n}a^0_{j|n})
$$
which is equivalent to (\ref{comm-cr1a}).
\end{proof}

\begin{lemma}\label{lm-equivalence}{\rm
The linear system in the unknowns $(a^0_{k|n})$, given by equations
(\ref{comm+0-ind}), (\ref{comm-cr1a}), i.e.
\begin{equation}\label{comm+0-indR1}
a^0_{k|n}a^+_{j|n-1}-a^0_{j|n}a^+_{k|n-1}
= a^+_{j|n-1}a^0_{k|n-1}-a^+_{k|n-1}a^0_{j|n-1}
\end{equation}
$$
a^{0}_{k|n}a^+_{j|n-1}- a^{0}_{j|n}a^+_{k|n-1} =
$$
\begin{equation}\label{comm-cr1aR1}
=X_{k}a^+_{j|n-1}-X_{j}a^+_{k|n-1}
+ 2i\hbox{Im}(a^+_{k|n-2}(a^+_{j|n-2})^*)
+ 2i\hbox{Im}(a^0_{k|n-1}a^0_{j|n-1})
\end{equation}
($j,k\in D$, $j<k$) is equivalent to the single linear system given by
(\ref{comm+0-indR1}).
}\end{lemma}

\begin{proof}
Since the left hand sides of (\ref{comm+0-indR1}) and
(\ref{comm-cr1aR1}) are equal, the same must be true for the right hand sides,
therefore one must have
$$
a^+_{j|n-1}a^0_{k|n-1}-a^+_{k|n-1}a^0_{j|n-1}
$$
\begin{equation}\label{nec-cond-sol-LS}
=X_{k}a^+_{j|n-1}-X_{j}a^+_{k|n-1}
+ 2i\hbox{Im}(a^+_{k|n-2}(a^+_{j|n-2})^*)
+ 2i\hbox{Im}(a^0_{k|n-1}a^0_{j|n-1}).
\end{equation}
Conversely, if (\ref{nec-cond-sol-LS}) holds, then also the right hand sides
of (\ref{comm+0-indR1}) and (\ref{comm-cr1aR1}) are equal, hence the
system (\ref{comm+0-indR1}), (\ref{comm-cr1aR1}) is equivalent to the
single system (\ref{comm+0-indR1}).\\
Now notice that right hand side of (\ref{nec-cond-sol-LS}) is equal to
$$
X_{k}a^+_{j|n-1}-X_{j}a^+_{k|n-1}
+ 2i\hbox{Im}(a^+_{k|n-2}(a^+_{j|n-2})^*)
+ 2i\hbox{Im}(a^0_{k|n-1}a^0_{j|n-1})
$$
$$
=X_{k}(X_{j|n-1} - a^{0}_{j|n-1} - (a^+_{j|n-2})^*)
-X_{j}(X_{k|n-1} - a^{0}_{k|n-1} - (a^+_{k|n-2})^*)
$$
$$
+ 2i\hbox{Im}(a^+_{k|n-2}(a^+_{j|n-2})^*)
+ 2i\hbox{Im}(a^0_{k|n-1}a^0_{j|n-1})
$$
$$
=X_{k}X_{j|n-1} - X_{k}a^{0}_{j|n-1} - X_{k}(a^+_{j|n-2})^*
$$
$$
-X_{j}X_{k|n-1} + X_{j}a^{0}_{k|n-1} + X_{j}(a^+_{k|n-2})^*
$$
$$
+ 2i\hbox{Im}(a^+_{k|n-2}(a^+_{j|n-2})^*)
+ 2i\hbox{Im}(a^0_{k|n-1}a^0_{j|n-1})
$$
$$
=  X_{j}(a^+_{k|n-2})^* - X_{k}(a^+_{j|n-2})^*
+ X_{j}a^{0}_{k|n-1} - X_{k}a^{0}_{j|n-1}
+ 2i\hbox{Im}(a^+_{k|n-2}(a^+_{j|n-2})^*)
+ 2i\hbox{Im}(a^0_{k|n-1}a^0_{j|n-1}).
$$
With similar arguments, the left hand side of (\ref{nec-cond-sol-LS})
is equal to
$$
a^+_{j|n-1}a^0_{k|n-1}-a^+_{k|n-1}a^0_{j|n-1}
$$
$$
=(X_{j|n-1} - a^{0}_{j|n-1} - (a^+_{j|n-2})^*)a^0_{k|n-1}
-(X_{k|n-1} - a^{0}_{k|n-1} - (a^+_{k|n-2})^*)a^0_{j|n-1}
$$
$$
=X_{j|n-1}a^0_{k|n-1} - a^{0}_{j|n-1}a^0_{k|n-1} - (a^+_{j|n-2})^*a^0_{k|n-1}
-X_{k|n-1}a^0_{j|n-1} + a^{0}_{k|n-1}a^0_{j|n-1} + (a^+_{k|n-2})^*a^0_{j|n-1}
$$
$$
=X_{j|n-1}a^0_{k|n-1} -X_{k|n-1}a^0_{j|n-1}
+ a^{0}_{k|n-1}a^0_{j|n-1} - a^{0}_{j|n-1}a^0_{k|n-1}
+ (a^+_{k|n-2})^*a^0_{j|n-1} - (a^+_{j|n-2})^*a^0_{k|n-1}.
$$
Therefore the identity (\ref{nec-cond-sol-LS}) holds iff
$$
X_{j|n-1}a^0_{k|n-1} -X_{k|n-1}a^0_{j|n-1}
+ a^{0}_{k|n-1}a^0_{j|n-1} - a^{0}_{j|n-1}a^0_{k|n-1}
+ (a^+_{k|n-2})^*a^0_{j|n-1} - (a^+_{j|n-2})^*a^0_{k|n-1}
$$
$$
=  X_{j}(a^+_{k|n-2})^* - X_{k}(a^+_{j|n-2})^*
+ X_{j}a^{0}_{k|n-1} - X_{k}a^{0}_{j|n-1}
+ 2i\hbox{Im}(a^+_{k|n-2}(a^+_{j|n-2})^*)
+ 2i\hbox{Im}(a^0_{k|n-1}a^0_{j|n-1})
$$ \eject
$$
\iff
+ 2i\hbox{Im}(a^{0}_{k|n-1}a^0_{j|n-1})
+ (a^+_{k|n-2})^*a^0_{j|n-1} - (a^+_{j|n-2})^*a^0_{k|n-1}
$$
$$
=  X_{j}(a^+_{k|n-2})^* - X_{k}(a^+_{j|n-2})^*
+ 2i\hbox{Im}(a^+_{k|n-2}(a^+_{j|n-2})^*)
+ 2i\hbox{Im}(a^0_{k|n-1}a^0_{j|n-1}).
$$
Thus, using the quantum decomposition, the identity (\ref{nec-cond-sol-LS})
can be re--written in the form
$$
(a^+_{k|n-2})^*a^0_{j|n-1} - (a^+_{j|n-2})^*a^0_{k|n-1}
=  X_{j}(a^+_{k|n-2})^* - X_{k}(a^+_{j|n-2})^*
+ 2i\hbox{Im}(a^+_{k|n-2}(a^+_{j|n-2})^*)
$$
$$
=  X_{j|n-2}(a^+_{k|n-2})^* - X_{k|n-2}(a^+_{j|n-2})^*
+ 2i\hbox{Im}(a^+_{k|n-2}(a^+_{j|n-2})^*)
$$
$$
=(a^+_{j|n-2}+a^0_{j|n-2}+(a^+_{j|n-3})^*)(a^+_{k|n-2})^*
- (a^+_{k|n-2}+a^0_{k|n-2}+(a^+_{k|n-3})^*)(a^+_{j|n-2})^*
+ 2i\hbox{Im}(a^+_{k|n-2}(a^+_{j|n-2})^*)
$$
$$
=a^+_{j|n-2}(a^+_{k|n-2})^*+a^0_{j|n-2}(a^+_{k|n-2})^*
+(a^+_{j|n-3})^*(a^+_{k|n-2})^*
$$
$$
- a^+_{k|n-2}(a^+_{j|n-2})^* - a^0_{k|n-2}(a^+_{j|n-2})^*
- (a^+_{k|n-3})^*(a^+_{j|n-2})^*
+ 2i\hbox{Im}(a^+_{k|n-2}(a^+_{j|n-2})^*)
$$
$$
=a^+_{j|n-2}(a^+_{k|n-2})^*+a^0_{j|n-2}(a^+_{k|n-2})^*
- a^+_{k|n-2}(a^+_{j|n-2})^* - a^0_{k|n-2}(a^+_{j|n-2})^*
+ 2i\hbox{Im}(a^+_{k|n-2}(a^+_{j|n-2})^*)
$$
$$
=a^+_{j|n-2}(a^+_{k|n-2})^* - a^+_{k|n-2}(a^+_{j|n-2})^*
+a^0_{j|n-2}(a^+_{k|n-2})^* - a^0_{k|n-2}(a^+_{j|n-2})^*
+ 2i\hbox{Im}(a^+_{k|n-2}(a^+_{j|n-2})^*)
$$
$$
=2i\hbox{Im}(a^+_{j|n-2}(a^+_{k|n-2})^*)
+a^0_{j|n-2}(a^+_{k|n-2})^* - a^0_{k|n-2}(a^+_{j|n-2})^*
+ 2i\hbox{Im}(a^+_{k|n-2}(a^+_{j|n-2})^*)
$$
$$
= a^0_{j|n-2}(a^+_{k|n-2})^* - a^0_{k|n-2}(a^+_{j|n-2})^*
$$
%In conclusion, the identity (\ref{nec-cond-sol-LS}) is
or equivalently:
\begin{equation}\label{adj-eqR}
(a^+_{k|n-2})^*a^0_{j|n-1} - (a^+_{j|n-2})^*a^0_{k|n-1}
= a^0_{j|n-2}(a^+_{k|n-2})^* - a^0_{k|n-2}(a^+_{j|n-2})^* .
\end{equation}
Taking the adjoint of the identity
$$
(a^+_{k})^*a^0_{j} - (a^+_{j})^*a^0_{k}
= a^0_{j}(a^+_{k})^* - a^0_{k}(a^+_{j})^*
$$
one finds
$$
a^0_{j}a^+_{k} - a^0_{k}a^+_{j}
= a^+_{k}a^0_{j} - a^+_{j}a^0_{k}.
$$
Restricting to $\mathcal P_{n-1}$ one obtains
$$
a^0_{j|n}a^+_{k|n-1} - a^0_{k|n}a^+_{j|n-1}
= a^+_{k|n-1}a^0_{j|n-1} - a^+_{j|n-1}a^0_{k|n-1}
$$
which gives the adjoint of (\ref{adj-eqR}).
Since this is equivalent to (\ref{comm+0-indR1}), we conclude that
the identity (\ref{nec-cond-sol-LS}) holds if and only if (\ref{comm+0-indR1})
holds. This proves the statement.
\end{proof}

\begin{lemma}\label{lm-zero-sol}{\rm
The inductive system of equations (\ref{comm+0-ind}), (\ref{comm-cr1a})
in the unknowns $a^{0}_{j|n}$, always admit
the zero solution, given by the sequence
\begin{equation}\label{sol-a0(j|n)=0}
a^{0}_{j|n} =0 \: , \qquad \forall j\in D \ , \ \forall n\in\mathbb N.
\end{equation}
}\end{lemma}

\begin{proof}
If $a^{0}_{j}=0$, (\ref{comm+0-ind}), i.e. (\ref{comm+0-indR1}) is identically satisfied.
Therefore the result follows from Lemma \ref{lm-equivalence}.
\end{proof}

%%%%%%%%%%%%%%%%%%%%%%%%%%%%%%%%%%%%%%%%%%%%%%%%%%%%%%%%%%%%%%%%%%%%%%%%%%%%%%%%%%%%%%%%%%%%%%%%%
\section{  The reconstruction theorem}\label{rec-theo}
%%%%%%%%%%%%%%%%%%%%%%%%%%%%%%%%%%%%%%%%%%%%%%%%%%%%%%%%%%%%%%%%%%%%%%%%%%%%%%%%%%%%%%%%%%%%%%%%%

%%%%%%%%%%%%%%%%%%%%%%%%%%%%%%%%%%%%%%%%%%%%%%%%%%%%%%%%%%%%%%%%%%%%%%%%%%%%%
\subsection{$3$--diagonal decompositions of $\mathcal P$}\label{3-diag-dec-P}
%%%%%%%%%%%%%%%%%%%%%%%%%%%%%%%%%%%%%%%%%%%%%%%%%%%%%%%%%%%%%%%%%%%%%%%%%%%%%

The goal of the present section is to abstract, from a given orthogonal
gradation, a minimal set of characteristics that allow an inductive
reconstruction of this gradation.\\
For two pre--Hilbert spaces $\mathcal H$ and $\mathcal K$,
we denote $\mathcal L_a(\mathcal H,\mathcal K)$ the $*$--algebra
of all adjointable linear operators from $\mathcal H$ to
$\mathcal K$ (see Appendix \ref{App-Orth-proj-PHS}).

Recall that $(e_j)_{j\in D}$ is the canonical basis of $\mathbb{C}^d$ and
that we use the notation
$$
a^\varepsilon_{j|k}:=a^\varepsilon_{e_j|k}
 \: ,\quad j\in D \ , \ \varepsilon\in\{+,0,-\}.
$$

\begin{definition}\label{df-3d-dec}{\rm
For $n\in\mathbb N^*$, a {\bf $3$--diagonal decomposition of}
$\mathcal P_{n]}$ is defined by:\\
(i) a vector space direct sum decomposition of
$\mathcal P_{n]}$
\begin{equation}\label{orth-dec-Pk]}
\mathcal P_{k]}
= \sum^{\centerdot}_{h\in \{0, \cdots, k \}} \mathcal P_{h}
\: , \qquad \forall \  k\in\{0,1,\dots , n\},
\end{equation}
such that each $\mathcal{P}_h$ is monic of order $h$,

\noindent (ii) for each $k\in\{0,1,\dots , n\}$, a pre--scalar
product $\langle \ \cdot \ , \ \cdot \ \rangle_{k} $ on $\mathcal P_{k}$,
% Condition (ii) above and the fact that the sum
% (\ref{df-3d-dec}) is direct, imply that there exists a
such that, denoting $\langle \ \cdot \ , \ \cdot \ \rangle_{n]}$ the unique
scalar product on $\mathcal P_{n]}$ characterized by the conditions that
the vector space decompositions (\ref{df-3d-dec}) are orthogonal
for the restriction of $\langle \ \cdot \ , \ \cdot \ \rangle_{n]}$
on each $\mathcal P_{k]}$:
\begin{equation}\label{orth-dec-Pk]2}
\mathcal P_{k]}
= \bigoplus_{h\in\{0,\cdots,k\}} \mathcal P_{h}
\: , \qquad \forall \, k\in\{0,1,\dots , n\},
\end{equation}
and for all $k\in\{0,1,\dots , n\}$
\begin{equation}\label{restr-scpr-Pn]}
\langle \ \cdot \ , \ \cdot \ \rangle_{n]}\Big|_{\mathcal P_{k}}
=\langle \ \cdot \ , \ \cdot \ \rangle_{k}
\end{equation}
 the restrictions of the operators $X_{e_j}$ on
$\mathcal P_{n-1]}$ are symmetric:
\begin{equation}\label{Omega(u,v|n)}
\langle X_{e_j}\xi , \eta  \rangle_{n]}
=\langle \xi  , X_{e_j}\eta  \rangle_{n]}
\: , \qquad  \xi , \eta\in\mathcal P_{n-1]} \ , \ j\in D,
\end{equation}
(iii) two families of pre--Hilbert space linear maps
(see Appendix \ref{App-Orth-proj-PHS} for the notations)
\begin{equation}\label{a+v|k}
a^+_{e_j|k}\in
\mathcal L_a((\mathcal P_{k},\langle \ \cdot \ , \ \cdot \ \rangle_{k}),
(\mathcal P_{k+1},\langle \ \cdot \ , \ \cdot \ \rangle_{k+1}))
\: ,  \qquad k\in\{0,1,\dots , n-1\},
\end{equation}
\begin{equation}\label{a0v|k}
a^0_{e_j|k}\in
\mathcal L_a((\mathcal P_{k},\langle \ \cdot \ , \ \cdot \ \rangle_{k}),
(\mathcal P_{k},\langle \ \cdot \ , \ \cdot \ \rangle_{k}))
\: ,  \qquad k\in\{0,1,\dots , n-1\},
\end{equation}
$j\in D$, such that:\\

(iii.1) for all $k\in\{1,\dots , n-1\}$ and $j\in D$,
$a^0_{e_j|k}$ is self-adjoint in the pre-Hilbert space sense;\\

(iii.2) the following identity is satisfied:
\begin{equation}\label{df-a+v|k1}
X_{e_j|k} = a^+_{e_j|k} + a^{0}_{e_j|k} + a^-_{e_j|k}
\: , \quad k\in\{0,1,\dots , n-1\} \ , \ j\in D,
\end{equation}
with the convention that $a_{e_j|-1}^+=0 $, and
\begin{equation}\label{df-a-v|k}
a^-_{e_j|k} := (a^+_{e_j|k-1})^*:(\mathcal P_{k},
\langle \ \cdot \ , \ \cdot \ \rangle_{k}) \to (\mathcal P_{k-1},
\langle \ \cdot \ , \ \cdot \ \rangle_{k-1})
\: ,  \qquad k\in\{0,1,\dots , n-1\},
\end{equation}
where $(a^+_{e_j|k-1})^*$ denotes, when no confusion is possible,
the pre--Hilbert space adjoint of $a^+_{e_j|k-1}$.\\

(iii.3) The operators $a^\pm_{e_j|k}$, $a^0_{e_j|k}$ satisfy
the commutation relations (\ref{comm+j-k0j0k1c0}), (\ref{comm+j-k0j0k1c}),
(\ref{comm+0-indR1}).
}\end{definition}

\noindent{\bf Remark}.
\begin{enumerate}
\item[1)] In the following, if no confusion can arise,
we will simply say that
\begin{equation}\label{notat-3diag-dec}
\left\{
\Big(\mathcal P_k \ , \ \langle \ \cdot \ , \ \cdot \ \rangle_{k} \Big)^{n}_{k=0}
\ , \;\; \left(a^{+}_{ \cdot \ |k}\right)^{n-1}_{k=0} \ , \;\;
\left(a^{0}_{ \cdot \ |k}\right)^{n-1}_{k=0}\right\}
\end{equation}
is a $3$--diagonal decomposition of $\mathcal P_{n]}$.
\item[2)]
Note that a priori all the objects defining a $3$--diagonal
decomposition of $\mathcal P_{n]}$ may depend on $n\in\mathbb N$.
\end{enumerate}

\begin{definition}\label{df-3d-ext}{\rm
(i) A $3$--diagonal decomposition of $\mathcal P_{n+1]}$
$$
\left\{\Big(\mathcal P_k(n+1) \ , \
\langle \ \cdot \ , \ \cdot \ \rangle_{n+1,k} \Big)^{n+1}_{k=0} \ , \;\;
\left(a^{+}_{ \ \cdot \ |k}(n+1)\right)^{n}_{k=0} \ , \;\;
\left(a^{0}_{ \ \cdot \ |k}(n+1)\right)^{n}_{k=0}\right\}
$$
is called {\bf an extension of a $3$--diagonal decomposition} of $\mathcal{P}_{n]}$
$$
\left\{\Big(\mathcal P_k(n) \ , \
\langle \ \cdot \ , \ \cdot \ \rangle_{n,k} \Big)^{n}_{k=0} \ , \;\;
\left(a^{+}_{ \ \cdot \ |k}(n)\right)^{n-1}_{k=0} \ , \;\;
\left(a^{0}_{ \ \cdot \ |k}(n)\right)^{n-1}_{k=0}\right\}
$$
if, in obvious notations
$$
\mathcal P_k(n) \ = \ \mathcal P_k(n+1)
\: ,  \qquad \forall \, k\in\{0,\cdots,n\},
$$
$$
\langle \ \cdot \ , \ \cdot \ \rangle_{n+1]}\Big|_{\mathcal P_{n]}}
 \ = \ \langle \ \cdot \ , \ \cdot \ \rangle_{n]}
$$
$$
a^{0}_{ \ \cdot \ |k}(n+1)=a^{0}_{ \ \cdot \ |k}(n)
\: ,  \qquad \forall \, k\in\{0,\cdots,n\},
$$
$$
a^{+}_{ \ \cdot \ |k}(n+1)=a^{+}_{ \ \cdot \ |k}(n)
\: ,  \qquad \forall \, k\in\{0,\cdots,n-1\},
$$

(ii) A {\bf $3$--diagonal decomposition of $\mathcal P$} is a
sequence of $3$--diagonal decompositions
$$
D_n:= \left\{\Big(\mathcal P_k(n) \ , \ \langle \ \cdot \ , \
\cdot \ \rangle_{n,k} \Big)^{n}_{k=0} \ , \;\;
\left(a^{+}_{ \ \cdot \ |k}(n)\right)^{n-1}_{k=0} \ , \;\;
\left(a^{0}_{ \ \cdot \ |k}(n)\right)^{n-1}_{k=0}\right\}
\: ,  \qquad n\in{\mathbb N},
$$
such that, for each $n\in\mathbb N$,
$D_{n+1}$ is an extension of $D_n$. In this case one simply writes
\begin{equation}\label{proj-Dn}
\left\{\Big(\mathcal P_n \ , \ \langle \ \cdot \ , \ \cdot \ \rangle_{n} \Big) \ , \;\;
a^{+}_{ \ \cdot \ |n} \ , \;\; a^{0}_{ \ \cdot \ |n}\right\}_{n\in{\mathbb N}}.
\end{equation}
}\end{definition}

\noindent{\bf Remark}.
Any $3$--diagonal decomposition of $\mathcal P_{n]}$ induces,
by restriction, a $3$--diagonal decomposition of
$\mathcal P_{k]}$ for any $k\leq n$.\\
The following Theorem motivates the introduction of the notion of
$3$--diagonal decomposition given above.

\begin{theorem}\label{3d-dec-P1}{\rm
Every state $\varphi$ on ${\mathcal P}$ uniquely defines a $3$--diagonal
decomposition of ${\mathcal P}$. Conversely, given a $3$--diagonal
decomposition of ${\mathcal P}$, there exists a unique state $\varphi$ on
${\mathcal P}$ such that the $3$--diagonal decomposition of ${\mathcal P}$,
associated to $\varphi$ according to the first part of the theorem, is the
given one.
}\end{theorem}

\begin{proof}
If the pre--scalar product on ${\mathcal P}$ is induced
by a state $\varphi$ on ${\mathcal P}$, then by Lemma
\ref{st-ind-sc-pr} the operators of multiplication by
the coordinates are symmetric for this pre--scalar product
and the quantum decompositions of the random variables $X_j$ $(i\in D)$
constructed in section \ref{sec-symm-jac-rel} provide a
$3$--diagonal decomposition of ${\mathcal P}$. The uniqueness
of the quantum decomposition implies the uniqueness of the corresponding
$3$--diagonal decomposition of ${\mathcal P}$. \\

Conversely, let be given a $3$--diagonal decomposition
of ${\mathcal P}$
%of the form (\ref{proj-Dn})
and denote
$\langle \ \cdot \ , \ \cdot \ \rangle$ the pre--scalar product induced
by it on $\mathcal P$. Then, by condition (ii) of Definition
\ref{df-3d-dec} and condition (ii) of Definition \ref{df-3d-ext}, for
each $n\in\mathbb N$, the restriction of the operator $X_{e_j}$ ($j\in D$) on
$\mathcal P_{n-1]}$ is symmetric with respect to the restriction of
$\langle \ \cdot \ , \ \cdot \ \rangle$ on $\mathcal P_{n-1]}$. Since
$\bigcup_{k\in{\mathbb N}}\mathcal P_{k]}=\mathcal P$, the operators $X_{e_j}$
are $\langle \ \cdot \ , \ \cdot \ \rangle$--symmetric on $\mathcal P$.\\
Lemma \ref{st-ind-sc-pr} then implies that the pre--scalar
product on ${\mathcal P}$ is induced by some state
$\varphi$ on ${\mathcal P}$ and this concludes the proof.
\end{proof}

%%%%%%%%%%%%%%%%%%%%%%%%%%%%%%%%%%%%%%%%%%%%%%%%%%%%%%%%%%%%%%%%%%%%%%%%%%%%%%%%%%%%%%%%
\subsection{Structure of $3$--diagonal decompositions of $\mathcal P$}
%%%%%%%%%%%%%%%%%%%%%%%%%%%%%%%%%%%%%%%%%%%%%%%%%%%%%%%%%%%%%%%%%%%%%%%%%%%%%%%%%%%%%%%%

Having established the equivalence between $3$--diagonal decomposition
of $\mathcal P$ and orthogonal gradations induced by states on $\mathcal P$,
our next goal is to produce a characterization of the
$3$--diagonal decomposition of $\mathcal P$. As a first step towards this
goal in this section we discuss the following problem:\\
{\it given a $3$--diagonal decomposition of $\mathcal P_{n]}$, classify
all its possible extensions in the sense of Definition \ref{df-3d-ext}}.\\

\begin{lemma}\label{class-3d-ext}{\rm
Let, for $n\in\mathbb N^*$,
\begin{equation}\label{3d-decPn]2}
\left\{
\Big(\mathcal P_k \ , \ \langle \ \cdot \ , \ \cdot \ \rangle_{k} \Big)^{n}_{k=0}
\ , \;\; \left(a^{+}_{ \cdot \ |k}\right)^{n-1}_{k=0} \ , \;\;
\left(a^{0}_{ \cdot \ |k}\right)^{n-1}_{k=0}\right\}
\end{equation}
be a $3$--diagonal decomposition of $\mathcal P_{n]}$ (see
(\ref{notat-3diag-dec})). Any $3$--diagonal extension of (\ref{3d-decPn]2})
defines a pair
\begin{equation}\label{pair-ext-3d-dec}
%\Big(\langle \ \cdot  \ ,  \ \cdot\rangle_{n+1}  \ , \
%a^0_{ \ \cdot \  |n+1} \Big)\equiv
\Big(\tilde\Omega_{n+1}  \ , \ a^0_{ \ \cdot \  |n} \Big)
\end{equation}
with the following properties:
\begin{description}
\item[(i)]
%denoting $\langle \ \cdot  \ ,  \ \cdot\rangle_{n+1}$
%the pre--scalar product on $\mathcal P_{n+1}$ induced by $\tilde\Omega_{n+1}$, \\
$a^0_{ \ \cdot \  |n}$ is a linear map
\begin{equation}\label{df-a0v|n-a}
a^0_{ \ \cdot \  |n} \ : \ v\in \mathbb C^d \ \longmapsto \
a^0_{v|n} \in\mathcal L_a(\mathcal P_{n} \ , \ \langle \ \cdot \ , \ \cdot \
\rangle_{n})
\end{equation}
such that:\\
-- for all $v\in \mathbb R^d $, $a^0_{v|n}$ is a
self--adjoint operator on the pre-Hilbert space
$(\mathcal P_{n} \ , \ \langle \ \cdot \ , \ \cdot \ \rangle_{n})$;\\

\item[(ii)] For each $n\in\mathbb N$ a
$\mathcal L_a((\mathcal P_{n},\langle \ \cdot  \ , \ \cdot\rangle_{n})$--valued
positive definite kernel on $\mathbb C^d$, denoted $\tilde\Omega_{n}$, mapping
real vectors onto real vectors and such that $\tilde\Omega_{0}\equiv 1$,
$\tilde\Omega_{1}$ is arbitrary and, for $n>1$ the pair
$$
((\tilde\Omega_{n})_{n\in\mathbb N} \ , \ (a^0_{e_j|n})_{j\in D} )
$$
is a solution of the joint system of inductive equations
(\ref{comm+j-k0j0k1c0}), (\ref{comm+j-k0j0k1c}),
(\ref{comm+0-ind}) and (\ref{comm-cr1a})
where the $a^+_{j|n} $ are defined by (\ref{Symm-Jac-rel-ind}) and
the $(a^+_{j|n})^*$ by the right hand side of (\ref{df-adj-au+}).
\end{description}
Conversely
%, given $n\in\mathbb N^*$ and the $3$--diagonal decomposition
% (\ref{3d-decPn]2}) of $\mathcal P_{n]}$,
any pair of the form
(\ref{pair-ext-3d-dec}), satisfying conditions (i) and (ii) above,
defines a $3$--diagonal decomposition of $\mathcal P$.
}\end{lemma}

\begin{proof}
Definition \ref{df-3d-dec} implies that any $3$--diagonal
decompositions of $\mathcal P_{n+1]}$ extending the given
one determines a pair (\ref{pair-ext-3d-dec}) with a selfadjoint
operator $a^0_{\ \cdot \ |n}$ and with positive definite kernel
$(\tilde\Omega_n(e_j,e_h))$ defined by
$$
 \tilde\Omega_{n+1}(e_j,e_h) :=a^-_{e_j|n+1}a^+_{e_h|n}
\in\mathcal L_a\left(\mathcal P_{n} \ ,  \
\langle \ \cdot \ , \ , \ \cdot \ \rangle_{n}\right)
\: ,  \qquad j,h\in D \ , \ n\in\mathbb N.
$$
Lemma \ref{struc-PD-kern} implies that the $\tilde\Omega_n$ satisfy
conditions (\ref{comm+j-k0j0k1c0}), (\ref{comm+j-k0j0k1c});
Lemma \ref{lm-1st-ind-rel-a0j} implies that the $a^0_{j|n+1}$ satisfy
condition (\ref{comm+0-ind}); Lemma \ref{th2-comm-crs} implies
that the $a^0_{j|n+1}$ satisfy condition (\ref{comm-cr1a}).
Therefore properties (i) and (ii) above are satisfied.\\

Conversely, given $n\in\mathbb N^*$, the $3$--diagonal decomposition
(\ref{3d-decPn]2}) of $\mathcal P_{n]}$,  and a pair of the form
(\ref{pair-ext-3d-dec}), satisfying conditions (i) and (ii) above,
define for each $j\in D$ the linear maps
\begin{equation}\label{1df-a+v|n1}
a^+_{e_j|n} :\mathcal P_{n} \longrightarrow \mathcal P_{n+1]}
\end{equation}
by the condition
\begin{equation}\label{1df-a+v|n-conv}
a^+_{e_j|n} \
:= \ X_j\Big|_{\mathcal P_{n}} \ - \ a^{0}_{e_j|n} \ - \ (a^+_{e_j|n-1})^*
\end{equation}
%which depends only on the given $3$--diagonal decompositions
%of $\mathcal P_{n]}$.
%Then, by Lemma \ref{Pk-pol-dk}, the same is true for
and let $\mathcal P_{n+1}$ be the vector space
constructed in Lemma \ref{Pk-pol-dk} with the choices
$$
A^{0}_{e_j|n+1}:=a^{0}_{e_j|n+1} \quad
\hbox{and} \quad A^{-}_{e_j|n+1}:=a^{-}_{e_j|n+1} =(a^+_{e_j|n})^* .
$$
That $\mathcal{P}_{n+1}$ is a monic sub--space of order $n+1$, of $\mathcal{P}_{n+1]}$
follows from Lemma \ref{Pk-pol-dk}. This proves that condition (i)
of Definition \ref{df-3d-dec} is satisfied.\\

Let $\langle \ \cdot  \ ,  \ \cdot\rangle_{n+1}$ be the pre--scalar product
on $\mathcal P_{n+1}$, induced by the positive definite kernel
$(\tilde\Omega_{n+1}(e_j,e_h))$ through the identity:
$$
\sum_{j,h\in D}\langle a^+_{e_j|n}\xi_j,a^+_{e_h|n}\eta_h\rangle_{n+1}
:=\sum_{j,h\in D}\langle \xi_j,\tilde\Omega_{n+1}(e_j,e_h)\eta_h\rangle_{n}
\: , \quad \xi_j,\eta_h\in\mathcal P_{n},
$$
and let $\xi\in\mathcal{P}_{n}$ be a zero norm vector.
Then for each $j\in D$ and $\xi\in\mathcal P_{n}$
$$
\|a^+_{e_j|n}\xi\|^2_{n+1}
=\langle a^+_{e_j|n}\xi ,a^+_{e_jh|n}\xi \rangle_{n+1}
=\langle \xi,\tilde\Omega_{n+1}(e_j,e_h)\xi\rangle_{n}=0.
$$
Thus the operators $a^+_{e_j|n}$ are pre--Hilbert space operators in the sense of
Definition \ref{term-pre-Hilb-sp}.\\
Let us prove that for each $j\in\{1,\cdots ,d\}$,
the restriction on $\mathcal P_{n]}$ of the multiplication operator
by $X_{e_j}$ is symmetric, i.e. that for each
$\xi , \eta\in\mathcal P_{n]}$, one has
\begin{equation}\label{mult-symm1}
\langle X_{e_j}\xi  , \eta  \rangle_{n+1]}
=\langle \xi  , X_{e_j}\eta  \rangle_{n+1]}.
\end{equation}

From (\ref{1df-a+v|n-conv})  we know that
\begin{equation}\label{df-a+|n-ind-2}
a^+_{e_j|n}+a^0_{e_j|n} + (a^+_{e_j|n-1})^* = X_{e_j|n}
\end{equation}
where the restriction is meant in the sense of right multiplication by
the projection onto $\mathcal P_{n}$, so that
both sides are zero outside $\mathcal P_{n}$.
This implies in particular that, for each $k\leq n$
$$
X_{e_j|k}: \mathcal P_{k}\to
\mathcal P_{k+1}\oplus\mathcal P_{k}\oplus\mathcal P_{k-1}.
$$
If both $\xi , \eta\in\mathcal P_{n-1]}$, then the identity (\ref{mult-symm1})
is reduced to the identity
$$
\langle X_{e_j}\xi  , \eta  \rangle_{n]}
=\langle \xi  , X_{e_j}\eta  \rangle_{n]}
$$
which holds because (\ref{class-3d-ext}) is a $3$--diagonal decomposition
of $\mathcal P_{n]}$.\\
Therefore it is sufficient to consider the case in which
$\xi  ,\eta \in\mathcal P_{n}\oplus\mathcal P_{n-1}$.\\
By symmetry the problem is reduced to the two cases:
$$
\eta\in \mathcal P_{n-1}  \qquad \textrm{and} \qquad   \xi\in \mathcal P_{n}
$$
$$
\eta\in \mathcal P_{n}    \qquad \textrm{and} \qquad  \xi\in \mathcal P_{n}
$$
Case 1 : $\eta\in \mathcal P_{n-1}  \ ; \  \xi\in \mathcal P_{n}$.\\
Using the mutual orthogonality of the spaces $\mathcal P_{k}$
for $k\leq n+1$, one finds:
$$
\langle X_{e_j}\xi  , \eta  \rangle_{n+1]}
=\langle \xi  , X_{e_j}\eta  \rangle_{n+1]} \ \Leftrightarrow
$$
$$
\Leftrightarrow
\ \langle (a^+_{e_j|n}+a^0_{e_j|n} + (a^+_{e_j|n-1})^*)\xi  ,
\eta  \rangle_{n+1]}
=\langle \xi,(a^+_{e_j|n-1}+a^0_{e_j|n-1}+(a^+_{e_j|n-2})^*)\eta\rangle_{n+1]}
$$
$$
\Leftrightarrow
\ \langle a^+_{e_j|n}\xi, \eta  \rangle_{n+1]}
+\langle a^0_{e_j|n}\xi , \eta  \rangle_{n+1]}
+ \langle (a^+_{e_j|n-1})^*\xi  , \eta  \rangle_{n+1]}=
$$
$$
=\langle \xi,a^+_{e_j|n-1}\eta\rangle_{n+1]}
+\langle \xi,a^0_{e_j|n-1}\eta\rangle_{n+1]}
+\langle \xi,(a^+_{e_j|n-2})^*\eta\rangle_{n+1]}
$$
$$
\Leftrightarrow
\langle (a^+_{e_j|n-1})^*\xi  , \eta  \rangle_{n-1}
=\langle \xi,a^+_{e_j|n-1}\eta\rangle_{n}
$$
that is identically satisfied because (\ref{class-3d-ext}) is a
$3$--diagonal decomposition of $\mathcal P_{n]}$.\\

Case 2 : $\eta\in \mathcal P_{n}  \ ; \  \xi\in \mathcal P_{n}$
$$
\langle X_{e_j}\xi  , \eta  \rangle_{n+1]}
=\langle \xi  , X_{e_j}\eta  \rangle_{n+1]} \ \Leftrightarrow
$$
$$
\Leftrightarrow \
\langle (a^+_{e_j|n}+a^0_{e_j|n} + (a^+_{e_j|n-1})^*)\xi, \eta\rangle_{n+1]}
=\langle \xi,(a^+_{e_j|n}+a^0_{e_j|n}+(a^+_{e_j|n-1})^*)\eta\rangle_{n+1]}
$$

$$
\Leftrightarrow
\ \langle a^+_{e_j|n}\xi, \eta  \rangle_{n+1]}
+\langle a^0_{e_j|n}\xi , \eta  \rangle_{n+1]}
+ \langle (a^+_{e_j|n-1})^*\xi  , \eta  \rangle_{n+1]}=
$$
$$
=\langle \xi,a^+_{e_j|n}\eta\rangle_{n+1]}
+\langle \xi,a^0_{e_j|n}\eta\rangle_{n+1]}
+\langle \xi,(a^+_{e_j|n-1})^*\eta\rangle_{n+1]}
$$

$$
\Leftrightarrow
\langle a^0_{e_j|n}\xi , \eta  \rangle_{n}
=\langle \xi,a^0_{e_j|n}\eta\rangle_{n}
$$
that is identically satisfied because, by assumption,
$a^0_{e_j|n}$ is self--adjoint for the
$\langle \ \cdot \ , \ \cdot \ \rangle_{n}$--scalar product.
Therefore the restriction on $\mathcal P_{n]}$,
of the multiplication operator by $X_{e_j}$ is symmetric, i.e.
condition (ii) of Definition (\ref{df-3d-dec}) is satisfied.\\

The linear maps $(a^0_{ \ \cdot \  |n+1})$ are self--adjoint
for the pre--scalar product $\langle \ \cdot  \ ,  \ \cdot\rangle_{n+1}$
because of assumption (i). This is equivalent to condition (iii.1) of
Definition (\ref{df-3d-dec}).\\

(\ref{1df-a+v|n-conv}) implies that condition (iii.2) of Definition
(\ref{df-3d-dec}) is satisfied.\\

Finally condition (iii.3) of the same Definition is satisfied because of
Condition (ii).

In conclusion: for any choice of the pair (\ref{pair-ext-3d-dec}),
satisfying conditions (i) and (ii) above, the triple
$$
\left\{ \Big(
\mathcal P_k \ , \ \langle \ \cdot \ , \ \cdot \ \rangle_{k} \Big)^{n+1}_{k=0}
\ , \;\; \left(a^{+}_{ \cdot \ |k}\right)^{n-1}_{k=0} \ , \;\;
\left(a^{0}_{ \cdot \ |k}\right)^{n-1}_{k=0}\right\}
$$
is a $3$--diagonal decomposition of $\mathcal P_{n+1]}$ extending
the given one (\ref{3d-decPn]2}). This concludes the proof.
\end{proof}

%%%%%%%%%%%%%%%%%%%%%%%%%%%%%%%%%%%%%%%%%%%%%%%%%%%%%%%%%%%%%%%%%%%%%%%%%%%%%%%%%%%%%%%%%%%%%%%%%%%%%%%%%%%%%
\section{ The $d$--dimensional Favard Lemma}\label{multi-dim-Favard-Lm}
%%%%%%%%%%%%%%%%%%%%%%%%%%%%%%%%%%%%%%%%%%%%%%%%%%%%%%%%%%%%%%%%%%%%%%%%%%%%%%%%%%%%%%%%%%%%%%%%%%%%%%%%%%%%%

We have seen that the $d$--dimensional analogue of the principal
Jacobi sequence $(\omega_{n})$ of a state on $\mathcal P$ is the
sequence of positive definite kernels $(\tilde\Omega_{n})$
and the $d$--dimensional analogue of the secondary
Jacobi sequence $(\alpha_{n})$ is the set of
sequences of self--adjoint operators $(a^0_{j|n})$ ($j\in D$)
(in this section we often use the notation
$a^{\varepsilon}_{j|n}=a^{\varepsilon}_{e_j|n}$ for
$\varepsilon\in \{+,0,-\}$, $j\in D$).
In the $1$--dimensional case, the $(\omega_{n})$ have the only
constraint $\omega_n=0 \: \Longrightarrow  \: \omega_{n+k}=0$,
while the $(\alpha_{n})$ are arbitrary real numbers.
In the $d$--dimensional case we have seen in section \ref{Impl-comm-rels}
that the commutation relations impose constraints both on the
$(\tilde\Omega_{n})$ and on the $(a^0_{j|n})$ ($j\in D$).
Fortunately, when written in inductive form, these constraints,
turn out to be {\bf linear}. In order to obtain the inductive formulation
of the $d$--dimensional extension of Favard Lemma
we introduce the following definition,
that expresses in a precise way the basic idea of these inductive
relations, namely that: given the $a^+_{j|n-1}$ ($j\in D$) and the scalar
product on $\mathcal P_n$ one chooses the $a^0_{j|n}$, compatibly
with the linear constraints and this uniquely defines the $a^+_{j|n}$.
The choice of the $a^+_{j|n}$ uniquely defines the vector space
$\mathcal P_{n+1}$ and, since the imaginary part of the kernel
$(\tilde\Omega_{n+1})$ is uniquely determined by the constraints,
its real part is only subjected to the constraints of positive--definitness
and of mapping real vectors of $\mathcal P_{n}$ into real vectors.

\begin{definition}\label{recurs-3diag-struc}{\rm
Given a linear basis $(e_j)$ of $\mathbb R^d$,
a {\bf recursive $3$--diagonal structure on $\mathcal P$ with respect
to the basis $(e_j)$} is defined by the following procedure. \\
(i) Define the vector sub--space with real structure
$$
\mathcal{P}_{0}:=\mathbb C\cdot \Phi_0
\equiv (\mathbb R\oplus i\mathbb R)\cdot \Phi_0
=:\mathcal{P}_{R,0}\dot + i\mathcal{P}_{R,0}
$$
and the scalar product $\langle \ \cdot \  , \ \cdot \ \rangle_{0}$ on it
uniquely determined by the condition $\|\Phi_0\|:=1$.\\
(ii) For each $j\in D$, choose arbitarily a self--adjoint operator
$$
a^{0}_{j|0}:(\mathcal{P}_{0}, \langle \ \cdot \ , \ \cdot \ \rangle_{0})
\to (\mathcal{P}_{0}, \langle \ \cdot \ , \ \cdot \ \rangle_{0})
$$
i.e. a real number $\tilde a^{0}_{j|0}\in \mathbb R$ characterized by
$a^{0}_{j|0}\Phi_0=:\tilde a^{0}_{j|0}\Phi_0$. \\
(iii) For each $j\in D$, define the linear operator
$a^{+}_{j|0}:\mathcal{P}_{0}\to \mathcal{P}_{1}$ by
$$
a^{+}_{j|0}:=X_j - a^{0}_{j|0}
$$
and the vector spaces
\begin{equation}\label{df-P1R}
\mathcal{P}_{R,1}:=
\mathbb R\hbox{--lin--span of} \
\{a^+_{j|0}\mathcal{P}_{R,0} : j\in D \}
=\mathbb R\hbox{--lin--span of} \ \{X_j-a^{0}_{j|0}\Phi_0 : j\in D \}
\end{equation}
\begin{equation}\label{df-P1}
\mathcal{P}_{1}:=\mathbb C\hbox{--lin--span of} \
\{a^+_{j|0}\mathcal{P}_{R,0} : j\in D \}
=\mathcal{P}_{R,1}+i\mathcal{P}_{R,1}.
\end{equation}

(iv) Choose arbitrarily an
$\mathcal{L}_{a}\left(
(\mathcal{P}_{0}, \langle \ \cdot \ , \ \cdot \ \rangle_{0})\right)$--valued
positive definite kernel $\tilde\Omega_{R,1}$ on
$\mathbb C^d\equiv \mathbb R^d\oplus i\mathbb R^d$ such that, for any
$u,v\in \mathbb R^d$, $\tilde\Omega_{R,1}(u,v)$ maps real vectors of
$\mathcal{P}_{0}$ into real vectors.
Equivalently, choose arbitrarily a pre--scalar product on
$\mathcal{P}_{1}$, real--valued on $\mathcal{P}_{R,1}$.
Define $\tilde\Omega_{1}:=\tilde\Omega_{R,1}$ and the pre--scalar product
$\langle \ \cdot \ , \ \cdot \ \rangle_{1}$ on $\mathcal{P}_{1}$, by
$$
\langle a^{+}_{j|0}\Phi_0,a^{+}_{k|0}\Phi_0\rangle_{1}
:=\langle \Phi_0,\tilde\Omega_{1}(e_{j},e_{k})\Phi_0\rangle_{0}.
$$
(v) Having defined, for $1\le k\le n$, the pre--Hilbert space
$(\mathcal{P}_{k}, \langle \ \cdot \ , \ \cdot \ \rangle_{k})$
with real structure $\mathcal{P}_{k}=\mathcal{P}_{R,k}+i\mathcal{P}_{R,k}$,
the linear operators
$a^+_{j|n-1}:\mathcal P_{n-1}\to \mathcal P_{n}$,
and the self--adjoint operators
$a^0_{j|n-1}: (\mathcal{P}_{R,n-1}, \langle \ \cdot \ , \ \cdot \ \rangle_{n-1})
\to(\mathcal{P}_{R,n-1}, \langle \ \cdot \ , \ \cdot \ \rangle_{n-1})$ ($j\in D$),
choose arbitrarily a self--adjoint solution
$a^0_{j|n}: (\mathcal{P}_{R,n}, \langle \ \cdot \ , \ \cdot \ \rangle_{n})
\to(\mathcal{P}_{R,n}, \langle \ \cdot \ , \ \cdot \ \rangle_{n})$ ($j\in D$)
of the linear system
\begin{equation}\label{comm+0-indR}
a^0_{k|n}a^+_{j|n-1}-a^0_{j|n}a^+_{k|n-1}
= a^+_{j|n-1}a^0_{k|n-1}-a^+_{k|n-1}a^0_{j|n-1}
\end{equation}
for all $j,k\in D$ such that $j<k$ (such solutions exist by Lemma
\ref{lm-zero-sol}).  \\
(vi) Define the linear operator
$$
a^+_{j|n}:=X_{j|n} - a^{0}_{j|n} - (a^+_{j|n-1})^* \ : \
\mathcal P_{n}\to \mathcal P
$$
and the vector spaces
\begin{equation}\label{df-PRn+1}
\mathcal P_{R,n+1}:=\mathbb R\hbox{--lin--span of} \
\{a^+_{j|n}\mathcal{P}_{R,n} \ : \ j\in D \}
\end{equation}
\begin{equation}\label{df-Pn+1}
\mathcal P_{n+1}:=\mathbb C\hbox{--lin--span of} \
\{a^+_{j|n}\mathcal{P}_{R,n} \ : \ j\in D \}
=\mathcal{P}_{R,n+1}+i\mathcal{P}_{R,n+1}.
\end{equation}
(vii) Choose {\bf arbitrarily} an
$\mathcal{L}_{a}\left(
(\mathcal{P}_{n}, \langle \ \cdot \ , \ \cdot \ \rangle_{n})\right)$--valued
positive definite kernel $\tilde\Omega_{R,n+1}$ on
$\mathbb C^d\equiv \mathbb R^d\oplus i\mathbb R^d$ such that, for any
$u,v\in \mathbb R^d$, $\tilde\Omega_{R,n+1}(u,v)$ maps
$\mathcal{P}_{R,n}$ into itself and define $\tilde\Omega_{n+1}(e_j,e_k)$ by:
\begin{equation}\label{comm+j-k0j0k1cR}
\tilde\Omega_{n+1}(e_j,e_k)
:=\tilde\Omega_{R,n+1}(e_j,e_k)
+ \hbox{Im}(a^+_{k|n-1}(a^+_{j|n-1})^*) + \hbox{Im}((a^0_{k|n})a^0_{j|n})
\end{equation}
and the pre--scalar product
$\langle \ \cdot \ , \ \cdot \ \rangle_{n+1}$ on $\mathcal{P}_{n+1}$ by:
$$
\langle a^{+}_{j|0}\xi_n,a^{+}_{k|0}\eta_n\rangle_{n+1}
:=\langle \xi_n,\tilde\Omega_{n+1}(e_{j},e_{k})\eta_n\rangle_{n}
\: , \qquad \xi_n,\eta_n\in \mathcal{P}_{n}.
$$
(viii) Having defined the pre--Hilbert space
$\left(\mathcal P_{n+1},
\langle \ \cdot \ , \ \cdot \ \rangle_{n+1}\right)$ with real structure
$\mathcal{P}_{n+1}=\mathcal{P}_{R,n+1}+i\mathcal{P}_{R,n+1}$, the
$a^+_{j|n}$ and the $a^{0}_{j|n}$ ($j\in D$), one can iterate
the construction of item (v) above.
}\end{definition}

\begin{theorem}\label{3diag-equiv-new-Favrd}{\rm
($d$--dimensional Favard Lemma )
For any linear basis $(e_j)$ of $\mathbb R^d$,
there is a one--to--one correspondence between states
on $\mathcal{P}$ and recursive $3$--diagonal structures
on $\mathcal P$ with respect to the basis $(e_j)$.
}\end{theorem}

{\bf Remark} Since, by adding an arbitrary symplectic kernel to a
positive definite kernel, the result is still a positive definite kernel,
equation (\ref{comm+j-k0j0k1cR}) does not
introduce additional constraints on the $\tilde\Omega_{R,n+1}$.

\begin{proof}
{\bf Necessity}. Let $\varphi$ be a state on $\mathcal P$.
From the results in section \ref{Impl-comm-rels}, it follows that the
$3$--diagonal decomposition of ${\mathcal P}$  associated to the pair
$({\mathcal P} \ , \ \varphi)$ according to Theorem \ref{3d-dec-P1},
defines a recursive $3$--diagonal structure
on $\mathcal P$ with respect to the basis $(e_j)$. \\
{\bf Sufficiency}. Given a recursive $3$--diagonal structure
on $\mathcal P$ with respect to the basis $(e_j)$, denote
\begin{equation}\label{3ddec-ass-R3ddec}
\left\{ \left(\mathcal P_k \ , \ \langle \ \cdot \ ,
 \ \cdot \ \rangle_{k} \right)
\ , \;\; \left(a^{+}_{ \cdot \ |k}\right) \ , \;\;
\left(a^{0}_{ \cdot \ |k}\right)\right\}_{k\in\mathbb N}
\end{equation}
the $3$--diagonal decomposition of $\mathcal P$ associtated to it, i.e.
$$
a^+_{ \ \cdot \ }:=\sum_{n\in\mathbb N}a^+_{ \ \cdot \ |n}
\: , \qquad a^0_{ \ \cdot \ }:=\sum_{n\in\mathbb N}a^0_{ \ \cdot \ |n}
\: , \qquad a^-_{ \ \cdot \ }:=(a^+_{ \ \cdot \ })^* .
$$
Then the commutation relations (\ref{comm+0-ind}) are satisfied
because of (\ref{comm+0-indR}) and Lemma \ref{lm-1st-ind-rel-a0j}.\\
The commutation relations (\ref{comm-cr1a}) are satisfied
because of (\ref{comm-cr1aR1}), Lemma (\ref{lm-equivalence})
and Lemma \ref{th2-comm-crs}.\\
The commutation relations (\ref{comm+j-k0j0k1c}) are satisfied
because of (\ref{comm+j-k0j0k1cR}), Lemma \ref{struc-PD-kern}
and the remark following it.
Since the $3$--diagonal decomposition (\ref{3ddec-ass-R3ddec})
is uniquely defined by the recursive $3$--diagonal structure,
it follows that the same is true for the unique state on $\mathcal P$
defined by it according to Theorem \ref{3d-dec-P1}. This proves the statement.
\end{proof}

\section{Appendix: Orthogonal projectors and adjoints on pre--Hilbert spaces}
\label{App-Orth-proj-PHS}

\begin{definition}\label{term-pre-Hilb-sp}{\rm
We use the following terminology:
\begin{enumerate}
\item [(1)] A pre--scalar product on a vector space $V$ is
a non identically zero positive definite Hermitean form on $V$.
\item [(2)] A scalar product on a vector space $V$ is
a pre--scalar product on $V$ whose only zero--norm vector is $0$.
\item [(3)] A pre--Hilbert space is a vector space equipped
with a pre--scalar product.
\item [(4)] A Hilbert space is a vector space equipped with a
scalar product and complete with respect to the
topology induced by it.
\end{enumerate}
}\end{definition}   

Let $(\mathcal H,\langle \ \cdot \ , \ \cdot  \  \rangle_{\mathcal H})$
and $(\mathcal K,\langle \ \cdot  \ , \ \cdot  \  \rangle_{\mathcal K})$,
be two pre--Hilbert spaces. In the following, when no confusion is possible,
we will omit the label from the two scalar products.\\
$\mathcal L_a(\mathcal H,\mathcal K)$
denotes the space of all adjointable linear operators from $\mathcal H$ to
$\mathcal K$.\\ By definition, $A\in\mathcal L_a(\mathcal H,\mathcal K)$
if and only if:\\
-- $A$ is a linear operator operator everywhere defined on $\mathcal H$;\\
-- $A$ maps zero norm vectors of $\mathcal H$ into zero norm vectors of
$\mathcal K$;\\
-- there exists a linear operator $A^*:\mathcal K\to\mathcal H$ such that
$$
\langle Ah,k\rangle_{\mathcal K} = \langle h,A^*k\rangle_{\mathcal H}
\: , \qquad \forall h\in\mathcal H \ , \ \forall k\in\mathcal K .
$$
In this case $A^*$ is called an adjoint of $A$ and
$A$ is called self--adjoint if $A=A^*$ for some choice of $A^*$.\\
{\bf Remark}. If $A^*$ and $A^+$ are two adjoints of $A$,
then the range of the operator $A^+-A^*$ is contained in the zero--norm
sub--space because
$$
\langle (A^+-A^\ast)k,h\rangle = \langle k ,Ah\rangle - \langle k,Ah\rangle =0
\: , \qquad \forall h\in\mathcal H \ , \ \forall k\in\mathcal K .
$$

\begin{lemma}\label{proj-PHS}{\rm
Let ${\mathcal H}$ be a
%${\mathbb C}$--
pre--Hilbert space and
let $\mathcal K$ be a finite dimensional sub--space of $\mathcal H$.
%linearly generated by a finite sub--set $(f_j)_{j\in F}$ of the $f_j$.
Denote ${\mathcal K}_0$ the sub--space of the zero--norm vectors in
${\mathcal H}$.\\
Then, for any choice of:\\
-- a linear complement ${\mathcal H}_1$ of ${\mathcal K}_0$ in ${\mathcal H}$,\\
-- a linear complement ${\mathcal K}_1$ of $\mathcal K\cap {\mathcal K}_0$
in ${\mathcal K}$,\\
-- a linear complement ${\mathcal K}_{0,1}$ of
$\mathcal K\cap {\mathcal K}_0$ in ${\mathcal K}_0$, \\
there exists a self--adjoint projection
$P_{\mathcal K}$ from ${\mathcal H}$ onto ${\mathcal K}$.\\
If $\mathcal H_1'$, $\mathcal K_1'$, $\mathcal K_{0,1}'$
are other choices of the above mentioned complements then,
denoting $P_{\mathcal K}'$ the orthogonal projection onto
$\mathcal K$, defined by the first part of the theorem, the range of
$P_{\mathcal K}-P_{\mathcal K}'$ is contained in the zero norm sub--space
of ${\mathcal H}$.\\
If $\mathcal K_1$ has an orthogonal basis $B$ and $\mathcal H$ a linear
basis $C$ such that the scalar products of elements of $B$ with elements of $C$
are real, then the projection $P_{\mathcal K}$ can be chosen so that the real
linear span of $C$ is mapped onto the real linear span of $B$.
}\end{lemma}

\begin{proof}{\rm
The assumptions imply the decompositions
\begin{equation}\label{K=K0+K1}
{\mathcal H}
= ({\mathcal K}_0\cap\mathcal K) \ \oplus \
{\mathcal K}_{0,1} \ \oplus \ {\mathcal H}_1
\: , \qquad {\mathcal K}
= ({\mathcal K}_0\cap\mathcal K) \ \oplus \ {\mathcal K}_1 ,
\end{equation}
that are orthogonal because ${\mathcal K}_0$ is
orthogonal to all vectors. Let $(k_j)_{j\in D_1}$, $D_1$ a finite set,
be a linear basis of ${\mathcal K}_1$.
Since by assumption $\mathcal K_0 \ \cap \ \mathcal K_1=\{0\}$,
the ortho--normalization procedure can be applied to the set
$(k_j)_{j\in D_1}$ leading to an ortho--normal basis $(e_j)_{j\in D_1}$
of ${\mathcal K}_1$.
Any vector $h\in\mathcal H$ can be written in a unique way as
$$
h=h_1+k_0+k_{0,1} \quad \hbox{with} \quad  h_1\in\mathcal H_1
\ , \ k_0\in(\mathcal K_0 \ \cap \ \mathcal K)
\ , \ k_{0,1}\in\mathcal K_{0,1}.
$$
The linear map defined by
\begin{equation}\label{df-PK(h)}
P_{\mathcal K}(h) := \sum_{j\in D_1}\langle e_j, h\rangle e_j + k_0
= \sum_{j\in D_1}\langle e_j, h_1\rangle e_j + k_0
\end{equation}
is clearly a pre--Hilbert space projection from ${\mathcal H}$
onto ${\mathcal K}$ and
$$
\langle P_{\mathcal K}(h),h'\rangle
= \sum_{j\in D_1}\overline{\langle e_j, h_1\rangle} \langle e_j ,h'\rangle + \langle k_{0},h'\rangle
= \sum_{j\in D_1}\langle e_j, h_1\rangle \langle e_j,h'\rangle
=\langle h,P_{\mathcal K}(h')\rangle.
$$
Therefore $P_{\mathcal K}$ is self--adjoint. By inspection from (\ref{df-PK(h)})
it follows that $P_{\mathcal K}$ does not depend on the choice of the
ortho--normal basis $(e_j)$ of ${\mathcal K}_1$.\\
Let $\mathcal H_1'$, $\mathcal K_1'$, $\mathcal K_{0,1}'$ be as
in the statement of theorem.
Then any vector $h\in\mathcal H$ has two decompositions
$$
h=h_1+k_0+k_{0,1}= h_1'+k_0'+k_{0,1}' \: , \quad
h_1\in\mathcal H \ , \ k_1'\in\mathcal K_1' \ , \
k_0,k_0'\in\mathcal K_0 \ , \ k_{0,1},k_{0,1}'\in\mathcal K_{0,1}
$$
hence $h_1$ differs from $h_1'$ by a zero norm vector.
A similar argument shows that, for each $e_j$ in the basis $(e_j)$
of $\mathcal K_1$, there exists $k_{0,j}\in\mathcal K_0\cap \mathcal K$
and $e'_j\in\mathcal K_1'$ such that
$$
e_j:=e'_j + k_{0,j}.
$$
The $e'_j$ are clearly ortho--normal and they are a basis of $\mathcal K_1'$
because it has the same (finite) dimension as $\mathcal K_1$.
Moreover one has
$$
P_{\mathcal K}(h) - k_0 = \sum_{j\in D_1}\langle e_j, h\rangle e_j
=\sum_{j\in D_1}\langle e'_j+k_{0,j}, h_1\rangle (e'_j + k_{0,j})
$$
$$
= \sum_{j\in D_1}\langle e'_j, h\rangle \langle e'_j,h'\rangle e'_j
+ \sum_{j\in D_1}\langle e'_j, h\rangle k_{0,j} + k_{0}'
=P_{\mathcal K}'(h)
+ \left(-k_{0}'+ \sum_{j\in D_1}\langle e'_j, h\rangle k_{0,j}\right)
$$
which shows that the range of $P_{\mathcal K}-P_{\mathcal K}'$
is contained in $\mathcal K_0$.\\
The last statement of the theorem is clear.
}\end{proof}

\begin{definition}\label{df-PHS-proj}{\rm
The projection $\: P_{{\mathcal K}_0}$, defined in Lemma \ref{proj-PHS},
will be called the orthogonal projection onto $\mathcal K_0$ associated
to the decompositions (\ref{K=K0+K1}).
}\end{definition}

\begin{lemma}\label{(adj-prHS)}{\rm
Let $(\mathcal H , \langle \cdot , \cdot \rangle_{\mathcal H} )$,
$(\mathcal K , \langle \cdot , \cdot \rangle_{\mathcal K} )$
be pre--Hilbert spaces,
suppose that $\mathcal H$ is finite dimensional and let
$$
A : (\mathcal H , \langle \cdot , \cdot \rangle_{\mathcal H} ) \rightarrow
(\mathcal K , \langle \cdot , \cdot \rangle_{\mathcal K} )
$$
be a linear operator.
Denote $\mathcal H_{0}$ (resp. $\mathcal K_{0}$) the zero norm
sub--space of $\mathcal H$ (resp. $\mathcal K$).
Suppose that $A$ has the property that $A\mathcal H_{0}\subseteq \mathcal K_{0}$.
Then for any vector space complement $\mathcal H_{1}$ of $\mathcal H_{0}$
there exists an adjoint of $A$.
}\end{lemma}

\begin{proof}
For any $k\in\mathcal K$, the map
$$
h\in\mathcal H \mapsto \langle A h , k \rangle_{\mathcal K}
= \langle k , A h \rangle_{\mathcal K}
$$
is a linear functional on $\mathcal H$, therefore it defines
an element of $\mathcal H^\ast$, the algebraic dual of $\mathcal H$,
denoted $\hat A k$ and characterized by the property
\begin{equation}\label{df-tilde-A}
\hat A k (h) = \langle k , A h \rangle_{\mathcal K},
\end{equation}
By assumption
$$
(\hat A k) (\mathcal H_{0}) =\{0\},
$$
therefore $\hat A k$ induces a linear functional on
$\mathcal H \backslash \mathcal H_{0}$.\\
Let $\mathcal H_1$ be a vector space complement of $\mathcal H_0$ so that
$\mathcal H = \mathcal H_1 \dot + \mathcal H_0$.
Then $\mathcal H_1$ is isomorphic to $\mathcal H \backslash \mathcal H_{0}$
as a linear space and, through this isomorphism, it becomes an Hilbert
space, because $\mathcal H$ is finite dimensional.
Therefore any linear functional $f_1$ on $\mathcal H_1$ is determined by an
element of $h_1\in\mathcal H_1$ through the identity
$$
f_1 (h_2) = \langle h_1 , h_2 \rangle_{\mathcal H_1}
\: , \qquad h_2 \in \mathcal H_1 .
$$
For any $k\in \mathcal K$, define $A^\ast k$ the element of $\mathcal H_1$
corresponding to $\hat A k$ in $\mathcal H_1$. Then
\begin{equation}\label{(df-A*eta-calK)}
\langle A^\ast k , h \rangle = \hat A k (h)
= \langle k , A h \rangle_{\mathcal K}.
\end{equation}
Thus the linear operator $k \in \mathcal K \mapsto A^\ast k\in\mathcal H_1$
is an adjoint of $A$. This proves the statement.
\end{proof}

\begin{definition}{\rm
In the notations and assumptions of Lemma \ref{(adj-prHS)},
the pre-Hilbert space linear operator $A^\ast$ defined in
Lemma \ref{(adj-prHS)} is called the adjoint of $A$ with respect
to decomposition $\mathcal H = \mathcal H_1 \dot + \mathcal H_0$.
}\end{definition}

%%%%%%%%%%%%%%%%%%%%%%%%%%%%%%%%%%%%%%%%%%%%%%%%%%%%%%%%%%%%%%%%%%%%%%%%%%%%%%%%%%%%%%%%%%%%%%%%%%%%%%
\section{Appendix: Interacting Fock spaces}\label{App-IFS}
%%%%%%%%%%%%%%%%%%%%%%%%%%%%%%%%%%%%%%%%%%%%%%%%%%%%%%%%%%%%%%%%%%%%%%%%%%%%%%%%%%%%%%%%%%%%%%%%%%%%%%

All constructions used in the following, like direct sums and tensor
products, are algebraic.
For any pair of pre--Hilbert spaces $(H,\langle \ \cdot \ , \ \cdot \ \rangle_H)$,
$(K,\langle \ \cdot \ , \ \cdot \ \rangle_K)$, \\
$\mathcal{L}_a((H,\langle \ \cdot \ , \ \cdot \ \rangle_H),
(K,\langle \ \cdot \ , \ \cdot \ \rangle_K))$, or simply when no confusion
is possible $\mathcal{L}_a(H,K)$, denotes the space of all
adjointable pre--Hilbert space maps $A\  \colon \ H\rightarrow K$,
such that there exists a linear map $A^*\  \colon \ K\rightarrow H$ satisfying
$$
\langle f,Ag\rangle_K=\langle A^*f,g\rangle_H \: , \qquad \forall g\in H
 \ , \ \forall f\in K.
$$
If $H=K$, then $\mathcal L_a(K ,\langle \ \cdot \ , \ \cdot \  \rangle_{K})$
has a natural structure of $*$--algebra
and we simply write $\mathcal L_a(K)$.
% For any gradation of $H$
% $$
% H=\bigoplus_{n\in\mathbb N} H_n
% $$
% and for each $m\in\mathbb Z$, we define the space
% $$
% \mathcal{L}_{a,m}(\mathcal{H})
% =\bigl\{A\in\mathcal{L}_a(\mathcal{H})\  \colon \ AH_n\subset %
% H_{n+m}~(n\in\mathbb N)\bigr\}
% $$

\begin{definition}\label{df-IFS}{\rm
Let $V$ be a vector space.
% $H$ be a pre--Hilbert space.
An {\bf interacting Fock space on $V$} is a pair:
\begin{equation}\label{df-IFS0}
\left\{ (H_{n}, \langle \ \cdot \ ,\ \cdot \ \rangle_{n})_{n\in\mathbb N}), a^+ \right\}
\end{equation}
such that:\\
-- $\bigl(H_{n}, \langle \ \cdot \ ,\ \cdot \ \rangle_{n}\bigr)_{n\in\mathbb N}$
is a sequence of pre--Hilbert spaces and
$\bigl(H_{0}, \langle \ \cdot \ ,\ \cdot \ \rangle_{0}\bigr)$ is uniquely determined 
by the conditions:
\begin{equation}\label{df-H0}
H_0=:\mathbb C\cdot \Phi_0  \qquad ;  \qquad \|\Phi_{0}\|= 1
\end{equation}
$\Phi_0$ is called the {\bf vacuum or Fock vector};\\
-- denoting $\langle \ \cdot \ ,\ \cdot \ \rangle$ the unique pre--Hilbert
space scalar product on the vector space direct sum of the
family $\bigl(H_n\bigr)_{n\in\mathbb N}$ which makes this direct sum
\begin{equation}\label{df-IFS1}
H:=\bigoplus_{n\in\mathbb N} (H_{n}, \langle \ \cdot \ ,\ \cdot \ \rangle_{n})
\end{equation}
an orthogonal sum, the linear operator
$$
a^+\  \colon \ V\rightarrow
\mathcal{L}_{a}
\left((H_{n}, \langle \ \cdot \ ,\ \cdot \ \rangle_{n})_{n\in\mathbb N}\right)
$$
satisfies the following conditions:
\begin{equation}\label{prop-creat}
H_{n+1}
= \hbox{lin-span} \ \left\{a^+(V)H_n \right\}
\: , \qquad \forall n\in\mathbb N .
\end{equation}
For each $v\in V$, one fixes a choice of adjoint of $a^+(v)$ denoted by
$a^-(v)$ (or simply $a_v$) so that
\begin{equation}\label{Fk-prescr}
a(v)\Phi_0=0 \ \hbox{Fock prescription} \: , \qquad \forall v\in V.
\end{equation}
The operators $a^+(v)$ ($f\in V$) are called {\bf creators} and
their adjoints $a(v)$ -- {\bf annihilators}.
The spaces $\bigl(H_n\bigr)_{n\in\mathbb N}$ are called the {\bf $n$--particle
spaces},
% (or sectors)
if $n=0$ one speaks of the vacuum space.
If
$$
\left\{ (H_{1,n}, \langle \ \cdot \ ,\ \cdot \ \rangle_{1,n})_{n\in\mathbb N}),
a_1^* \right\}
$$
is another IFS on a vector space $V_1$, a {\bf morphism}
from
$\left\{ (H_{n},
\langle \ \cdot \ ,\ \cdot \ \rangle_{n})_{n\in\mathbb N}),a^+ \right\}$
to
 $\left\{ (H_{1,n}, \langle \ \cdot \ ,\ \cdot \ \rangle_{1,n})_{n\in\mathbb N}),
a_1^* \right\}$ is a linear map $U_1:V\to V_1$ and a
linear isometry
$$
U:\bigoplus_{n\in\mathbb N} (H_{n}, \langle \ \cdot \ ,\ \cdot \ \rangle_{n})
\to \bigoplus_{n\in\mathbb N} (H_{1,n},
\langle \ \cdot \ ,\ \cdot \ \rangle_{1,n})
$$
such that $U$ is gradation preserving and
%(resp. ) (resp. unitaries)
$$
Ua^+_{v}U^{*}=a^*_{1,U_1v}\: , \qquad \forall v\in V.
$$
The pair $(U_1,U)$ is an {\bf isomorphism} if $U_1$ is invertible
and $U$ is onto up to vectors of norm zero.
}\end{definition}

{\bf Remark}. For any $f\in V$, since the annihilator $a(f)$ is defined
as the adjoint of the creator $a^+(f)$, its action
on $\Phi_0$ is not defined. However, {\bf since the gradation (\ref{df-IFS1})
is $1$--sided}, the only possible way to define it compatibly with the
condition that $a(f)=(a^+(f))^*$, is to define
\begin{equation}\label{df-H-1}
H_{-1}:= \{0\}
\end{equation}
or equivalently to introduce the Fock prescription (\ref{Fk-prescr}).

{\bf Remark}. Recall that, by definition of pre--Hilbert space linear map,
each $a^+(f)$ ($f\in V$) maps zero--norm vectors into zero--norm vectors.
The existence of
%the
a pre--Hilbert space adjoint of $a^+(v)$ with respect to the pre--scalar product
(\ref{df-Omegan-scal-prod-stnd}), which by definition must be defined
on the whole space $H^{\otimes (n+1)}$, is equivalent to the condition
that for any $\xi_{n+1}\in H^{\otimes (n+1)}$ the map
$$
\eta_{n}\in
(H^{\otimes n}, \langle \ \cdot \ , \ \cdot \ \rangle_{n})\mapsto
\langle \xi_{n+1},a^+_v\eta_{n}\rangle_{n+1}
$$
can be extended to a continuous linear functional on the domain of $a^+(v)$,
which by definition is the whole algebraic tensor product $H^{\otimes n}$.
In the case of Hilbert spaces this happens if and only if there are constants
$c_{\xi_{n+1},v}$ such that
\begin{equation}\label{ineq-exst-adj}
\left|\langle \xi_{n+1} , v\otimes \eta_{n}\rangle_{n+1}\right|
\leq c_{\xi_{n+1},v} \left\| \eta_{n}\right \|_{n}
\end{equation}
but in the infinite dimensional case the condition that the whole
algebraic tensor product $H^{\otimes n}$ is in the domain of
the adjoint, is not automatically guaranteed.

%%%%%%%%%%%%%%%%%%%%%%%%%%%%%%%%%%%%%%%%%%%%%%%%%%%%%%%%%%%%%%%%%%%%%%%%%%%%%%%%%%%%%%%%%%%%%%%
\subsection{Example: The full Fock space}\label{Ex-FFS}
%%%%%%%%%%%%%%%%%%%%%%%%%%%%%%%%%%%%%%%%%%%%%%%%%%%%%%%%%%%%%%%%%%%%%%%%%%%%%%%%%%%%%%%%%%%%%%%

The {\bf full Fock space} $\mathcal{F}(V)$ on a pre--Hilbert space
$(V,\langle \ \cdot \ , \ \cdot \ \rangle_V)$ is
obtained by setting $H_n=V^{\otimes n}$ equipped with natural inner product
given by the $n$--fold tensor product:
\begin{equation}\label{df-otimes-tens-prod}
\langle f_n\otimes\cdots\otimes f_1, g_n\otimes\cdots\otimes g_1
\rangle_{\otimes n}
:=\langle f_n,g_n\rangle_V \langle f_{n-1},g_{n-1}\rangle_V \cdots
 \langle f_1,g_1\rangle_V
\end{equation}
$f_n,\dots, f_1, g_n,\dots, g_1\in V$.
Creators on the full Fock space are denoted by $\ell^*(f)$
($f\in V$) and their action on each $H_n$ is defined by setting
\begin{equation}\label{df-ell*(f)}
\ell^*(f)f_n\otimes\ldots\otimes f_1:=f\otimes f_n\otimes\ldots\otimes f_1
\end{equation}
$$
\ell^*_f\Phi_0:=\ell^*(f)\Phi_0=f
$$
The adjoint of $\ell(f)$, with respect to the pre--scalar product
(\ref{df-otimes-tens-prod}), is:
$$
\ell(f)f_n\otimes\ldots\otimes f_1
=\langle f,f_n\rangle f_{n-1}\otimes\ldots\otimes f_1
$$
$$
\ell(f)\Phi_0=0.
$$

%%%%%%%%%%%%%%%%%%%%%%%%%%%%%%%%%%%%%%%%%%%%%%%%%%%%%%%%%%%%%%%%%%%%%%%%%%%%%%%%%%%%%%%%%%%%%%%%%%%
\subsection{ The tensor representation of an IFS}\label{tens-repr-IFS}
%%%%%%%%%%%%%%%%%%%%%%%%%%%%%%%%%%%%%%%%%%%%%%%%%%%%%%%%%%%%%%%%%%%%%%%%%%%%%%%%%%%%%%%%%%%%%%%%%%%

\begin{lemma}\label{lm-df-IFS2}{\rm
Every IFS
\begin{equation}\label{df-IFS2}
\left\{ (H_{n}, \langle \ \cdot \ ,\ \cdot \ \rangle_{n})_{n\in\mathbb N}),
a^+ \right\}
\end{equation}
on a vector space $V$ is isomorphic, in the sense of Definition \ref{df-IFS},
to an IFS of the form
\begin{equation}\label{tens-repr-IFS-V}
\left\{ \left(V^{\otimes n} \ ,  \
\langle \ \cdot \ ,  \ \cdot \ \rangle_{\otimes, n}\right)
 \ , \ \ell^* \right\}
\end{equation}
where the pre--scalar products
$\langle \ \cdot \ ,  \ \cdot \ \rangle_{\otimes, n}$ are given by
\begin{equation}\label{tens-repr-IFS-V-sc-pr}
\langle u_n\otimes\dots\otimes u_1,v_n\otimes \dots \otimes v_1\rangle_{\otimes,n}
:=\langle a^+(u_n)\cdots a^+(u_1)\Phi_0,
a^+(v_n)\cdots a^+(v_1)\Phi_0\rangle_{n}
\end{equation}
($u_n,v_n, \dots , u_1, v_1\in V$) and the operator $\ell^*$ is defined,
in the notation (\ref{df-ell*(f)}), by
\begin{equation}\label{T-1a*(v)T=a*otimes(v)}
T^{-1}a^+(v)T=\ell^*(v)
\: , \qquad \forall v\in V.
\end{equation}
}\end{lemma}

\begin{proof}
By the universal property of the tensor product, for each $n\in\mathbb N$,
the map
\begin{equation}\label{H(n+1)=HtensHn}
v\otimes h_n\in V\otimes H_n \to a^+(v)h_n\in H_{n+1}
\end{equation}
has a unique linear extension denoted $T_{n,n+1}\ : \ V\otimes H_n \to H_{n+1}$.\\
% The map $T_{n,n+1}$ is onto because of (\ref{prop-creat}) and
% (\ref{tens-repr-IFS-V-sc-pr}).
One easily verifies that the left hand side of (\ref{prop-creat}) is a vector
space.  \\
% Iterating the identity (\ref{prop-creat}) one sees that,
% for all $n\in\mathbb N$
% $$
% H_n = \ \hbox{linear span of} \
% \{a^+(v_n)\cdots a^+(v_1)\Phi_0 \ : \ v_n,\ldots ,v_1\in V \}
% $$
Iterating the maps (\ref{H(n+1)=HtensHn}), one sees that the
linear extensions of the maps
\begin{equation}\label{tens-iso-gen1}
T_n:v_n\otimes \dots \otimes v_1\in V^{\otimes n} \to
a^+(v_n)\cdots a^+(v_1)\Phi_0\in H_{n}
\end{equation}
($n\in\mathbb N$)
are well defined and define a graded vector space homomorphism
\begin{equation}\label{tens-iso-gen2}
T:=\bigoplus_{n} T_n  \ : \ \hbox{Tens}(V)
=\bigoplus_{n\in\mathbb N} V^{\otimes n} \to \bigoplus_{n\in\mathbb N} H_{n}
\end{equation}
which, by construction, satisfies (\ref{T-1a*(v)T=a*otimes(v)}).\\
Defining the pre--scalar products
$\langle \ \cdot \ ,  \ \cdot \ \rangle_{\otimes, n}$ by
(\ref{tens-repr-IFS-V-sc-pr}), the maps $T_n$ become pre--Hilbert space
unitary isomorphisms, hence $T$ an IFS isomorphism.
This defines the IFS (\ref{tens-repr-IFS-V}).
\end{proof}

\begin{definition}\label{df-tens-repr-IFS-V}{\rm
The isomorphic realization (\ref{tens-repr-IFS-V}), of the IFS
on $V$ given by (\ref{df-IFS0}), is called the
{\bf tensor representation} of the IFS (\ref{df-IFS0}).
}\end{definition}

%%%%%%%%%%%%%%%%%%%%%%%%%%%%%%%%%%%%%%%%%%%%%%%%%%%%%%%%%%%%%%%%%%%%%%%%%%%%%%%%%%%%%%%%%%%%%%%%%%%%%%%%%%%%%%%%
\subsection{Standard Interacting Fock spaces}\label{Stand-ifs}
%%%%%%%%%%%%%%%%%%%%%%%%%%%%%%%%%%%%%%%%%%%%%%%%%%%%%%%%%%%%%%%%%%%%%%%%%%%%%%%%%%%%%%%%%%%%%%%%%%%%%%%%%%%%%%%%

\begin{definition}\label{df-stnd-IFS}{\rm
An IFS
$\left\{ (H_{n}, \langle \ \cdot \ ,\ \cdot \ \rangle_{n})_{n\in\mathbb N}),
a^+ \right\}$ {\bf on a pre--Hilbert space}
$(H, \langle \ \cdot \ ,\ \cdot \ \rangle_{H})$ is called {\bf standard}
if, in its tensor representation (\ref{tens-repr-IFS-V}) (with $V=H$),
the pre--scalar products have the form
\begin{equation}\label{df-Omegan-scal-prod-stnd}
\langle \ \cdot \ ,\ \cdot \ \rangle_{\otimes ,n}
=\langle \ \cdot \ ,\Omega_{n} \ \cdot \ \rangle_{H^{\otimes n}}
\end{equation}
where, for $f_j, g_j\in H$ ($j=1,\dots , n$)
\begin{equation}\label{otimes-tens-prodH}
\langle f_{n}\otimes\cdots\otimes f_1,
g_{n}\otimes\cdots\otimes g_1\rangle_{H^{\otimes n}}
:=\langle f_n,g_n\rangle_H \langle f_{n-1},g_{n-1}\rangle_H \cdots
 \langle f_1,g_1\rangle_H
\end{equation}
is the natural scalar product on $H^{\otimes n}$ and
$$
\Omega_{n} \ : \ H^{\otimes n} \  \to \ H^{\otimes n}
$$
is a positive linear operator.
}\end{definition}

\textbf{Remark}. If $H$ is finite dimensional, then every IFS on $H$ is standard.

%%%%%%%%%%%%%%%%%%%%%%%%%%%%%%%%%%%%%%%%%%%%%%%%%%%%%%%%%%%%%%%%%%%%%%%%%%%%%%%%%%%%%%%%%%%%%%%%%%%%%%%%%%%%%
\subsection{Interacting Fock space and positive definite operator--valued kernels}\label{ifs-PD-op-kern}
%%%%%%%%%%%%%%%%%%%%%%%%%%%%%%%%%%%%%%%%%%%%%%%%%%%%%%%%%%%%%%%%%%%%%%%%%%%%%%%%%%%%%%%%%%%%%%%%%%%%%%%%%%%%%%

The existence of the creation and annihilation operators poses some
restrictions on the sequence of scalar products defining an IFS. To describe
this restrictions we introduce the following definition.

\begin{definition}\label{scal-prod-H(n+1)a}{\rm
Let $S$ be a set and $B$ a $*$--algebra. A map
$\Omega: S \times S \to B$ is called a
{\bf $B$--valued positive definite kernel on $S$} if,
for any finite sub--set $F\subseteq S$ and any map $b:F\to B$, one has
$$
\sum_{s,t\in F} b_s^* \Omega_{s,t} b_t\geq 0
$$
$\Omega$ is called {\bf linear} if $S$ is a vector space and the map
$(s,t)\in S \times S \ \mapsto \ \Omega_{s,t}\in B$ is sesqui--linear. If
$$
B:=\mathcal{L}_{a}((H, \langle \ \cdot \ ,\ \cdot \ \rangle))
$$
is the $*$--algebra of adjointable operators on a pre--Hilbert space
$(H, \langle \ \cdot \ ,\ \cdot \ \rangle)$ we simply speak of a
 {\bf positive definite linear kernel on $S$ based on}
$(H, \langle \ \cdot \ ,\ \cdot \ \rangle)$
}\end{definition}

{\bf Remark}. Any $B$--valued positive definite kernel on $S$ defines
a linear kernel on the free vector space $V_S$ generated by $S$. Conversely, if
$V$ is a vector space a $B$--valued positive definite linear kernel on $V$
is uniquely determined by its values on a Hamel basis of $V$.\\

{\bf Remark}. From now on we restrict our attention to the case of interest
for the present paper, namely that in which all IFS are based on finite
dimensional vector spaces. \\
For a discussion of the general case we refer to the paper
\cite{[AcSk98]} where the notion of
positive definite kernel with values in a $*$--algebra was introduced.

\begin{lemma}\label{PDkern-sc-pr}{\rm
Let be given:\\
-- a finite dimensional pre--Hilbert space
$(\mathcal{K},\langle\cdot ,\cdot\rangle_{\mathcal{K}})$;\\
-- two finite dimensional vector spaces $W,V$;\\
-- an
$\mathcal{L}_a (\mathcal{K} , \langle \cdot , \cdot \rangle_{\mathcal{K}})$--valued
PD Kernel $\tilde{\Omega}$ on $V$; \\
-- a linear map
$a^+ : V \rightarrow \mathcal{L}_a ((\mathcal{K}
, \langle \cdot , \cdot \rangle_{\mathcal{K}}) , W)$ such that
$$
\hbox{lin--span}(a_{V}^+ \mathcal{K}) = W.
$$
Then there exists a unique pre--scalar product
$\langle \cdot , \cdot \rangle_{W}$ on $W$ such that
\begin{equation}\label{df-W-scal-prod}
\langle a_{u}^+ \xi , a_{v}^+ \eta \rangle_W
= \langle \xi , \tilde{\Omega} (u,v) \eta \rangle_{\mathcal{K}} \: ,
\qquad \forall u,v \in V, \xi, \eta \in \mathcal{K}.
\end{equation}
Moreover the adjoint of $a_{u}^+$, denoted
$(a_{u}^+)^* :(W,\langle \cdot , \cdot \rangle_{W})
\to (\mathcal{K},\langle\cdot ,\cdot\rangle_{\mathcal{K}})$ satisfies
\begin{equation}\label{Omega(uv)=(au+)*av+}
\tilde{\Omega} (u , v) =(a_{u}^+ )^*a_{v}^+ .
\end{equation}
In particular, the action of $(a_{u}^+)^\ast$ on $W$ is given, up to
addition of vectors of zero norm, by the identity
\begin{equation}\label{df-adj-au+}
 (a_{u}^+)^\ast \Phi
= \sum_{j \in D} \tilde{\Omega} (u, e_j) \xi_j\: ,
\qquad \Phi=\sum_{j \in D} a_{e_j}^+ \xi_j .
\end{equation}
}\end{lemma}

\begin{proof}
Let ${\mathcal K}_0\subseteq {\mathcal K}$ and ${\mathcal H}_1\subseteq W$ be
sub--spaces as in Lemma \ref{proj-PHS} with ${\mathcal H}$ replaced by $W$.
Let $e \equiv (e_j)_{j \in D}$ be a linear basis of $V$ and
$(P_{\mathcal{K},k})_{k \in D_{\mathcal{K}}}$
an orthonormal basis of $\mathcal{K}$.
The set
$$
\{ a_{e_j}^+P_{\mathcal{K},k} : j \in D , k \in D_{\mathcal{K}} \}
$$
is a system of generators of $W$.
Therefore there exist sets
$$
D_0 \subseteq D\: ,  \qquad D_{0\mathcal{K}} \subseteq D_{\mathcal{K}},
$$
such that the set
\begin{equation}\label{(150423-7)}
\{ a_{e_j}^+ \Phi_k \ : \ j \in D_0 \ , \ k \in D_{0,\mathcal{K}} \}
\end{equation}
is a linerar basis of $W$.
Define $\forall j, j' \in D_0$ , $\forall k, k' \in D_{0,\mathcal{K}}$
\begin{equation}\label{(150423-8)}
\langle a_{e_j}^+ \Phi_k , a_{j'}^+ \Phi_{k'} \rangle_W
:= \langle \Phi_k , \tilde{\Omega} (e_j , e_{j'}) \Phi_{k'} \rangle_{\mathcal{K}}.
\end{equation}
Then there exists a unique pre--scalar product
$\langle \cdot , \cdot \rangle_{W}$ on $W$ such that its restriction on the
linear basis $(\ref{(150423-7)})$ is given by $(\ref{(150423-8)})$.
By sesqui--linearity $\langle \cdot , \cdot \rangle_{W}$ satisfies
(\ref{df-W-scal-prod}).\\
We know from Lemma (\ref{(adj-prHS)}) that the map $a_{u}^+$ is
adjointable and is a pre--Hilbert space operator. Moreover any adjoint
of $a_{n}^+$ satisfies
$$
(a_{n}^+)^\ast \Phi = \sum_{j \in F} (a_{n}^+)^\ast a_{e_j}^+ (\xi_{j})
\: , \qquad \Phi=\sum_{j \in F}a_{e_j}^+ (\xi_{j}).
$$
By definition of $\tilde\Omega$ this implies that, for any
$\Psi,\Phi\in\mathcal K$ and any $u,v\in V$, one has
$$
\langle \Psi,\tilde{\Omega} (u , v)\Phi\rangle_W
=\langle a_{u}^+ \Psi, a_{v}^+ \Phi\rangle_W
=\langle \Psi, (a_{u}^+ )^*a_{v}^+ \Phi\rangle_W .
$$
This implies that the identity (\ref{Omega(uv)=(au+)*av+})
is satisfied up to addition of a zero norm vector.
But we know that any vector $\Phi \in a_{V}^+\mathcal K$ has the form
$\Phi = \sum_{j \in D} a_{e_j}^+ \xi_j$ for some vectors $\xi_j\in\mathcal K$.
Therefore up to addition of a zero norm vector
$$
 (a_{n}^+)^\ast \Phi
= \sum_{j \in D} (a_{u}^+ )^*a_{e_j}^+\xi_j
= \sum_{j \in D} \tilde\Omega (u, e_j) \xi_j
$$
and this proves (\ref{df-adj-au+}).
\end{proof}

{\bf Remark}.
Every IFS on a vector space $V$
\begin{equation}\label{df-IFS3}
\left\{ (H_{n}, \langle \ \cdot \ ,\ \cdot \ \rangle_{n})_{n\in\mathbb N}),
a^+ \right\}
\end{equation}
defines a sequence $(\tilde\Omega_{n})$ with
the follwing properties:
\begin{equation}\label{df-Omega0IFS}
\tilde\Omega_{0}\equiv 1
\end{equation}
is the constant kernel equal to $1$ on the Hilbert space
\begin{equation}\label{df-H0IFS}
(H_{0}, \langle \ \cdot \ ,\ \cdot \ \rangle_{0})
:=(\mathbb C, \langle z ,w \rangle_{0}:=\bar z w  \ (z ,w\in\mathbb C)).
\end{equation}
For $n\in\mathbb N$,
$\tilde\Omega_{n+1}$ is the $\mathcal{L}_{a}((H_{n},
\langle \ \cdot \ ,\ \cdot \ \rangle_{n}))$--valued
{\bf linear kernel} on $V$ defined by
\begin{equation}\label{n-IFS-PD-kern}
\tilde\Omega_{n+1}(u,v):=a(u)a^+(v)\Big |_{H_n}
\in \mathcal{L}_{a}((H_{n}, \langle \ \cdot \ ,\ \cdot \ \rangle_{n}))
\: , \quad u,v\in V.
\end{equation}
Because of (\ref{prop-creat}), the positive definite
kernel $\tilde\Omega_{n+1}$ uniquely determines the
pre--scalar product
$\langle \ \cdot \ ,\ \cdot \ \rangle_{n+1}$ through the identity
% a sequence
% $(\tilde\Omega_{n+1})_{n\in\mathbb N}$ such that,
\begin{equation}\label{scal-prod-H(n+1)a2}
\langle a^+(u)h_n, a^+(v)h'_n \rangle_{n+1}
=\langle h_n, a(u)a^+(v)h'_n \rangle_{n}
=:\langle h_n, \tilde\Omega_{n+1}(u,v)h'_n \rangle_{n}
\end{equation}
$(u,v\in V \ , \ h_n,h'_n\in H_{n})$.\\

{\bf Remark}. The converse of this statement is most conveniently
formulated using the tensor representation of the IFS (\ref{df-IFS3})
and its proof is based on the following result.

\begin{lemma}\label{scal-prod-VtensHn-a}{\rm
Let $V$ be a finite dimensional vector space.
%and $(H_{n}, \langle \ \cdot \ ,\ \cdot \ \rangle_{n})$
%a pre--Hilbert space.

(i) Any pair of pre--scalar products
$\langle \ \cdot \ ,\ \cdot \ \rangle_{V^{\otimes (n+1)}}$,
$\langle \ \cdot \ ,\ \cdot \ \rangle_{V^{\otimes n}}$,
on $V^{\otimes (n+1)}$, $V^{\otimes n}$ respectively,
defines, through the prescription
\begin{equation}\label{scal-prod-VtensHn-PDK}
\langle u\otimes h_n, v\otimes h'_n \rangle_{V^{\otimes (n+1)}}
=\langle h_n, \tilde\Omega^{\otimes}_{n+1}(u,v)h'_n
\rangle_{V^{\otimes n}}
%\quad ;\quad
\end{equation}
($u,v\in V \ , \ h_n,h'_n\in V^{\otimes n}$) an
$\mathcal{L}_{a}((V^{\otimes n},
\langle \ \cdot \ , \ \cdot \ \rangle_{V^{\otimes n}}))$--valued
PD kernel $\tilde\Omega^{\otimes}_{n+1}$ on $V$ such that
\begin{equation}\label{Omega(n+1,u,v)=l(n+1,n,u)ln*(v)}
\tilde\Omega^{\otimes}_{n+1}(u,v) = \ell_{n+1}(u)\ell_n^*(v)
\end{equation}
where $\ell^*(u)$ is the restriction on $H_{n}$ of the operator defined
by (\ref{df-ell*(f)}) and $\ell_{n+1}(u)$ denotes the adjoint of
the pre--Hilbert space linear map
\begin{equation}\label{df-ln*(v)2}
\ell_n^*(u):=\ell^*(u)\Big|_{V^{\otimes n}} \ : \
(V^{\otimes n},\langle \ \cdot \ ,\ \cdot \ \rangle_{V^{\otimes n}}
\ \to \ (V^{\otimes (n+1)},
\langle \ \cdot \ ,\ \cdot \ \rangle_{V^{\otimes (n+1)}}).
\end{equation}
(ii) Conversely, any pair $(\tilde\Omega^{\otimes}_{n+1},
\langle \ \cdot \ ,\ \cdot \ \rangle_{V^{\otimes n}})$,
where $\langle \ \cdot \ ,\ \cdot \ \rangle_{V^{\otimes n}}$
is a pre--scalar product on  $V^{\otimes n}$ and
$\tilde\Omega^{\otimes}_{n+1}$ is a
$\mathcal{L}_{a}((H_{n}, \langle \ \cdot \ ,\ \cdot \ \rangle_{n}))$--valued
PD kernel on $V$ defines, by the prescription
(\ref{scal-prod-VtensHn-PDK}), a pre--scalar product
$\langle \ \cdot \ ,\ \cdot \ \rangle_{V^{\otimes (n+1)}}$
on $V^{\otimes (n+1)}$
satisfying (\ref{Omega(n+1,u,v)=l(n+1,n,u)ln*(v)}).
%In both cases, the map $T_{n,n+1}$ defined by (\ref{H(n+1)=HtensHn})
%is a pre--Hilbert space linear map.
}\end{lemma}

\begin{proof}
{\bf (i)}. Given
$\langle \ \cdot \ ,\ \cdot \ \rangle_{V^{\otimes (n+1)}}$, for each
$u,v\in V$
the map
\begin{equation}\label{iso-scal-tens-prod}
(h_n,h'_n)\in V^{\otimes n}\times V^{\otimes n} \ \mapsto \
\langle u\otimes h_n, v\otimes h'_n \rangle_{V^{\otimes (n+1)}}
\end{equation}
is sesqui--linear. Since $V^{\otimes n}$ is finite dimensional,
the map (\ref{iso-scal-tens-prod}) defines, for each
$u,v\in V$ a linear map
$$
\tilde\Omega^{\otimes}_{n+1}(u,v) \ : \ H_n \ \to \ H_n
$$
that, by construction, satisfies (\ref{scal-prod-VtensHn-PDK}) and
is adjointable because of finite dimensionality.
Again by finite dimensionality the map (\ref{df-ln*(v)2})
is adjointable and satisfies (\ref{Omega(n+1,u,v)=l(n+1,n,u)ln*(v)}).

{\bf (ii)}. Conversely, given $\tilde\Omega^{\otimes}_{n+1}$, define
$\langle \ \cdot \ ,\ \cdot \ \rangle_{V\otimes H_n}$ by the right hand side of
(\ref{scal-prod-VtensHn-PDK}). By definition of
 $\mathcal{L}_{a}((H_{n}, \langle \ \cdot \ ,\ \cdot \ \rangle_{n}))$--valued
PD kern $\Omega^{\otimes}_{n+1}$ on $V$, this gives a pre--scalar product on
$V\otimes H_n$.
The same argument as in the proof of (i) shows that the map
(\ref{df-ln*(v)2}) is adjointable and
satisfies (\ref{scal-prod-VtensHn-PDK}) and therefore
(\ref{Omega(n+1,u,v)=l(n+1,n,u)ln*(v)}).
This proves (ii).
\end{proof}

\begin{theorem}\label{equiv-IFS-seq-PDkern}{\rm
Let $(H_{n}, \langle \ \cdot \ ,\ \cdot \ \rangle_{n})$ be an IFS
on a finite dimensional vector space $V$ and let
\begin{equation}\label{n-tens-pwr-V}
\left\{ \left(V^{\otimes n} \ ,  \
\langle \ \cdot \ ,  \ \cdot \ \rangle_{\otimes, n}\right)
 \ , \ \ell^* \right\}
\end{equation}
be its tensor representation defined by Lemma \ref{lm-df-IFS2}. Then
the sequence of pre--scalar products
$(\langle \ \cdot \ ,  \ \cdot \ \rangle_{\otimes, n})$ is
uniquely defined by
a sequence  $(\tilde\Omega^{\otimes}_{n})$ with
the follwing properties:
\begin{equation}\label{df-tilde-Omega-otimes0}
\tilde\Omega^{\otimes}_{0}\equiv 1
\end{equation}
is the constant kernel on $V$, identically equal to $1$,
based on the Hilbert space
\begin{equation}\label{df-H0-sc-pr0}
(H_{0}, \langle \ \cdot \ ,\ \cdot \ \rangle_{0})
:=(\mathbb C, \langle z ,w \rangle_{0}:=\bar z w  \ (z ,w\in\mathbb C))
\end{equation}
and $\tilde\Omega^{\otimes}_{n+1}$ is the $\mathcal{L}_{a}((V^{\otimes n},
\langle \ \cdot \ , \ \cdot \ \rangle_{\otimes n}))$--valued PD kernel on $V$
defined by (\ref{Omega(n+1,u,v)=l(n+1,n,u)ln*(v)}).\\

Conversely, let the sequence $(\tilde\Omega^{\otimes}_{n})$ be
inductively defined as follows: $\tilde\Omega^{\otimes}_{0}$ and
$(H_{0}, \langle \ \cdot \ ,\ \cdot \ \rangle_{0})$ are defined
respectively by (\ref{df-tilde-Omega-otimes0}) and (\ref{df-H0-sc-pr0}).\\
Having defined, for $0\leq m\leq n$, the pre--scalar product
$\langle \ \cdot \ ,  \ \cdot \ \rangle_{\otimes, m}$ on $V^{\otimes m}$,
$\tilde\Omega^{\otimes}_{n+1}$ is an arbitrary
$\mathcal L_a(V^{\otimes n},
\langle \ \cdot \ ,  \ \cdot \ \rangle_{\otimes, n})$--valued kernel on $V$. \\
Then, with $\ell^*$ defined by (\ref{df-ell*(f)}),
$((V^{\otimes n},\langle \ \cdot \ ,  \ \cdot \ \rangle_{\otimes, n})_{n\in\mathbb N},\ell^*)$
is an IFS on $V$.
}\end{theorem}

\begin{proof}
Applying the Remark after Definition \ref{scal-prod-H(n+1)a} to the
tensor representation of $(H_{n}, \langle \ \cdot \ ,\ \cdot \ \rangle_{n})$,
one obtains the required sequence $(\tilde\Omega^{\otimes}_{n})$.\\

Conversely, if the sequence $(\tilde\Omega^{\otimes}_{n})$ is
defined as in the second part of the theorem then, according to
Lemma \ref{scal-prod-VtensHn-a}, the pair $(\tilde\Omega^{\otimes}_{n+1},
\langle \ \cdot \ ,  \ \cdot \ \rangle_{\otimes, n})$
defines, by the prescription
(\ref{scal-prod-VtensHn-PDK}), a pre--scalar product
$\langle \ \cdot \ ,\ \cdot \ \rangle_{V^{\otimes (n+1)}}$
on $V^{\otimes (n+1)}$
satisfying (\ref{Omega(n+1,u,v)=l(n+1,n,u)ln*(v)}).
\end{proof}

\subsection{Symmetric interacting Fock spaces}\label{App-S-IFS}

\begin{definition}\label{df-SIFSS}{\rm
An IFS on a vector space $V$ is called {\bf symmetric}, if the
creators commute.
}\end{definition}

The following Lemma shows that, in the tensor representation of a symmetric IFS,
the tensor algebra can be replaced by the symmetric tensor algebra.

\begin{lemma}\label{df-sym-IFS}{\rm
Every symmetric IFS
\begin{equation}\label{df-sym-IFS2}
\left\{ (H_{n}, \langle \ \cdot \ ,\ \cdot \ \rangle_{n})_{n\in\mathbb N}),
a^+ \right\}
\end{equation}
on a vector space $V$ is isomorphic, in the sense of Definition \ref{df-IFS}
to an IFS of the form
\begin{equation}\label{tens-repr-sym-IFS-V}
\left\{ \left(V^{\widehat\otimes  n} \ ,  \
\langle \ \cdot \ ,  \ \cdot \ \rangle_{\widehat\otimes, n}\right)
 \ , \ \hat\ell^* \right\}
\end{equation}
where:\\
-- for all $n\in{\mathbb N}$, $ V^{\widehat\otimes  n}$ denotes the $n$--th
symmetric algebraic tensor power of $V$ and by definition
\begin{equation}\label{df-sym-TP0}
V^{\widehat\otimes  0}:={\mathbb C}\cdot\Phi
\: ,  \qquad \langle\Phi,\Phi\rangle_0=1,
\end{equation}
-- the isomorphism is given by the unique linear extension of the map
\begin{equation}\label{df-hat-T}
\hat T(u_n\widehat\otimes \dots \widehat\otimes u_1)
:=a^+(u_n)\cdots a^+(u_1)\Phi   \: , \qquad n\in{\mathbb N} \ , \ u_j\in V,
\end{equation}
-- the pre--scalar products
$\langle \ \cdot \ ,  \ \cdot \ \rangle_{\widehat\otimes, n}$ are given,
for any $n\in{\mathbb N}$ and $u_1,v_1,\dots ,u_n,v_n\in V$, by
\begin{equation}\label{tens-repr-sym-IFS-V-sc-pr}
%\in V^{\widehat\otimes n} \to
\langle u_n\widehat\otimes \dots \widehat\otimes u_1,
v_n\widehat\otimes \dots \widehat\otimes v_1\rangle_{\otimes, n}
:=\langle a^+(u_n)\cdots a^+(u_1)\Phi,
a^+(v_n)\cdots a^+(v_1)\Phi\rangle_{n},
\end{equation}
-- the operator $\hat\ell^*$ defined by
\begin{equation}\label{T-1a*(v)T=hatl*otimes(v)2}
\hat\ell^*(v)(u_n\widehat\otimes \dots \widehat\otimes u_1)
:= v\widehat\otimes u_n\widehat\otimes \dots \widehat\otimes u_1
\: , \qquad \forall v,u_n , \dots , u_1\in V,
\end{equation}
 up to addition of zero--norm vectors satisfies
\begin{equation}\label{T-1a*(v)T=hatl*otimes(v)}
\hat T^{-1}a^+(v)\hat T=\hat\ell^*(v)
\: , \qquad \forall v\in V.
\end{equation}
}\end{lemma}

\begin{proof}
The mutual commutativity of the creators implies that, in the
notations of Lemma \ref{lm-df-IFS2}, the maps $T_n$ ($n\in\mathbb N$) satisfy
\begin{equation}\label{tens-iso-gen3}
T_n(v_n\otimes \dots \otimes v_1)
=T_n(v_n\widehat\otimes \dots \widehat\otimes v_1)
=:\hat T_n(v_n\widehat\otimes \dots \widehat\otimes v_1)
\: , \quad n\in\mathbb N .
\end{equation}
This shows that the graded vector space homomorphism $T$, defined by
(\ref{tens-iso-gen2}) can be restricted to the symmetric tensor
algebra $\hat Tens(V)$ thus defining the graded vector space homomorphism
\begin{equation}\label{sym-tens-iso}
\hat T:=\bigoplus_{n} \hat T_n  \ : \ \hat Tens(V)
=\bigoplus_{n\in\mathbb N} V^{\widehat\otimes n} \to \bigoplus_{n\in\mathbb N} H_{n}
\end{equation}
where $\hat T_n$ is given by (\ref{df-hat-T}).
In this restriction the pre--scalar products defined by
(\ref{tens-repr-IFS-V-sc-pr}) become (\ref{tens-repr-sym-IFS-V-sc-pr}) and
the condition that the spaces $V^{\widehat\otimes n}$ are mutually orthogonal
uniquely defines the pre--scalar product on $\hat Tens(V)$. With this
scalar product $T$ becomes a unitary gradation preserving isomorphism.
Therefore, with $\ell^*_{v_n}$ given by (\ref{T-1a*(v)T=hatl*otimes(v)2}),
the identity (\ref{tens-repr-sym-IFS-V-sc-pr}) can be rewritten in the form
$$
\langle u_n\widehat\otimes \dots \widehat\otimes u_1,
\widehat\ell^*_{v_n}v_{n-1}\widehat\otimes \dots
\widehat\otimes v_1\rangle_{\widehat\otimes, n}
=\langle a^+(u_n)\cdots a^+(u_1)\Phi,
a^+(v_n)a^+(v_{n-1})\cdots a^+(v_1)\Phi\rangle_{n}
$$
$$
=\langle T(u_n\widehat\otimes \dots \widehat\otimes u_1),
T(T^{-1}a^+(v_n)T)v_{n-1}\widehat\otimes \dots
\widehat\otimes v_1\rangle_{n}
$$
$$
=\langle u_n\widehat\otimes \dots \widehat\otimes u_1,
(T^{-1}a^+(v_n)T)v_{n-1}\widehat\otimes \dots
\widehat\otimes v_1\rangle_{\widehat\otimes, n}
$$
Therefore
$$
\widehat\ell^*_{v_n}v_{n-1}\widehat\otimes \dots
\widehat\otimes v_1
-(Ta^+(v_n)T^{-1})Tv_{n-1}\widehat\otimes \dots\widehat\otimes v_1
$$
is a zero--norm vector and this proves (\ref{T-1a*(v)T=hatl*otimes(v)}).\\
The unitarity of $T$ and (\ref{T-1a*(v)T=hatl*otimes(v)})
imply the adjointability of the maps $\ell^*(v)$ ($v\in V$) because
the maps $a^+(v)$ admit pre--Hilbert space adjoints by definition.
Therefore, with this definition $\hat T$ becomes an isomorphism of IFS.
\end{proof}

\begin{theorem}\label{equiv-SIFS-seq-PDkern}{\rm
Every symmetric IFS $(H_{n}, \langle \ \cdot \ ,\ \cdot \ \rangle_{n})$
on a finite dimensional vector space $V$ uniquely defines
a sequence  $(\tilde\Omega^{\widehat\otimes}_{n})$ with
the following properties:
\begin{equation}\label{df-tilde-Omega-otimes0-S}
\tilde\Omega^{\widehat\otimes}_{0}\equiv 1
\end{equation}
is the constant kernel on $V$, identically equal to $1$, on the Hilbert space
\begin{equation}\label{df-H0-sc-pr0-S}
(H_{0}, \langle \ \cdot \ ,\ \cdot \ \rangle_{0})
:=(\mathbb C, \langle z ,w \rangle_{0}:=\bar z w  \ (z ,w\in\mathbb C))
\end{equation}
and $\tilde\Omega^{\widehat\otimes}_{n+1}$ is the
%an arbitrary
$\mathcal{L}_{a}((V^{\widehat\otimes n},
\langle \ \cdot \ , \ \cdot \ \rangle_{\widehat\otimes n}))$--valued PD kernel on $V$
defined by (\ref{Omega(n+1,u,v)=l(n+1,n,u)ln*(v)})
\begin{equation}\label{scal-prod-VtensHn-PDK-S}
\langle u\widehat\otimes h_n, v\widehat\otimes h'_n
\rangle_{V^{\widehat\otimes (n+1)}}
=\langle h_n, \tilde\Omega^{\widehat\otimes}_{n+1}(u,v)h'_n
\rangle_{V^{\widehat\otimes n}}
%\quad ;\quad
\end{equation}
($u,v\in V \ , \ h_n,h'_n\in V^{\widehat\otimes n}$), where
$\langle \ \cdot \ , \ \cdot \ \rangle_{\widehat\otimes n}$ is the pre--scalar
product induced on $V^{\widehat\otimes n}$ by the symmetric tensor
representation of $(H_{n}, \langle \ \cdot \ ,\ \cdot \ \rangle_{n})$.\\

Conversely, let the sequence $(\tilde\Omega^{\widehat\otimes}_{n})$ be
inductively defined as follows: $\tilde\Omega^{\widehat\otimes}_{0}$ and
$(H_{0}, \langle \ \cdot \ ,\ \cdot \ \rangle_{0})$ are defined
respectively by (\ref{df-tilde-Omega-otimes0-S}) and (\ref{df-H0-sc-pr0-S}).
Having defined, for $0\leq m\leq n$, the pre--scalar product
$\langle \ \cdot \ ,  \ \cdot \ \rangle_{\widehat\otimes, m}$ on $V^{\widehat\otimes n}$,
$\tilde\Omega^{\widehat\otimes}_{n+1}$ is an arbitrary
$\mathcal L_a(V^{\widehat\otimes n},\langle \ \cdot \ ,  \
\cdot \ \rangle_{\widehat\otimes m})$--valued kernel on $V$. Then
$((V^{\widehat\otimes n},\langle \ \cdot \ ,  \ \cdot \
\rangle_{\widehat\otimes, n}),\ell^*)$, where $\hat\ell^*$ is defined by
\begin{equation}\label{df-ell*(f)-S}
\ell^*(f)f_n\widehat\otimes\ldots\widehat\otimes f_1
:=f\widehat\otimes f_n\widehat\otimes\ldots\widehat\otimes f_1
\end{equation}
is a symmetric IFS on $V$.
}\end{theorem}

\begin{proof}
The proof is based on the remark that Lemma \ref{scal-prod-VtensHn-a}
and Theorem (\ref{equiv-IFS-seq-PDkern}) continue to hold for symmetric
tensor products and their proofs are just verbal adjustments of those in
the non symmetric case.
\end{proof}


\vspace*{8pt}


\begin{thebibliography}{99}


\bibitem{[AcBo98]}
L. Accardi and M. Bo\.{z}ejko,
{\it Interacting Fock space and Gaussianization of
probability measures},
IDA--QP (Infin. Dim. Anal. Quantum Probab. Rel. Topics)
1 (1998), 663-670.
%
\bibitem{[AcKuoSt04b]}
L. Accardi, H.-H. Kuo and A.I. Stan,
{\it Characterization of probability measures
through the canonically associated interacting Fock spaces},
IDA--QP (Inf. Dim. Anal. Quant. Prob. Rel. Top.) 7 (4) (2004), 485-505
%
\bibitem{[AcKuoSta05]}
L. Accardi, H.-H. Kuo and A. Stan,
{\it Probability measures in terms of creation, annihilation,
and neutral operators},
Quantum Probability and Infinite Dimensional Analysis:
From Foundations to Applications. [QP--PQ XVIII],
M. Sch\"urmann and U. Franz (eds.),
World Scientific (2005), 1-11.
%
\bibitem{[AcKuoSta07]}
L. Accardi, H.-H. Kuo and A. Stan,
{\it Moments and commutators of probability measures},
IDA--QP (Infin. Dim. Anal. Quantum Probab. Rel. Topics)
10 (4) (2007), 591-612.
%
\bibitem{[AcNh02]}
L. Accardi and M. Nhani,
{\it Interacting Fock Spaces and Orthogonal Polynomials in
several variables},
The Crossroad of Non-Commutativity, Infinite-Dimensionality,
Obata, Hora, Matsui (eds.) World Scientific (2002), 192--205.
Preprint Volterra, N. 523 (2002).

\bibitem{alv97} Accardi, L., Lu, Y. G., and
Volovich, I.:
The QED Hilbert module and interacting Fock spaces;
{\it IIAS Reports} No. 1997-008 (1997) International
Institute for Advanced Studies, Kyoto

\bibitem{[AcSk98]}
Accardi, L. and Skeide, M.:
Interacting Fock space versus full Fock module;
Commun. Stoch. Anal. 2 (3) (2008) 423-–444,
Volterra Preprint N. 328 (1998)

\bibitem{[AlpJorgKim14]}
Daniel Alpay, Palle E. T. Jorgensen, David P. Kimsey,
{\it Moment problems in an infinite number of variables}
IDA--QP (Infin. Dim. Anal. Quantum Probab. Rel. Topics), 21 (4) (2015)

\bibitem{Chiha78}
T.S. Chihara,
{\it An Introduction to Orthogonal Polynomials},
Gordon \& Breach, New York (1978).
%
\bibitem{[DuXu01]}
C.F. Dunkl and Y. Xu,
{\it Orthogonal polynomials of several variables},
Cambridge University Press (2001).
%
\bibitem{[ErMagOberTric53]}
A. Erdely, W. Magnus, F. Oberhettinger and F. F. Tricomi,
{\it Higher transcendental functions},
Vol 2 McGraw--Hill (1953).
%
\bibitem{[IsmAsk84]}
M. Ismail and R. Askey  (Eds.),
{\it Recurrence relations, continued fractions and orthogonal
polynomials},
American Mathemaical Society  49  (300) (1984), 1-114.
%
\bibitem{[Koor90]}
T.H. Koornwinder,
{\it Orthogonal polynomials in connection with quantum groups},
Nevai, P.  (ed.) Orthogonal Polynomials, Kluwer Acad.
Publ. (1990), 257-292.
%
\bibitem{[Kowa82a]}
M.A. Kowalski,
{\it The recursion formulas for polynomials in $n$ variables},
SIAM J. Math. Anal. 13 (1982), 309-315.
%
\bibitem{[Kowa82b]} M.A. Kowalski,
{\it Orthogonality and recursion formulas for polynomials
in $n$ variables},
SIAM J. Math. Anal. 13 (1982), 316-323.
%
\bibitem{[KrShe67]}
H.L. Krall and I.M. Sheffer,
{\it Orthogonal polynomials in two variables},
Ann. Mat. Pura Appl. 76 (4) (1967), 325-376.

\bibitem{[Macdon98]}
I.G. Macdonald,
{\it Symmetric Functions and Orthogonal Polynomials},
AMS 1998.
%
\bibitem{[MarcVanAs06]}
F. Marcell\`{a}n and  Van Assche (Editors),
{\it Orthogonal Polynomials and Special Functions},
Springer (2006) W. 3-540-31062-2.
%

\bibitem{Nussb66} A. E. Nussbaum,
{\it Quasi--analytic vectors},
Ark. math. 6 (1966).

\bibitem{st1943} J. A. Shohat and J. D. Tamarkin,
{\it The Problem of Moments},
Mathematical Surveys, Number I (Amer. Math. Soc., 1943).
%
\bibitem{sz1975} M. Szeg\"{o},
{\it Orthogonal Polynomials},
Amer. Math. Soc. Colloq. Publ. vol.23, Providence,
R.I. (1975) 4th ed.
%
\bibitem{[Xu93]} Y. Xu,
{\it Unbounded commuting operators and multivariable
orthogonal polynomials},
Proc. Amer. Math. Soc. 119 (4) (1993), 1223-1231.
%
\bibitem{[Xu97a]} Y. Xu,
{\it On orthogonal polynomials in several variables},
Fields Institute Communications 14 (1997), 247-270.

\bibitem{[Xu04]} Yuan Xu,
{\it Lecture notes on orthogonal polynomials of several variables},
Inzell Lectures on Orthogonal Polynomials,
W. zu Castell, F. Filbir, B. Forster (eds.)
Advances in the Theory of Special Functions and Orthogonal Polynomials,
Nova Science Publishers
Volume 2, 2004, Pages 135--188

\end{thebibliography}
\end{document}